%% file: primed-braids-draft2.tex
\documentclass[10pt]{article}

\input{preamble.tex}

\begin{document}
\title{Extending a word property for twisted Coxeter systems}
\author{Eric Marberg}
\date{}

\maketitle

\begin{abstract}
 We  prove two extensions of  Hansson and Hultman's word property for
 certain analogues of reduced words associated to twisted involutions in Coxeter groups.
 Our first extension concerns the superset of such words in which terms with a natural commutativity property may be 
 optionally primed. Our other extension involves variants of these words in which 
a defining minimal length condition is relaxed. In type $\mathsf{A}$ the sets considered are
closely related to generating functions for Schur $Q$-functions and $K$-theoretic Schur $P$-functions.
\end{abstract}

\setcounter{tocdepth}{2}

\section{Introduction}

Let $(W,S)$ be a Coxeter system.
A \defn{reduced word} for an element $w \in W$
is a minimal length sequence $(s_1,s_2,\dots,s_n)$ with $s_i \in S$ and $w=s_1s_2\cdots s_n$.
We write $\cR(w)$ for the set of all reduced words for $w$.

For any $s,t \in S$ let $m(s,t)$ denote the order of the product $st \in W$.
For each $s,t \in S$ with $2 \leq m(s,t) < \infty$ there is an associated  \defn{braid relation} 
on finite sequences of simple generators, which we write as
\be\label{braid-eq} (\dash ,\underbrace{s, t, s, t, \dots}_{m(s,t)\text{ factors}},\dash) \sim( \dash,\underbrace{ t ,s, t,s, \dots}_{m(s,t)\text{ factors}}, \dash). \ee
Here and in similar expressions, the corresponding symbols ``$\dash$'' 
on either side of $\sim$ are required to mask identical subsequences.
It is a well-known result of Matsumoto \cite{Matsumoto} and Tits \cite{Tits} that 
the braid relations \eqref{braid-eq} span and preserve each set $\cR(w)$ for $w \in W$; see \cite[Thm 3.3.1]{CCG} for a proof.
This is sometimes called the \defn{word property} for Coxeter groups.

This article is concerned with similar word properties for variants of the following construction.
Let $w \mapsto w^*$ be a self-inverse group automorphism 
of $W$ that preserves $S$, that is, an involution of the associated Coxeter graph.
We refer to  $(W,S,*)$ as a \defn{twisted Coxeter system}.
Suppose $a = (s_1, s_2\cdots, s_n)$ is a reduced word for an element of $W$.  
There is a unique subword $(s_{i_1}, s_{i_2}, \dots, s_{i_m})$ of maximal length
such that 
$\hat a:=(s_{i_m}^*, \dots, s_{i_2}^*, s_{i_1}^*, s_1,s_2,\dots, s_n)$ is also a reduced word.

One can show that $\hat a$ is always a reduced word for an element of the set of
 \defn{twisted involutions}
$
 \I_*(W) := \{ w\in W : w^{-1} = w^*\}.
 $
The sequence $a$ is defined to be an \defn{involution word} for $z \in \I_*(W)$
if $a$ is of minimal length such that $\hat a \in \cR(z)$.
Let $\iR(z)$ be the set of involution words for $z$.
Figure~\ref{twisted_invol4321-fig}
shows an example of this set when $W$ is a finite symmetric group and $* \neq \id$.
We will review some more constructive definitions of $\iR(z)$
in Section~\ref{prelim-sect}.\footnote{An interesting but even less constructive definition, which holds whenever $W$ is finite and is conjectured in general \cite[Conj. 4.2]{HMP2},
is that $\iR(z)$ consists of the reduced words for all minimal length elements $w \in W$ satisfying $w^*z \leq w$ in strong Bruhat order.}

Involution words have been studied previously 
in a few different forms.
In special cases they correspond to maximal chains in the weak order posets discussed in \cite{Brion2001,CJW}.
They are the same (though sometimes written in the opposite order)
as the \defn{reduced $\underline S$-expressions} in \cite{HanssonHultman,Hultman2,Hultman3},  \defn{reduced $I_*$-expressions} in \cite{HuZhang1,HuZhang2,HuZhang3}, and  \defn{admissible sequences} in \cite{RichSpring}.

The braid relations \eqref{braid-eq} preserve but usually do not span  the set  $\iR(z)$.
For example, suppose $s,t \in S$ are fixed by $*$ with $2<m(s,t) < \infty$.
Then the $1 + \lfloor \tfrac{1}{2}m(s,t)\rfloor$ element sequences $a=(s,t,s,\dots)$ and $b=(t,s,t,\dots)$ 
are both involution words for the longest element of the finite dihedral subgroup $\langle s,t\rangle$,
despite not being connected by any braid relations. Moreover, if an involution word 
begins with $a$ then replacing 
this initial subword with $b$ produces another involution word for the same element of $\I_*(W)$.

Hu and Zhang show in \cite{HuZhang1} that these \defn{half-braid relations} plus the usual braid relations are sufficient to span $\iR(z)$
in type $\mathsf{A}$ when $*$ is the identity map.
Hansson and Hultman \cite{HanssonHultman} extend this result to arbitrary twisted Coxeter systems as follows. 

For each pair of involution words 
for the longest element of a finite $*$-invariant parabolic subgroup of $W$,
there is a corresponding \defn{initial relation}
preserving $\iR(z)$.
Adding all such relations to the usual braid relations generates a relation spanning every set $\iR(z)$.
However, this includes many extraneous relations. In fact, Hansson and Hultman show in \cite{HanssonHultman} that 
it is only necessary to add initial relations derived from finite $*$-invariant parabolic subgroups of types $\mathsf{A}_3$, $\mathsf{BC}_3$, $\mathsf{D}_4$,
$\mathsf{H}_3$, and $\mathsf{I}_2(n)$. These are precisely the finite Coxeter systems  for which the complement of the Coxeter graph is disconnected.
For the precise statement of Hansson and Hultman's word property, see Section~\ref{prelim-sect}.

In this article we are interested in two generalizations of  $\iR(z)$.
First, we study the set of \defn{primed involution words} $\iR^+(z)$, which may be described as follows.
Above, we associated to each reduced word $a = (s_1, s_2,\dots, s_n)$
a ``doubled'' reduced word of the form
$\hat a =(s_{i_m}^*, \dots, s_{i_2}^*, s_{i_1}^*, s_1,s_2,\dots, s_n)$.
We refer to the indices in $\{1,2,\dots,n\} \setminus \{ i_1,i_2,\dots, i_m\}$ as the \defn{commutations} in $a$.
Each element of $\iR^+(z)$ consists of an involution word for $z$ paired with an arbitrary set of its commutations;
we think of this object as a sequence formed by adding primes to certain letters in an involution word.

Next, we examine the set of \defn{(reduced) involution Hecke words} $\iH(z)$.
This is the set of reduced words for all elements $w \in W$ satisfying $(w^{-1})^* \circ w = z$,
where $\circ : W \times W \to W$ is the \defn{Demazure product} defined in Section~\ref{prelim-sect}.
For examples of 
$\iR^+(z)$ and $\iH(z)$ see Figures~\ref{primed_invol4321-fig} and \ref{hecke4321-fig}.

Our main results, Theorems~\ref{hh-thm2} and \ref{hecke-thm}, give word properties for $\iR^+(z)$ and $\iH(z)$
when $(W,S,*)$ is an arbitrary twisted Coxeter system. The form of both theorems is very similar to the main result of 
Hansson and Hultman \cite{HanssonHultman}. In addition to a set of relevant substitutes for the braid relations \eqref{braid-eq},
to get a spanning relation
one must add certain exceptional relations corresponding to each finite $*$-invariant parabolic subgroups of type $\mathsf{A}_3$, $\mathsf{BC}_3$, $\mathsf{D}_4$,
$\mathsf{H}_3$, or $\mathsf{I}_2(n)$.
However, some work is required to extend the proofs in \cite{HanssonHultman} to our cases of interest.

To explain the motivation for these results, we specialize to type $\mathsf{A}$ with $*=\id$.
Then reduced words and involution words may be identified with positive integer sequences, 
while primed involution words 
become sequences of elements from the set $\{1'<1<2'<2<\dots\}$.

From an enumerative perspective, passing from involution words to primed involution words
is a fairly trivial extension, which just accounts for an extra power of two factor appearing in some 
generalizations of \defn{Schubert polynomials} studied in \cite{HMP4,WyserYong}.
The dynamics of the relations connecting all words in $\iR^+(z)$ for $z =z^{-1} \in S_n$, however, 
turn out to be much more complicated than for the relations connecting $\iR(z)$.

The article \cite{Marberg2022} and its sequel \cite{MT2021} construct certain \defn{crystals} on the sets of increasing factorizations of words in 
$\iR(z)$ and $\iR^+(z)$, respectively. The crystal operators for these structures are composed of the relations 
in Theorems~\ref{hh-thm2}, and some proofs in \cite{MT2021} rely on the results in this article. 

The crystals based on $\iR(z)$ have characters that are sums of \defn{Schur $P$-polynomials} $P_\lambda$
while the crystals based on $\iR^+(z)$ have characters that are sums of \defn{Schur $Q$-polynomials} $Q_\lambda$.
Although there is a simple identity $Q_\lambda = 2^{\ell(\lambda)} P_\lambda$ relating these functions,
there is no easy way of deducing the main theorems about the second family of crystals from the first (such as the fact that 
up to isomorphism
they are 
closed under tensor products).

%

The rest of this article is organized as follows. Section~\ref{prelim-sect} reviews some
preliminaries while Section~\ref{some-sect} contains a couple of general results about involution words.
Our extensions of Hansson and Hultman's word property appear in Sections~\ref{rel-sect1} and \ref{rel-sect2}.
Sections~\ref{app-sect1} and \ref{app-sect2} discuss a few applications of these results.

\subsection*{Acknowledgments}

This work was partially supported by Hong Kong RGC grants ECS 26305218 and GRF 16306120.

\section{Preliminaries}\label{prelim-sect}

For the duration of the article $(W,S,*)$ denotes a twisted Coxeter system with length function $\ell : W \to \NN$.
%
%

There is a unique associative operation $\circ : W \times W \to W$, often called the \defn{Demazure product}, satisfying
$v \circ w = vw$ for all $v, w \in W$ with $\ell(vw) = \ell(v) + \ell(w)$
and $s\circ s =s $ for all $s \in S$. 
One way to derive this is to set $a_s=1$ and $b_s=0$ in \cite[Thm. 7.1]{Humphreys}
and then notice that $\{T_w : w \in W\}$ is a monoid under multiplication; alternatively, see 
the discussion in \cite[\S3.10]{RichSpring}.


In terms of the $\circ$ operation, the set of reduced words $\cR(w)$ for $w \in W$ consists of the
minimal length sequences $(s_1,s_2,\dots,s_n)$ with $s_i \in S$ and $w=s_1\circ s_2 \circ \cdots\circ s_n$.
An analogous way of defining 
an \defn{involution word} for $z \in W$ is as a minimal length sequence $(s_1,s_2,\dots,s_n)$ with $s_i \in S$ and 
\be\label{circcirc}
z = s_n^* \circ \cdots \circ s_2^* \circ s_1^*  \circ s_1 \circ s_2 \circ \cdots \circ s_n.
\ee
This definition is equivalent to the one in the introduction.
As $\circ$ is associative with $(u\circ v)^* = u^* \circ v^*$ and $(u\circ v)^{-1} = v^{-1} \circ u^{-1}$,
an involution word for $z$ is just a reduced word for a minimal length element $w \in W$ with 
$z = (w^{-1})^* \circ w.$
If $z$ is in the set of twisted involutions
$\I_*(W) := \{ w\in W : w^{-1} = w^*\}$
then
\be\label{oo-eq}
 s^* \circ z \circ s = \begin{cases}  
z &\text{if $\ell(z) > \ell(zs)$} \\
zs &\text{if $\ell(z) < \ell(zs)$ and $zs=s^*z$} \\
s^*zs&\text{if $\ell(z) < \ell(zs)$ and $zs\neq s^*z$}
\end{cases}
\ee
for all $s \in S$
by  \cite[Lem. 3.4]{Hultman2}. It follows 
that
$\I_*(W) = \left\{ (w^{-1})^* \circ w : w \in W\right\}$
so $z \in W$ has an involution word if and only if $z \in \I_*(W)$.
We continue to write $\iR(z)$ for the set of all involution words for $z \in \I_*(W)$; see
Figure~\ref{twisted_invol4321-fig} for an example.

Sometimes another equivalent definition of $\iR(z)$ is used.
Define a set of underlined symbols $\underline S := \{ \underline s : s \in S\}$.
There is a unique right action of the free monoid on $\underline S$ on the set $\I_*(W)$ satisfying
\be\label{uaction-eq} z \underline s = \begin{cases} zs &\text{if }zs=s^*z \\ s^* zs&\text{otherwise}\end{cases}\ee
for $z \in \I_*(W)$ and $ s \in S$ \cite[Def. 2.1]{HanssonHultman}.
For this action, one has $z \underline s\hs \underline s = z$.
One can show that
the involution words for $z \in \I_*(W)$
are the minimal length sequences $(s_1,s_2,\dots,s_n)$ with $s_i \in S$
and $z = 1 \underline s_1 \underline s_2 \cdots \underline s_n$ \cite[Cor. 2.6]{HMP2}.
This means involution words are the same as \defn{reduced $\underline S$-expressions} in \cite{HanssonHultman,Hultman2,Hultman3}. 

We mention two other properties of these words,
which we will often use implicitly. 
Fix $s \in S$ and $z \in \I_*(W)$. Then $z$ has an involution word ending in $s$ if and only if $\ell(zs) < \ell(z)$
\cite[Lem. 3.8]{Hultman2}.
It also follows from \cite[Lem. 3.8]{Hultman2} that
a sequence satisfying \eqref{circcirc} belongs to $\iR(z)$ if and only if the elements
\[ s_1^*  \circ s_1,\quad s_2^* \circ s_1^*  \circ s_1 \circ s_2, \quad s_3^* \circ s_2^* \circ s_1^*  \circ s_1 \circ s_2\circ s_3, \quad\dots\] are all distinct, in which case  $
s_i^* \circ \cdots  \circ s_1^*  \circ s_1\circ \cdots s_i =
1\underline s_1 \underline s_2  \cdots\underline s_i$ for all $i $.

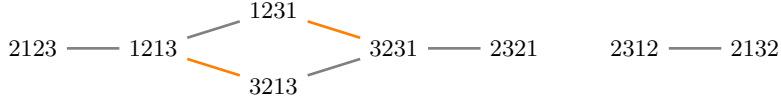
\begin{figure}[h]
\begin{center}
\begin{tikzpicture}[>=latex,line join=bevel,xscale=1.6]
  \pgfsetlinewidth{1bp}
  \small
    \pgfsetcolor{gray}
\node (2123) at (0,0) {$2123$};
\node (1213) at (1,0) {$1213$};
\node (1231) at (2,0.5) {$1231$};
\node (3213) at (2,-0.5) {$3213$};
\node (3231) at (3,0) {$3231$};
\node (2321) at (4,0) {$2321$};
\node (2312) at (5,0) {$2312$};
\node (2132) at (6,0) {$2132$};
\draw (2123) -- (1213) -- (1231);
\draw (3231) -- (2321);
\draw (2312) -- (2132);
\draw [color=orange] (1231) -- (3231);
\draw [color=orange] (1213) -- (3213);
\draw (3213) -- (3231);
\end{tikzpicture}
\end{center}
\caption{Involution words for the longest element $z=(1,4)(2,3)$ in the symmetric group $S_4$ relative to the unique Coxeter automorphism $*\neq \id$. 
Each expression $abcd$ stands for $(s_a,s_b,s_c,s_d)$ where $s_i := (i,i+1) \in S_4$ and $s_i^* = s_{4-i}$.
The grey edges show all braid relations between these words while the colored edges show all
\defn{half-braid relations} in the sense of Examples~\ref{half-braid} and \ref{half-braid2}.}\label{twisted_invol4321-fig}
\end{figure}

\section{Hansson and Hultman's relations}\label{hh-rel-sect}

%
%

The definition above shows that $\iR(z)$ is a union of sets of the form $\cR(w)$ for certain elements $w \in W$.
Thus the braid relations \eqref{braid-eq} always 
preserve but typically do not span the set $\iR(z)$.
Hansson and Hultman show in \cite{HanssonHultman} that one can connect $\iR(z)$ by adding certain relations of the following kind.

\begin{definition}\label{init-rel-def}
Choose  $J \subset S$ with $J=J^*$ such that $W_J := \langle J \rangle$ is finite. Suppose $(s_1,s_2,\dots,s_n)$ and $(t_1,t_2,\dots,t_n)$ are involution words for the longest element  $w_0^J\in W_J$. We refer to any relation of the form
\be\label{init-rel-eq}(s_1,s_2,\dots,s_n,\dash) \sim  (t_1,t_2,\dots,t_n,\dash) \ee
as an \defn{initial relation}, whose
 \defn{type} is the isomorphism class of $(W_J,J,*)$.
\end{definition}

Unlike with the ordinary braid relations, words connected by initial relations of the form \eqref{init-rel-eq}
can only differ in their first $n$ letters.

With one exception, we will only need to name the type of $(W_J,J,*)$ when $W_J$ is finite and irreducible.
In this case 
we denote the isomorphism class of $(W_J,J)$  
either by $\mathsf{X}_n$
where $\mathsf{X} \in \{\text{$\mathsf{A}$, $\mathsf{BC}$, $\mathsf{D}$, $\mathsf{E}$, $\mathsf{F}$, $\mathsf{G}$, $\mathsf{H}$}\}$
and $|J|=n$, or by $\mathsf{I}_2(n)$ in the case when $|J| =2$ and $|W_J| = 2n$. 
We use the same symbol to indicate the type of $(W_J,J,*)$ when $*=\id$.

\begin{example}\label{half-braid}
The twisted subsystem $(W_J,J,*)$ has type $ \mathsf{I}_2(n)$ if $J = \{s,t\}$, $s^*=s\neq t^*=t$, and $m(s,t)=n$.
When $n<\infty$ there is one initial relation
\be\label{i2-eq1} (\underbrace{s,t,s,t,s,t,s,\dots}_{1+\lfloor n/2\rfloor \text{ factors}},\dash) \sim (\underbrace{t,s,t,s,t,s,t,\dots}_{1+\lfloor n/2\rfloor \text{ factors}},\dash).
\ee
This relation can be ignored when  $n=2$, which is the unique case when $(W_J,J)$ is reducible, since then it coincides with an ordinary braid relation.
\end{example}

We only require names for two types of systems $(W_J,J,*)$ with $*\neq \id$:

\begin{example}\label{half-braid2}
The twisted subsystem $(W_J,J,*)$ has type $^2 \mathsf{I}_2(n)$ if $J = \{s,t\}$, $s^*=t \neq t^*=s$, and $m(s,t) =n$.
When $n<\infty$ there is one initial relation
\be\label{i2-eq2} (\underbrace{s,t,s,t,s,t,\dots}_{\lceil n/2\rceil \text{ factors}},\dash) \sim (\underbrace{t,s,t,s,t,s,\dots}_{\lceil n/2\rceil \text{ factors}},\dash).
\ee
This relation is meaningful even when $n=2$, as then it lets us replace the single letter $s$ by $t$ at the beginning of a word.
\end{example}

Following \cite{HanssonHultman}, we refer to \eqref{i2-eq1} and \eqref{i2-eq2} as \defn{half-braid relations}.

\begin{example} The twisted subsystem $(W_J,J,*)$ has type $^2\mathsf{A}_{n}$ if
we can write $J = \{s_1,s_2,\dots,s_{n}\}$ where $s_i^* = s_{n+1-i}$ for all $i $
 and where $m(s_i,s_j)$ is $3$ if $|i-j| = 1$ or $2$ if $|i-j|>1$.
There are multiple initial relations of this type; in rank $n=3$ one such relation is $ (s_2,s_3,s_1,s_2,\dash) \sim (s_2,s_3,s_2,s_1,\dash).$
\end{example}

The following theorem   extends earlier case-by-case results in \cite{HuZhang1,HuZhang2,HuZhang3,Marberg2017}.

\begin{theorem}[{\cite[Thm. 1.2]{HanssonHultman}}]
\label{hh-thm}
Let $z \in \I_*(W)$. Then $\iR(z)$ is an equivalence class under
the transitive closure of the braid relations for $(W,S)$ plus all initial relations
of type $^2\mathsf{A}_3$, $\mathsf{BC}_3$, $\mathsf{D}_4$, $\mathsf{H}_3$, $\mathsf{I}_2(n)$, or $^2\mathsf{I}_2(n)$ for $2\leq n < \infty$.
\end{theorem}


Hansson and Hultman also prove a more explicit form of this result with a minimal set of spanning relations \cite[Thm. 4.1]{HanssonHultman},
similar to our Theorem~\ref{hh-thm2}.


\section{Some general properties}\label{some-sect}

This section contains two general propositions that slightly refine 
the main technical lemma in \cite[\S3.1]{HanssonHultman}.

Fix $s,t \in S$ with $m(s,t)<\infty$ and let $\Delta = w_0^{\{s,t\}}   $ 
be the longest element of the finite dihedral subgroup $W_{\{s,t\}} = \langle s,t\rangle$.
Choose a map $\theta : \{s,t\} \to W$.
If $\theta(\{s,t\}) = \{s,t\} $ then $\theta$ extends to a Coxeter
involution of $W_{\{s,t\}}$
and we 
 define 
$m(s,t;\theta)$ to be the common length of all involution words in $\cR_{\mathsf{inv},\theta}(\Delta)$.
If $\theta(\{s,t\}) \neq \{s,t\}$ then we set $m(s,t;\theta) := \ell(\Delta)$ 
to be the common length of all reduced words in $\cR(\Delta)$.
More explicitly one has \cite[Prop. 7.7]{HMP2} 
\be\label{m-eq} m(s,t;\theta) := \begin{cases} 
\tfrac{1}{2} m(s,t)+\tfrac{1}{2}  & \text{if $m(s,t)$ is odd and $\theta(\{s,t\}) = \{s,t\}$}  \\
 \tfrac{1}{2} m(s,t) +1 & \text{if $m(s,t)$ is even, $\theta(s)=s$, and $\theta(t)=t$} \\
 \tfrac{1}{2} m(s,t) & \text{if $m(s,t)$ is even, $\theta(s) =t$, and $\theta(t)=s$} \\
 m(s,t) &\text{otherwise}.
\end{cases}
\ee
For convenience we also set $m(s,t;\theta) :=\infty$ if $s, t \in S$ and $m(s,t) =\infty$.

For $z \in W$ let $\Adstar_z : W \to W$ be the group automorphism $w \mapsto (zwz^{-1})^*$.
The formula \eqref{m-eq} has the following consequence.

\begin{corollary}
Let $s,t \in S$ and $z \in \I_*(W)$. Then $m(s,t;\Adstar_z ) \leq m(s,t)$,
with equality if and only if 
either $m(s,t) \in \{1,\infty\}$,
 $m(s,t) = 2$ and $zs \neq t^* z$, or
 $m(s,t)\in \{3,4,5,\dots\}$ and $\{ zs,zt\} \neq \{s^*z,t^*w\}$.
\end{corollary}


Denote the \defn{right descent set} of $w$ by $\DesR(w) := \{ s \in S : \ell(ws) < \ell(w)\}$.

\begin{proposition}\label{tprop1}
Suppose $s,t \in S$, $y \in \I_*(W)$, and $(r_1,r_2,\dots,r_k) \in \iR(y)$.
Let $n>0$ be an integer.
Then the words
\be\label{rwords-eq}
(r_1,r_2,\dots,r_k, \underbrace{\dots,t,s,t,s}_{n\text{ terms}}) \quand (r_1,r_2,\dots,r_k,\underbrace{\dots,s,t,s,t}_{n\text{ terms}})
\ee
both belong to $\iR(z)$ for some $z \in \I_*(W)$ 
if and only if
$\{s,t\}\cap \DesR(y)=\varnothing$ and $n=m(s,t;\Adstar_y)$, and in this case it holds that
 $m(s,t;\Adstar_y) = m(s,t;\Adstar_z)$.
\end{proposition}

\begin{proof}
Let $w \in W$. 
When $m(s,t) <\infty$, set $\Delta := w_0^{\{s,t\}} $ as above.
It is well-known that if $\{s,t\} \subset \DesR(w)$ then $m(s,t)<\infty$
and $\ell(w\Delta) = \ell(w) - \ell(\Delta)$,
while if $\{s,t\}\cap \DesR(w)=\varnothing$ and $m(s,t) <\infty$, then $\ell(w\Delta) = \ell(w)  + \ell(\Delta)$ \cite[Lem. 1.2.1]{GP}.
Taking inverses gives similar left-handed properties.

Assume that  $\{s,t\} \cap \DesR(y) = \varnothing$ and $n=m(s,t;\Adstar_y)$ so that
$m(s,t) <\infty$. We first argue that the words in \eqref{rwords-eq} belong to $\iR(z)$ for some $z \in \I_*(W)$.
Let $\theta = \Adstar_y$ and note that $yw = \theta(w)^* y$. If $\theta(\{s,t\}) = \{s,t\}$,
then by using the observations in the previous paragraph one can check that 
$y \Delta \in \I_*(W)$ and that 
both words in \eqref{rwords-eq} belong to $\iR(y \Delta)$.
In this case, since conjugation by $\Delta$ preserves $\{s,t\}$ and since $\Delta$ is central in $W_{\{s,t\}}$
when $m(s,t)$ is even, it follows that $m(s,t;\Adstar_z) = m(s,t;\theta)$ for $z=y\Delta$.

The exchange condition implies that $s^*$ (respectively, $t^*$) is a left descent of $y\Delta$ if and 
only if $\theta(s)$ (respectively, $\theta(t)$) belongs to $\{s,t\}$.
Therefore if $\theta(\{s,t\})$ and  $\{s,t\}$ are disjoint
then $\Delta^* y \Delta \in \I_*(W)$ and both words in \eqref{rwords-eq} belong to $\iR(\Delta^* y \Delta)$.
In this case 
if $z=\Delta^* y \Delta$ then
 $\Adstar_z(\{s,t\}) = \{ \Delta \theta(s) \Delta, \Delta \theta(t)\Delta\}$ 
must also be disjoint from $\{s,t\}$,
so $m(s,t;\Adstar_{z}) = m(s,t;\theta) = m(s,t)$.

By similar reasoning,
if $\theta(s)=s$ and $\theta(t) \notin\{s, t\}$ then neither $s^*$ nor $t^*$ is a left descent of $ys\Delta$.
It follows that both words in \eqref{rwords-eq} belong to $\iR(z)$ for $z := \Delta^* ys\Delta \in \I_*(W)$
and that $\Adstar_z(t) \notin \{s,t\}$, so $m(s,t;\Adstar_{z}) = m(s,t;\theta) = m(s,t)$.
The final case when $\theta(s)\notin\{ s,t\}$ and $\theta(t) = t$ is handled by a symmetric argument.


Now assume both words in \eqref{rwords-eq} belong to $\iR(z)$ for some $z \in \I_*(W)$.
We must have $\{s,t\} \cap \DesR(y) = \varnothing$ by the definition of an involution word
and $m(s,t) <\infty$ since both $s$ and $t$ are right descents of $z$.
If $n \neq m(s,t;\Adstar_y)$ then 
our hypothesis gives us one pair of involution words of length $k+n$ for $z$
while the preceding argument gives another pair of involution words of length $k+m(s,t;\Adstar_y)$ for another element of $\I_*(W)$.
But each of the shorter words is a prefix of one of the longer words, so swapping these prefixes in the longer words
should result in two new elements in $\iR(z)$. This is impossible since these words would have adjacent repeated letters.
\end{proof}

The following is an analogue of the already mentioned fact
that
if $w \in W$ and $s,t \in \DesR(w)$, 
then $m(s,t)<\infty$ and $w$ has reduced words ending with both of the $m(s,t)$-letter sequences
$(\dots,t,s,t,s)$ and $(\dots,s,t,s,t)$ \cite[Lem. 1.2.1]{GP}.

\begin{proposition}\label{tprop2}
Suppose  $z \in \I_*(W)$, $s,t \in \DesR(z)$, and $n=m(s,t;\Adstar_z)$.
Then $n<\infty$ and there exists a unique $ y \in \I_*(W)$
such that the words 
\[
(r_1,r_2,\dots,r_k, \underbrace{\dots,t,s,t,s}_{n\text{ terms}}) \quand (r_1,r_2,\dots,r_k,\underbrace{\dots,s,t,s,t}_{n\text{ terms}})
\]
are both in $\iR(z)$ for some (equivalently, every) $(r_1,r_2,\dots,r_k) \in \iR(y)$.
\end{proposition}

\begin{proof}
Since $\{s,t\} \subset \DesR(w)$, we have $m(s,t)<\infty$.
The uniqueness of $y$ follows from the description of $\iR(z)$ in terms of the monoid action \eqref{uaction-eq},
so we just need to establish existence.
One can at least construct an element $y \in \I_*(W)$ with $\{s,t\}\cap \DesR(y)=\varnothing$ such that 
\be\label{qqqq} z = \underbrace{s^* \circ t^* \circ s^* \circ t^*  \circ \cdots}_{\alpha\text{ factors}} \circ y \circ \underbrace{\cdots \circ t \circ s \circ t \circ s}_{\alpha\text{ factors}}\ee
where $\alpha \geq 0$ is minimal: in terms of the action \eqref{uaction-eq}, $y$ is the first element in the sequence 
$z \underline s$,
$z \underline s \hs \underline t$,
$z \underline s \hs \underline t \hs \underline s$, \dots
whose length is less than the element which follows.
%
Appending the $\alpha$-letter sequence $(\dots,t,s,t,s)$ to any word in $ \iR(y)$ gives an element of $\iR(z)$,
and
by Proposition~\ref{tprop1} appending either of the $m(s,t;\Adstar_y)$-letter sequences $(\dots,t,s,t,s)$ or $(\dots,s,t,s,t)$ to any word in $ \iR(y)$
gives two involution words for some element of $\I_*(W)$.

In view of these properties, we cannot have $\alpha <m(s,t;\Adstar_y)$ as $s$ and $t$ are both in $\DesR(z)$,
and 
we cannot have $m(s,t;\Adstar_y) < \alpha$ as this would let us construct an involution word for $z$ with equal adjacent letters.
Thus $\alpha=m(s,t;\Adstar_y)$.
Proposition~\ref{tprop1} implies that $m(s,t;\Adstar_y)=m(s,t;\Adstar_z)=n$ so the result follows.
\end{proof}

These results extend \cite[Lem. 3.6]{HanssonHultman},
which asserts that if $s$ and $t$ are distinct right descents of $z \in \I_*(W)$
then
$\iR(z)$ contains two words of the form \eqref{rwords-eq} with $n=m(s,t)$
if and only if 
$m(s,t;\Adstar_z) = m(s,t) $.
Proposition~\ref{tprop1} implies the ``only if'' part of this claim 
while Proposition~\ref{tprop2} implies the ``if'' direction.

  \section{Relations for primed words}\label{rel-sect1}
  
An index $i$ is a \defn{commutation} 
in an involution word $(s_1,s_2,\dots,s_n)$ 
if  $s_i^* y = ys_i$ for $y:=  s_{i-1}^* \circ \cdots \circ s_2^* \circ s_1^* \circ s_1 \circ s_2 \circ \cdots \circ s_{i-1}$.
A \defn{primed involution word} for $z \in \I_*(W)$ is a sequence formed from an involution word for $z$ by adding primes to some set of letters indexed by commutations. The elements of such a sequence belong to $S \sqcup S'$ where $S' := \{ s' : s \in S\}$ is a duplicate set of formal symbols.
We write $\iR^+(z)$ for the set of all primed involution words for $z$. Figure~\ref{primed_invol4321-fig} shows an example of this set.

\begin{figure}[h]
\begin{center}
\input{primed_invol4321.png.tex}
\end{center}
\caption{Primed involution words for $z=(1,4)(2,3)\in S_4$ with $*=\id$. 
Each expression like $13'21$ is an abbreviation for a sequence like $(s_1,s_3',s_2,s_1)$ where $s_i := (i,i+1) \in S_4$.
Grey edges are \defn{primed braid relations} while colored edges are
\defn{primed half-braid relations} in the sense of Definitions~\ref{pbr-def} and \ref{pbr-half-def}.
}\label{primed_invol4321-fig}
\end{figure}
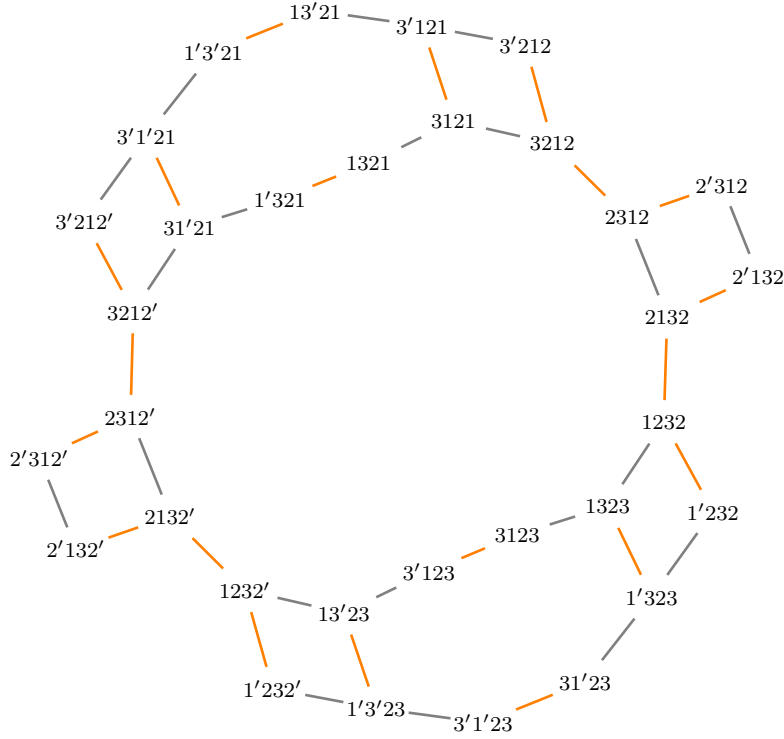

The number of commutations is the same in every involution word
for a fixed $z \in \I_*(W)$ \cite[Prop. 2.5]{Hultman3}: in view of \eqref{oo-eq}
this number must be $2 \rho_*(z) - \ell(z)$
where $\rho_*(z)$ denotes the common length of every word in $\iR(z)$.
The cardinality of $\iR^+(z)$ is therefore $2^{2 \rho_*(z) - \ell(z)}|\iR(z)|$.
Our main result in this section is a version of Theorem~\ref{hh-thm} for these sets.

\begin{lemma}\label{primed-lem}
Let $z \in \I_*(W)$ and 
$a =(a_1,a_2,\dots,a_l)\in \iR(z)$.
Choose distinct simple generators $s,t \in S$ with $n := m(s,t)<\infty$
and suppose  that
\[(a_{i+1},a_{i+2},\dots,a_{i+n})=(s,t,s,t,\dots)\]
for some $0\leq i \leq l-n$. Then the following properties hold:
\ben
\item[(a)] No index $j$ with $i+1<j<i+n$ is a commutation in $a$.
\item[(b)] If the indices ${i+1}$ and ${i+n}$ are both commutations in $a$ then $n=2$.
\item[(c)] Form $b  \in \iR(z)$
from $a$ by replacing the subword
$(a_{i+1},a_{i+2},\dots,a_{i+n})$  by
$(t,s,t,s,\dots)$. Then the set of commutations for $b$ is
the image of the set of commutations for $a$
under the transposition $(i+1 \leftrightarrow i+n)$.
\een
\end{lemma}

In particular,  part (c) means that ${i+1}$ (respectively, ${i+n}$) is a commutation in $a$
if and only if ${i+n}$ (respectively, ${i+1}$) is a commutation in $b$.

\begin{proof}
It follows from \eqref{oo-eq} that if there are $q$ commutations in $a$,
then we can form $2^q$ reduced words for $z$
by removing exactly one of $a_j^*$ or $a_j$ 
from
the doubled word  $(a_l^*,\dots, a_2^*,a_1^*,a_1,a_2,\dots,a_l)$ 
for each commutation $j$ in $a$.
Part (a) follows since if any of the indices
 ${i+2},{i+3},\dots,{i+n-1}$
were commutations, then one of these words would contain 
two adjacent letters equal to $s$ or $t$, which is impossible since each is reduced.
This observation is also noted within the proof of \cite[Lem. 3.3]{HanssonHultman}.

To prove part (b), suppose that both ${i+1}$ and ${i+n}$ are commutations in $a$.
Let $y = a_i^*\circ\cdots \circ a_1^* \circ a_1 \circ \cdots \circ a_i$
so that we have 
$ a_{i+1}^* \circ y \circ a_{i+1} = ys = s^*y $.
Also define 
$w: =a_{i+2}a_{i+3}\cdots a_{i+n-1}= tstst\cdots$ ($n-2$ factors).
If $n$ is odd then $w=w^{-1}$ so
part (a) implies that
\[
\ba
a_{i+n}^*\circ \cdots \circ a_{1}^*   \circ a_{1} \circ \cdots \circ a_{i+n} &
= s^* \circ w^* \circ s^*\circ y \circ s \circ w \circ s 
\\&= w^*ys ws \\&= s^*w^* ysw \\&= s^*w^*s^* yw
\ea\]
and therefore $ y(swsw) = (wsws)^* y$.
In this case $swsw = (st)^{n-1} = ts$ and $wsws= st$ so
we have $s^*t^* y = yts$.
Likewise if $n$ is even then
\[
\ba
a_{i+n}^*\circ \cdots \circ a_{1}^* \circ  a_{1} \circ \cdots \circ a_{i+n} &
= t^* \circ (w^{-1})^* \circ s^*\circ y \circ s \circ w \circ t
\\&=  (w^{-1})^*ys wt \\&= t^* (w^{-1})^* ysw \\&= t^* (w^{-1})^*s^* yw
\ea\]
so $y (swtw^{-1}) = (wtw^{-1}s)^* y$,
and as 
 $swtw^{-1} = (st)^{n-1}=ts$ and $wtw^{-1}s=st$
 it follows again that $s^*t^* y = yts$.
Thus for either parity of $n$ we have 
\[ t^* ys = s^*(s^*t^*y)s = s^*(yts)s = s^* y t = yst.\]
This implies that ${i+2}$ is a commutation in $a$, 
which by part (a) can only occur if $n=2$.
This proves part (b).

For part (c) 
we may assume that $i+n =l$.
The desired property is clear if $n=2$, so suppose $n\geq 3$.
Note that $b \in \iR(z)$ since involution words are closed under the usual braid relations.
The number of commutations is the same in all involution words for
$z$, so it follows from parts (a) and (b) 
that if ${i+1}$ is a commutation in $a$
then exactly one of ${i+1}$ or ${i+n}$ is a commutation in $b$,
and that if ${i+n}$ is a commutation in $b$ then exactly one of 
${i+1}$ or ${i+n}$ is a commutation in $a$.
Therefore it suffices to show that $i+1$ (respectively $i+n$) is not a commutation in both $a$
and $b$.

Again let $y= a_i^*\circ\cdots \circ a_1^* \circ a_1 \circ \cdots \circ a_i \in \I_*(W)$.
If $i+1$ were a commutation in both $a$ and $b$ then we would have $\{ y s,yt\} = \{s^*y,t^*y\}$
or equivalently $\Adstar_y(\{s,t\}) = \{s,t\}$,
and if $i+n$ were a commutation in both $a$ and $b$ then we would have $\{ z s,zt\} = \{s^*z,t^*z\}$
or equivalently $\Adstar_z(\{s,t\}) = \{s,t\}$.
Neither condition can occur since $m(s,t)$ is odd and
Proposition~\ref{tprop1} implies that $m(s,t;\Adstar_y) = m(s,t;\Adstar_z) = m(s,t)$.
\end{proof}

We may now describe versions of braid relations for words in $S\sqcup S'$.

\begin{definition}\label{pbr-def}
The \defn{primed braid relations} for the twisted Coxeter system $(W,S,*)$ are the  relations on words with all letters in $S \sqcup S'$ 
that have
\be
(\dash,s',t',\dash) \sim(\dash,t',s',\dash)
\ee
for any $s,t \in S$ such that $m(s,t)=2$,  as well as
\be\label{rel2}
(\dash,\underbrace{s,t,s,t,\dots}_{n\text{ factors}},\dash) \sim(\dash,\underbrace{t,s,t,s,\dots}_{n\text{ factors}},\dash)
\ee
for any $s,t \in S$ such that $n=m(s,t)\in \{2,3,4,\dots\}$, as well as
\be
(\dash,\underbrace{s',t,s,t,\dots,s}_{n\text{ factors}},\dash) \sim(\dash,\underbrace{t,s,t,s,\dots,t'}_{n\text{ factors}},\dash)
\ee
for any $s,t \in S$ such that $n=m(s,t) \in \{3,5,7,\dots\} $, and finally
\be\label{rel4}
(\dash,\underbrace{s',t,s,t,\dots,t}_{n\text{ factors}},\dash) \sim(\dash,\underbrace{t,s,t,s,\dots,s'}_{n\text{ factors}},\dash)
\ee
for any $s,t \in S$ such that $n=m(s,t)\in \{2,4,6,\dots\}$.
\end{definition}

Note that if $s,t\in S$ have $m(s,t) = 2$ then 
\[
(\dash,s,t,\dash) \sim(\dash,t,s,\dash)\quand (\dash,s',t,\dash) \sim(\dash,t,s',\dash)
\]
are  also primed braid relations, as these are special cases of \eqref{rel2} and \eqref{rel4}.

It follows from Lemma~\ref{primed-lem} that the primed braid relations preserve each set
$\iR^+(z)$. 
These relations cannot always span $\iR^+(z)$,  since none of them changes the number of primed letters in a word.
To get a spanning relation, we must add the following
  analogues of initial relations.

\begin{definition}
Choose  $J \subset S$ with $J=J^*$ such that $W_J = \langle J \rangle$ is finite.
If  $(s_1,s_2,\dots,s_n)$ and $(t_1,t_2,\dots,t_n)$ both belong to $\iR^+(w_0^J)$, then we refer to 
\[ (s_1,s_2,\dots,s_n,\dash) \sim  (t_1,t_2,\dots,t_n,\dash)\]
as a \defn{primed initial relation}, whose \defn{type} is the isomorphism class of $(W_J,J,*)$.
\end{definition}

\begin{theorem}\label{hh-thm2}
Let $z \in \I_*(W)$. Then $\iR^+(z)$ is an equivalence class under
the transitive closure of the primed braid relations for $(W,S)$ plus all primed initial relations
of type $\mathsf{A}_1$, $^2\mathsf{A}_3$, $\mathsf{BC}_3$, $\mathsf{D}_4$, $\mathsf{H}_3$, $\mathsf{I}_2(n)$, or $^2\mathsf{I}_2(n)$ for $2\leq n < \infty$.
\end{theorem}

This is the simplest possible extension of Theorem~\ref{hh-thm}.
The only type not listed earlier is $\mathsf{A}_1$, which contributes  $(s,\dash) \sim (s',\dash)$ for $s=s^* \in S$.

\begin{proof}
Fix a primed involution word $a = (s_1,s_2,\dots ,s_n ) \in \iR^+(z)$.
It suffices by Theorem~\ref{hh-thm} to show that this word is equivalent under 
the given relation to a word in $\iR(z) \subset \iR^+(z)$.
It is enough to check this when $s_1,s_2,\dots,s_{n-1} \in S$ and $s_n=s' \in S'$ for some $s\in S$.
In this case $zs=s^*z$ and $s \in \DesR(z)$.

If $n=0$ then  $z=s=s^*$ and $a = (s') \sim (s)$ using the primed initial relation of type $\mathsf{A}_1$.
If $n>0$ then $ s_{n-1}\in \DesR(z)$,
so $z$ has at least one other involution word ending in $s_{n-1}$.
Theorem~\ref{hh-thm}
implies that we may apply a sequence of braid relations and initial relations
to transform $(s_1,s_2,\dots,s_{n-1},s)$
to this word.

Consider what happens if we try to apply this sequence to the word $a$,
with primed braid relations in place of ordinary braid relations
but still using unprimed initial relations.
One of two cases must occur. Either some primed braid relation in this sequence 
moves the single prime from $s_n=s'$ to an earlier letter, 
or we reach a point where we wish to  apply an initial relation of length $n$.
(Otherwise, the relations would not change the last letter of 
$a$.)

In the first case, we may assume by induction (on the position of the primed letter)
that some sequence of primed braid relations and primed initial relations turns 
$a$ into an element of $\iR(z)$ as needed.
In the second case, we just substitute the initial relation we want to apply with a
primed initial relation that removes all primes from our word,
and we again get an element of $\iR(z)$.
\end{proof}


We do not need to include all primed initial relations of the types indicated 
to span $\iR^+(z)$. 
Our next theorem describes one possible choice for a minimal set of sufficient relations.
Recall the formula for $m(s,t;\theta)$ from \eqref{m-eq}.

\begin{definition}\label{pbr-half-def}
The \defn{primed half-braid relations} for the twisted Coxeter system $(W,S,*)$ are 
  relations on words with all letters in $S \sqcup S'$ that have
\be\label{phb1}
(\underbrace{ \dots,t,s,t,s,t,s}_{m(s,t;*)\text{ letters}},\dash) \sim
(\underbrace{ \dots,s,t, s, t, s,t}_{m(s,t;*)\text{ letters}},\dash)
\ee
for any $s,t \in S$ such that $\{s^*,t^*\} = \{s,t\}$ and $m(s,t)<\infty$, as well as
\be\label{phb2}
(\underbrace{ \dots,t,s,t, s, t, s}_{m(s,t;*)\text{ letters}},\dash) \sim
(\underbrace{ \dots,t,s,t, s, t, s'}_{m(s,t;*)\text{ letters}},\dash)
\ee
for any $s,t \in S$ such that either $s^*=s$ and $t^*=t$ and $m(s,t) \in \{4,6,8,\dots\}$, 
or $s^*=t$ and $t^*=s$ and $m(s,t) \in \{1,3,5,\dots\}$.
\end{definition}


If  $s=s^*$ then the  type $\mathsf{A}_1$  primed initial relation
$ (s,\dash) \sim (s',\dash)$ is the instance of  \eqref{phb2} with $s=t$.
If $s\neq t$ then \eqref{phb1} and \eqref{phb2} are primed initial relations of type $\mathsf{I}_2(n)$ or $^2\mathsf{I}_2(n)$ for $n:=m(s,t)<\infty$,
and all primed initial relations for the finite dihedral types arise in this way.
  
If 
$s^*=s$, $t^*=t$, and $m(s,t)=2$, then \eqref{phb2} is not considered to be a primed half-braid relation.
We exclude this case because then $m(s,t;*) = 2$
so
 \eqref{phb2} is the relation
$ (t,s,\dash) \sim (t,s',\dash)$,
 which one can alternatively get using
 the primed braid relations as
$(t,s,\dash) \sim (s,t,\dash)\sim (s',t,\dash)\sim (t,s',\dash)$.

\begin{figure}[h]
\[\barr{cccc}
\barr{c} \athreediagram \\ {^2\mathsf{A}_3}\\ {(\text{with }a^* = c)} \earr
&
\barr{c} \bthreediagram \\ \mathsf{BC}_3 \earr
&
\barr{c} \dfourdiagram \\ \mathsf{D}_4 \earr
&
\barr{c} \hthreediagram \\ \mathsf{H}_3 \earr
\earr
\]
\caption{Coxeter graphs for types 
 $^2\mathsf{A}_3$, $\mathsf{BC}_3$, $\mathsf{D}_4$, and $\mathsf{H}_3$}
 \label{fig1}
\end{figure}
 
\def\underlinea{A}
\def\underlineb{B}
\def\underlinec{C}
\def\underlined{D}
\def\underlinex{X}

In the following theorem, we assume that the 
  Coxeter graph for any twisted subsystem $(W_J,J,*)$ of type
 $^2\mathsf{A}_3$, $\mathsf{BC}_3$, $\mathsf{D}_4$, or $\mathsf{H}_3$
is labeled as in Figure~\ref{fig1}.

\begin{theorem}\label{pr-rel-thm}
Suppose $z \in \I_*(W)$. Then $\iR^+(z)$ is an equivalence class under 
the transitive closure of the primed braid relations, the primed half-braid relations,
and the symmetric relation with the following properties for each 
subset  $J=J^*\subseteq S$
labeled as in Figure~\ref{fig1}:
\begin{itemize}

\item If $(W_J,J,*)$ is of type $^2\mathsf{A}_3$ then
\[\ba  (b,c,a,b ,\dash)  &\sim ( b,c,b,a,\dash), \\
(b,c,a,b ,\dash)  &\sim ( b,c,a,b',\dash), \text{ and }\\
( b,c,b,a ,\dash)  &\sim ( b,c,b,a',\dash). \ea
\]

\item If $(W_J,J,*)$ is of type $\mathsf{BC}_3$ then 
\[\ba
( a,b,c,a,b,a,\dash) &\sim
 ( a,b,c,b,a,b,\dash),
 \\
 ( a,b,c,a,b,a,\dash) &\sim
 ( a,b,c,a,b,a',\dash), \text{ and }\\
  ( a,b,c,b,a,b,\dash) &\sim
 (a,b,c,b,a,b',\dash).
\ea
\]

\item  If $(W_J,J,*)$ is of type $\mathsf{D}_4$ then 
\[\ba
( d,b,a,c,b,a, c,d,\dash) &\sim
( d,b,a,c,b,a,d,c,\dash),
\\
( d,b,a,c,b,a, c,d,\dash) &\sim
( d,b,a,c,b,a, c,d',\dash), \text{ and }
\\
( d,b,a,c,b,a,d,c,\dash) &\sim
( d,b,a,c,b,a,d,c',\dash).
\ea
\]

\item  If $(W_J,J,*)$ is of type $\mathsf{H}_3$ then 
\[\ba
( a,c,b,a,c,b,a,b, c,\dash) &\sim
(  a,c,b,a,c,b,a,c,b,\dash), 
\\
( a,c,b,a,c,b,a,b, c,\dash) &\sim
(   a,c,b,a,c,b,a,b, c',\dash), \text{ and }
\\
( a,c,b,a,c,b,a,c,b,\dash) &\sim
(  a,c,b,a,c,b,a,c,b',\dash).
\ea
\]

\end{itemize}
\end{theorem}

This is also a very straightforward extension of Hansson and Hultman's results:
the first relations in each type involve no primed letters and are 
the same as the initial relations listed in \cite[Thm. 4.1]{HanssonHultman}.

\begin{proof}
Consider the two prefixes in the first  initial relation listed for each type.
One can check that $\sim$ relates each word
to any way of adding primes to commutations among its letters. 
Thus in type $^2\mathsf{A}_3$ 
we have 
$({b}',c,a,{b}')\sim ({b},c,a,{b}') \sim ({b},c,a,{b}) \sim  ({b}',c,a,{b})$
and
$(b',c,b,a') \sim(b,c,b,a') \sim (b,c,b,a) \sim (b',c,b,a)$,
while in type $\mathsf{BC}_3$, less trivially, we have 
\[ \ba \boxed{(a',b,c,a',b,a')} &\sim \boxed{(a,b,c,a',b,a')}  \sim(a,b,a',c,b,a') \sim (a,b,a,c,b,a') 
\\&\sim \boxed{(a,b,c,a,b,a')}  \sim \boxed{(a',b,c,a,b,a')} \sim(a,b,c,a,b,a') \\&
\sim\boxed{(a,b,c,a,b,a)}\sim \boxed{(a',b,c,a,b,a)}\sim(a,b,c,a,b,a)
\\&\sim(a,b,a,c,b,a) \sim(a,b,a',c,b,a) \sim \boxed{(a,b,c,a',b,a)} \\&\sim \boxed{(a',b,c,a',b,a)}
\ea\]
and
\[
\ba
\boxed{(a',b,c,b',a,b')}& \sim \boxed{(a,b,c,b',a,b')} \sim (a,c',b,c,a,b') \sim (c',a,b,c,a,b') 
\\&\sim  (c,a,b,c,a,b')  \sim (a,c,b,c,a,b') \sim \boxed{(a,b,c,b,a,b')}
\\&\sim \boxed{(a',b,c,b,a,b')} \sim (a,b,c,b,a,b')\sim \boxed{(a,b,c,b,a,b)}
\\&\sim \boxed{(a',b,c,b,a,b)}\sim (a,b,c,b,a,b) \sim (a,c,b,c,a,b) 
\\&\sim (c,a,b,c,a,b)\sim (c',a,b,c,a,b)\sim (a,c',b,c,a,b)
\\& \sim \boxed{(a,b,c,b',a,b)} \sim \boxed{(a',b,c,b',a,b)}.
\ea
\]
The eight boxed words in these chains of equivalences show all of the different ways of adding primes to the three commutations
in $(a,b,c,a,b,a)$ and $(a,b,c,b,a,b)$.

Checking our claim in the other two types is a similar (and easier) calculation.
Ignoring primes converts  $\sim$ into the transitive closure of the ordinary braid relations and the
extra relations listed in \cite[Thm. 4.1]{HanssonHultman}.
This relation connects all words in $\iR(z)$ \cite[Thm. 4.1]{HanssonHultman}. Thus, repeating the argument in 
the proof of Theorem~\ref{hh-thm2} shows 
that any word in $\iR^+(z)$ is connected by our relation $\sim$ to its unprimed form in $\iR(z)$,
which suffices to prove the result.
\end{proof}

\section{Relations for Hecke words}\label{rel-sect2}

There is another generalization of involution words which is sometimes relevant.
An \defn{involution Hecke word} for $z \in \I_*(W)$ is any finite sequence 
$(s_1,s_2,\dots,s_n)$ with $s_i \in S$ and
$z = s_n^* \circ \cdots \circ s_2^* \circ s_1^*  \circ s_1 \circ s_2 \circ \cdots \circ s_n.$
In type $\mathsf{A}$ when $*=\id$ 
these sequences form the set of \defn{orthogonal Hecke words} for $z$ defined in \cite[\S1.2]{Marberg2019a}.

Let $\iHH(z)$ denote the set of involution Hecke words for $z \in \I_*(W)$. This set is infinite if $z \neq 1$,
and  most of its elements are not reduced words.
However, it is easy to describe a relation spanning its elements.

\begin{proposition}\label{hecke-prop}
Suppose $z \in \I_*(W)$. Then $\iHH(z)$ is an equivalence class under
the transitive closure of the equivalence relation in 
Theorem~\ref{hh-thm}
and the symmetric relations $ (\dash,s,s,\dash)\sim (\dash,s,\dash)$ for each $s \in S$.
\end{proposition}

\begin{proof}
Denote this equivalence relation by $\sim$. It
suffices by Theorem~\ref{hh-thm} to show that any $a=(a_1,a_2,\dots,a_n) \in \iHH(z)$ is equivalent 
under $\sim$
to 
an element of $\iR(z)$.
If $a \notin \iR(z)$ then there is a minimal index $i \in \{1,2,\dots,n-1\}$ such that 
$a_{i}$ is a right descent of $a_{i-1}^* \circ \cdots \circ a_2^* \circ a_1^* \circ a_1 \circ a_2 \circ \cdots \circ a_{i-1}$,
in which case $a \sim b$ for some word $b = (b_1,\dots,b_{i}, a_{i},\dots,a_n)$
with $b_i = a_{i}$,
and hence also with $b \sim (b_1,\dots,b_i,a_{i+1},\dots,a_n)$.
By induction on length we may assume that the latter word is equivalent to an element of $\iR(z)$ as needed.
\end{proof}

Let $\iH(z)$ denote the set of involution Hecke words for $z \in \I_*(W)$
that are reduced words (usually for elements of $W$ other than $z$).
Then we have \be\label{iH-eq} \iH(z) = \bigsqcup_{w \in \cB_*(z)} \cR(w)\ee
where $\cB_*(z) := \{ w \in W : (w^{-1})^* \circ w = z\}$ is the set of \defn{Hecke atoms} studied in \cite{HMP2}.
For an example of $\iH(z)$, see  Figure~\ref{hecke4321-fig}.
Another interesting problem is to describe a relation that spans and preserves 
each of these sets.
 
 \begin{figure}[h]
\begin{center}
\input{hecke4321.png.tex}
\end{center}
\caption{Reduced involution Hecke words for   $z=(1,4)(2,3)\in S_4$ with $*=\id$. 
Words are represented here using the same conventions as in Figures~\ref{twisted_invol4321-fig} and \ref{primed_invol4321-fig}.
Grey edges correspond to ordinary braid relations while colored edges correspond to the
\defn{mixed half-braid relations} in the sense of Definition~\ref{mixed-def}.}\label{hecke4321-fig}
\end{figure}

Note that $\iH(z)$ contains $\iR(z)$ and
 is preserved by all ordinary braid relations, but not by the initial relations \eqref{init-rel-eq},
 at least as given in Definition~\ref{init-rel-def}, as these may lead from $\iH(z) \setminus \iR(z)$ to words that are not reduced.
 The appropriate analogue of initial relations for $\iH(z)$ is as follows.
 
 \begin{definition}
Choose  $J \subset S$ with $J=J^*$ such that $W_J = \langle J \rangle$ is finite
with longest element $w_0^J$. Suppose $(s_1,s_2,\dots,s_p)$ and $(t_1,t_2,\dots,t_q)$ are two elements of $\iH(w_0^J)$.
We refer to any relation of the form
\be\label{init-hecke-rel-eq}(s_1,s_2,\dots,s_p,r_1,r_2,\dots,r_k) \sim  (t_1,t_2,\dots,t_q,r_1,r_2,\dots,r_k) \ee
where $(r_1,r_2,\dots,r_k)$ is a reduced word for an element of 
\[^JW := \{ w \in W : \ell(sw)> \ell(w) \text{ for all }s \in J\}\]
as an \defn{initial Hecke relation}, whose
 \defn{type} is the isomorphism class of $(W_J,J,*)$.
\end{definition}

Such relations preserve each of the sets $\iH(z)$ for $z \in \I_*(W)$.

\begin{example}
The initial Hecke relations of type $\mathsf{I}_2(n)$ or $^2 \mathsf{I}_2(n)$ are
\be\label{o-mixed-eq}
  (\underbrace{s,t,s,t,\dots}_{p\text{ terms}},r_1,r_2,\dots,r_k)
\sim (\underbrace{t,s,t,s,\dots}_{q\text{ terms}},r_1,r_2,\dots,r_k)
\ee
where $s,t \in S$ have $\{s^*,t^*\}=\{s,t\}$ and $m(s,t;*) \leq p,q \leq m(s,t)<\infty$ and 
where $(r_1,r_2,\dots,r_k) $ 
is a reduced word for $w \in W$ with $\ell(sw) = \ell(tw) > \ell(w)$.
\end{example}

 \begin{theorem} \label{hecke-thm}
Suppose $z \in \I_*(W)$. Then $\iH(z)$ is an equivalence class under
the transitive closure of the braid relations for $(W,S)$ plus all initial Hecke relations
of type $^2\mathsf{A}_3$, $\mathsf{BC}_3$, $\mathsf{D}_4$, $\mathsf{H}_3$, $\mathsf{I}_2(n)$, or $^2\mathsf{I}_2(n)$ for $2\leq n < \infty$.
\end{theorem}

\begin{proof}
Denote the given equivalence relation by $\sim$. Given $b,c \in \iH(z)$, write $b \to c$ if $\ell(b) \geq \ell(c)$ and $b\sim c$  
such that the two words differ by a sequence of braid relations or a single initial Hecke relation.

Suppose $a = (a_1,a_2,\dots,a_n) \in \iH(z)$.
By Theorem~\ref{hh-thm} it suffices to show that $a$ is equivalent under $\sim$ to an element of $\iR(z)$.
We will prove a more specific claim:    there are words $a^0, a^1, a^2, \dots, a^l $ with 
\[a = a^0 \to a^1 \to a^2\to \cdots \to a^l \in \iR(z).\]
This is clear if $n=0$. Assume $n>0$, let $b:=(a_1,a_2,\dots,a_{n-1})$, and write
$y $ for the element of $\I_*(W)$ with $b  \in \iH(y)$.
By induction we may assume that there are words $b^i $ with
$b = b^0 \to b^1 \to b^2\to \cdots \to b^l \in \iR(y)$.
For each $i$ let $a^i$ be the word formed by adding $a_n$ to the end of $b^i$.

Suppose there exists a minimal $i  \in \{1,2,\dots,l\}$ such that $a^{i-1} \to a^i$ fails to hold.
Then $b^{i-1}$ and $b^i$ must differ by a single initial Hecke relation, so we may assume that 
these words are the left and right sides of \eqref{init-hecke-rel-eq} with $p\geq q$.
Since $i$ is minimal, we have $a=a^0 \sim a^{i-1}$ so $a^{i-1} \in \iH(z)$, which means that the subword $(r_1,r_2,\dots,r_k,a_n)$
must be a reduced word for an element of $W$ that is not in $^JW$.

As $(r_1,r_2,\dots,r_k)$ is itself a reduced word for an element of $^JW$,
it follows (e.g., from   \cite[Lem. 1.2.6]{GP}) that $(r_1,r_2,\dots,r_k,a_n)$
 is connected by a sequence of braid relations to 
 $(s,  r_1,  r_2,\dots,  r_k)$ for some $s \in J$ with 
 $  (s_1,s_2,\dots,s_p, s) \in \iH(w_0^J)$.
Therefore
$ a^{i-1} \to (s_1,\dots,s_p, s,  r_1,\dots,  r_k) \to b^i$,
which implies that $y=z$, so on setting
$\tilde a^i := (s_1,\dots,s_p, s,  r_1,\dots,  r_k)$ we have \[ a=a^0\to \dots \to a^{i-1} \to \tilde a^i \to b^i \to \dots \to b^l \in \iR(y) = \iR(z) \]
as predicted by our claim.

If no such $i$ exists and $a^l \in \iR(z)$ then $a^0 \to a^1 \to \dots \to a^l$ as needed.
If $a^l \notin \iR(z)$ then $a_n$ must be a right descent of $y$. In this case
we may assume by Theorem~\ref{hh-thm} that $b^l$ ends in $a_n$, so 
there must exist an index $i$ where $a^{i-1} \to a^i$ fails,
and we can apply the argument above to deduce our claim.
\end{proof}

Like Theorem~\ref{hh-thm2}, the preceding result
includes more relations than are necessary to generate an equivalence relation spanning the sets of interest.

\begin{definition}\label{mixed-def}
The \defn{mixed half-braid relations} 
for the twisted Coxeter system $(W,S,*)$ are the  relations on words with all letters in $S$ of the form
\be\label{mixed-eq}
  (\underbrace{s,t,s,t,\dots}_{p\text{ terms}},r_1,r_2,\dots,r_k)
\sim (\underbrace{s,t,s,t,\dots}_{p+1\text{ terms}},r_1,r_2,\dots,r_k)
\ee
where
 $s,t \in S$ are such that $m(s,t;*) \leq p < m(s,t)<\infty$ and 
  $(r_1,r_2,\dots,r_k) $ is a reduced word
for some $w \in W$ with $\ell(sw) = \ell(tw)>\ell(w)$.
 \end{definition}
 
 Note that $m(s,t;*) < m(s,t) <\infty$ implies that $\{s^*,t^*\} = \{s,t\}$.
The transitive closure of the mixed half-braid relations and the usual braid relations
include all of the type $\mathsf{I}_2(n)$ or $^2 \mathsf{I}_2(n)$  initial Hecke relations \eqref{o-mixed-eq}.

We may now give a version of 
Theorem~\ref{pr-rel-thm} for reduced involution Hecke words.
As in that result, we assume below that the 
  Coxeter graph for any subsystem $(W_J,J,*)$ of type
 $^2\mathsf{A}_3$, $\mathsf{BC}_3$, $\mathsf{D}_4$, or $\mathsf{H}_3$
is labeled as in Figure~\ref{fig1}.

\begin{theorem}\label{pr-rel-hecke-thm}
Suppose $z \in \I_*(W)$. Then $\iH(z)$ is an equivalence class under 
the transitive closure of the braid relations, the mixed half-braid relations,
and the symmetric relation with the following properties for each $J=J^*\subseteq S$:
\begin{itemize}

\item If $(W_J,J,*)$ is of type $^2\mathsf{A}_3$ then 
\[ \ba (b,c,a,b ,\dash) 
&\sim ( b,c,b,a,\dash)\sim ( b,c,b,a,b,\dash).\ea
\]

\item If $(W_J,J,*)$ is of type $\mathsf{BC}_3$ then 
\[\ba
(a,b,c,a,b,a,\dash) &\sim
 (a,b,c,b,a,b,\dash) \sim (a,b,c,b,a,b,a,\dash).
\ea
\]

\item  If $(W_J,J,*)$ is of type $\mathsf{D}_4$ then 
\[\ba
( d,b,a,c,b,a,c,d,\dash) &\sim
( d,b,a,c,b,a,d,c,\dash) \\&\sim ( d,b,a,c,b,a,d,c,d,\dash).
\ea
\]

\item  If $(W_J,J,*)$ is of type $\mathsf{H}_3$ then 
\[\ba
( a,c,b,a,c,b,a,b,c,\dash) &\sim
(  a,c,b,a,c,b,a,c,b,\dash) \\&\sim (  a,c,b,a,c,b,a,c,b,c,\dash)\\&\sim (  a,c,b,a,c,b,a,c,b,a,\dash)\\&\sim (  a,c,b,a,c,b,a,c,b,a,b,\dash).
\ea
\]

\end{itemize}
The meaning of the symbols ``$\dash$'' here is slightly different than above: 
in each sequence of relations, this symbol stands for an arbitrary reduced word 
 for an element of $^JW$ (rather than an arbitrary word as previously).
\end{theorem}

One can translate this result into a description of an equivalence relation on the group $W$
classifying the sets $\cB_*(z) $ in \eqref{iH-eq}.
In types $\mathsf{A}_n$ and $\mathsf{BC}_n$ the theorem can be used in this way to recover \cite[Thm. 6.4]{HMP2}
and \cite[Thm. 9.4]{HM2021}.

\begin{proof}
Let $\approx$ be the transitive closure of the relation on $W$ that connects two elements if they have reduced words that differ by a mixed half-braid relation.
In types $^2\mathsf{A}_3$, $\mathsf{BC}_3$, $\mathsf{D}_4$, and $\mathsf{H}_3$ the respective sizes of 
the sets $\cB_*(w_0^J)$ are 7, 13, 29, and 37.
These sets are divided by $\approx$ into  3, 3, 3, and 5 equivalence classes, respectively.
It suffices to check that each equivalence class has an element with a reduced word that matches
one of the prefixes of the relations given in the relevant type.
This, along with our other observations in this proof, follows by a straightforward (finite) computer calculation. 
\end{proof}

\section{Simply braided systems}\label{app-sect1}

We say that a twisted Coxeter system $(W,S,*)$ is \defn{simply braided} if no subsystem $(W_J,J,*)$ with $J=J^*\subseteq S$
is of type $^2\mathsf{A}_3$, $\mathsf{BC}_3$, $\mathsf{D}_4$, or $\mathsf{H}_3$.
As Hansson and Hultman observe in \cite[Cor. 5.1]{HanssonHultman}, this occurs precisely when
each set $\iR(z)$ is spanned by just the braid relations
 and half-braid relations.
Theorems~\ref{hh-thm2} and \ref{hecke-thm} imply a similar fact about $\iR^+(z)$ and $\iH(z)$.  

\begin{theorem}
The following properties are equivalent:
\ben
\item[(a)] The twisted Coxeter system $(W,S,*)$ is simply braided.

\item[(b)] Each set
 $\iR(z)$  for $z \in \I_*(W)$ is an equivalence class for the transitive closure of the braid relations
\eqref{braid-eq} and half-braid relations \eqref{i2-eq1} and \eqref{i2-eq2}.

\item[(c)] Each set
 $\iR^+(z)$  for $z \in \I_*(W)$ is an equivalence class for the transitive closure of the primed braid relations
and primed half-braid relations in Definitions~\ref{pbr-def} and \ref{pbr-half-def}.

\item[(d)] Each set
 $\iH(z)$ for $z \in \I_*(W)$ is an equivalence class for the transitive closure of the braid relations  \eqref{braid-eq}
and mixed half-braid relations  \eqref{mixed-eq}.
\een
\end{theorem}

This theorem significantly strengthens \cite[Prop. 1.10]{Marberg2017},
which is equivalent to just the assertion that (b) $\Rightarrow$ (d).
 
We define a twisted Coxeter system $(W,S,*)$ to be \defn{irreducible} if $*$ acts transitively on the 
components of the relevant Coxeter graph.
This occurs if and only if $W$ is irreducible or has two irreducible factors interchanged by $*$.
Comparing the definition of simply braided with the classification of the Coxeter graphs of positive type in \cite{Humphreys}
yields the following observation about irreducible systems:

 \begin{proposition}
Suppose
$(W,S,*)$ is irreducible of finite or affine type. 
Then $(W,S,*)$ is simply braided if and only if 
either $s^* \neq s$ for all $s \in S$;
$*=\id$ and $W$ has type ${\mathsf{A}}_n$ or $\tilde {\mathsf{A}}_n$;
or 
$W$ has type $\tilde {\mathsf{A}}_2$, $\tilde {\mathsf{C}}_2$, $\tilde {\mathsf{G}}_2$, or ${\mathsf{I}}_2(n)$ for $3 \leq n \leq \infty$.
 \end{proposition}
 




For each $w \in W$, there is a natural connected graph with vertex set $\cR(w)$,
in which two reduced words form an edge if they differ by a single braid relation.
The properties of this \defn{reduced word graph} (in particular, its diameter) 
have been studied 
for finite Coxeter systems in several places \cite{DahlbergKim,DehAut,GaetzGao,ReinerRoichman}. 

There are three similar graphs one can associate to the finite vertex sets $\iR(z)$, $\iR^+(z)$, and $\iH(z)$,
which we call the \defn{involution word graph}, \defn{primed involution word graph},
and \defn{involution Hecke word graph} of $z \in \I_*(W)$.
In the (primed) involution word graph, 
each edge corresponds to a single (primed) braid relation or (primed) half-braid relation;
in the involution Hecke word graph, each edge corresponds to a single braid relation or mixed half-braid relation.
We have already seen examples of these graphs in Figures~\ref{twisted_invol4321-fig}, \ref{primed_invol4321-fig}, and \ref{hecke4321-fig}.
They are always connected if $(W,S,*)$ is simply braided,
and their properties may be of independent interest. At a minimum,
such graphs often make for interesting pictures. Figures~\ref{invol54321-fig},
\ref{primed_invol54321-fig},
\ref{twisted54321-fig}, and 
\ref{primed_twisted54321-fig} show some larger examples.

\begin{figure}[h]
\begin{center}
\input{invol54321.png.tex}
\end{center}
\caption{Involution word graph for the longest element in the twisted Coxeter system of type $\mathsf{A}_4$. Grey edges correspond
to ordinary braid relations. 
}
\label{invol54321-fig}
\end{figure}
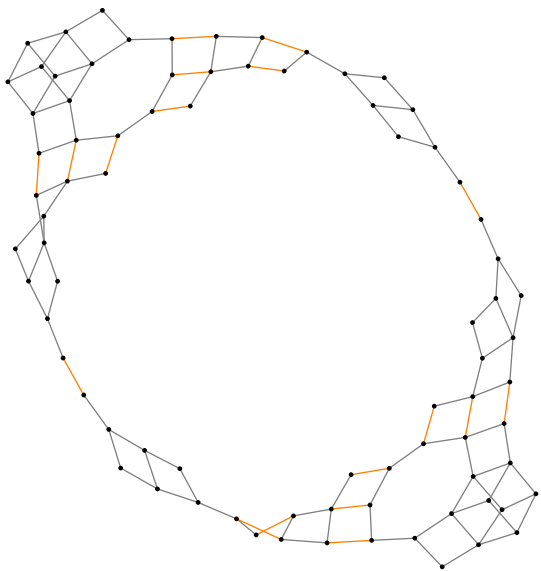

\begin{figure}[h]
\begin{center}
\input{primed_invol54321.png.tex}
\end{center}
\caption{Primed involution word graph for the longest element in the twisted Coxeter system of type $\mathsf{A}_4$. Grey edges correspond
to ordinary braid relations. 
}\label{primed_invol54321-fig}
\end{figure}

\begin{figure}[h]
\begin{center}
\input{twisted54321.png.tex}
\end{center}
\caption{Involution word graph for the longest element in the twisted Coxeter system of type $^2\mathsf{A}_4$. Grey edges correspond
to ordinary braid relations.
}\label{twisted54321-fig}
\end{figure}

\begin{figure}[h]
\begin{center}
\input{primed_twisted54321.png.tex}
\end{center}
\caption{Primed involution word graph for the longest element in the twisted Coxeter system of type $^2\mathsf{A}_4$. Grey edges correspond
to ordinary braid relations. 
}\label{primed_twisted54321-fig}
\end{figure}

\section{Relations in type A}\label{app-sect2}

Fix an integer $n\geq 2$ and let $[n] := \{1,2,\dots,n\}$.
%
For each $i \in \ZZ$ let $s_i$ denote the permutation of $\ZZ$
 that interchanges $i+nk$ and $i+1+nk$ for each $k \in \ZZ$ while fixing all other integers. 
Then 
$\{s_1,s_2,\dots,s_n\}$
is a Coxeter generating set for the \defn{affine symmetric group} $\tS_n$
of bijections $w : \ZZ \to \ZZ$ with
 $
 w(i+n)= w(i) +n$ for all $i \in \ZZ $ and $ \sum_{i\in [n]} w(i) = \sum_{i \in [n]} i.
 $
 The corresponding length function $\ell :\tS_n \to \NN$ has $\ell(ws_i) = \ell(w) - 1$ if and only if $w(i) > w(i+1)$.

In this final section, we discuss what our general results reduce to 
when
$W=\tS_n$,  $S= \{s_1,s_2,\dots,s_n\}$, and $*=\id$.
All statements here apply equally well to the finite symmetric group $S_n\cong\langle s_1,s_2,\dots,s_{n-1}\rangle\subset \tS_n$. 

Write 
$\cR_{\mathsf{inv}}(z) := \iR(z) $ and define $ \cR^+_{\mathsf{inv}}(z)$, 
$ \mathcal{H}_{\mathsf{inv}}(z) $, and $ \cHinvred(z) $ analogously.
We 
abbreviate $(s_{a_1}, s_{a_2},\dots,s_{a_l})$
as the word $a_1a_2\cdots a_l$ with $a_i \in [n]$,
and write elements of $\cR^+_{\mathsf{inv}}(z) $ as words 
with letters in $\{ 1' < 1<\dots<n'<n\}$.

\begin{lemma}\label{fixed-lem}
If $a_1a_2\cdots a_l$ is an involution word for $z=z^{-1} \in \tS_n$, then $i \in [l]$ is a commutation
if and only if $a_i$ and $1+a_i$ are fixed points of 
$ y:=s_{i-1} \circ \cdots \circ s_2 \circ s_1 \circ s_2 \circ \cdots \circ s_{i-1},$
in which case $(a_i, 1+a_i)$ is a cycle of $s_{a_i} \circ y \circ s_{a_i}$.
\end{lemma}
\begin{proof}
A generator $s_a$ commutes with $y$ if and only if either 
$y(a) = a+1$ and $y(a+1) = a$, or $y(a) = a$ and $y(a+1) = a+1$. The first case cannot occur when $a=a_i$
as then $s_a$ would be a right descent of $y$.
 \end{proof}

An involution word in any type has no adjacent repeated letters,
so a primed involution word has no consecutive subwords of the form
$aa$, $a'a$, $aa'$, or $a'a'$. 

\begin{proposition}
A primed involution word for $z=z^{-1}\in \tS_n$ 
has no consecutive subwords of the form
$a'(a+1)'$, $(a+1)'a'$, $ab'a$, $a'b'a$, $a'ba'$, $ab'a'$, or $a'b'a'$
and does not begin with any consecutive subwords of the form
$ a(a+1)'$, $ (a+1)a'$, $aba$, $a'ba$, or $aba'$,
replacing $a+1$ by $1$ if $a=n$.
Such a word may only contain $aba$, $a'ba$, or $aba'$ 
as consecutive non-initial subwords if $b -a \in \{-1,1\} +n\ZZ$.
\end{proposition}

\begin{proof}
The seven consecutive subwords are forbidden by Lemmas~\ref{primed-lem} and \ref{fixed-lem}.
The first two initial subwords are forbidden as $s_{a}$ and $s_{a+1}$ do not commute.
The last three initial subwords are forbidden since no (unprimed) involution word can begin with $aba$, for
 $s_a \circ s_b \circ s_ a \circ s_a \circ s_b \circ s_a$ is either 
$s_as_b $ if $s_a$ and $s_b$ commute 
or else $s_bs_as_b$, which in either case is equal to $s_b \circ s_ a \circ s_a \circ s_b$.
The last claim holds since $s_a\circ s_b \circ s_a = s_a\circ s_b$ if $b-a \notin \{-1,1\}+n\ZZ$.
\end{proof}

Define $\simA$ to be the transitive closure of the 
symmetric relation on words with all letters in $\{ 1' < 1< 2'<2<\dots<n'<n\}$
that has
\be 
(a,\dash) \simA (a',\dash)
\quand
(a,b,\dash) \simA (b,a,\dash)
\ee
for all $a,b \in [n]$, along with
\be
\ba
 (\dash,a,b,\dash) &\simA (\dash,b,a,\dash), \\
  (\dash,a,b',\dash) &\simA (\dash,b',a,\dash),  \text{ and } \\
   (\dash,a',b',\dash) &\simA (\dash,b',a',\dash)
   \ea
   \ee
   for all $a,b \in [n]$ with $a-b \notin \{-1,0,1\} + n\ZZ$,
   and finally with
   \be
   \ba
   (\dash, a,b,a,\dash) &\simA (\dash, b,a,b,\dash) \text{ and }
   \\
   (\dash, a',b,a,\dash) &\simA (\dash, b,a,b',\dash) 
   \ea
   \ee
   for all $a,b\in [n]$ with $a-b \in \{-1,1\} + n\ZZ$. Below, let $z=z^{-1} \in \tS_n$.
   Theorems~\ref{hh-thm} and \ref{hh-thm2} and Proposition~\ref{hecke-prop} applied to type $\tilde{\mathsf{A}}_{n-1}$
   have this corollary:
   
\begin{corollary}
The set $\cR_{\mathsf{inv}}^+(z)$ is an equivalence class under $\simA$.
The set $\cR_{\mathsf{inv}}(z)$ is an equivalence class for the restriction of  $\simA$
 to unprimed words,
%
%
while $\cH_{\mathsf{inv}}(z)$ is an equivalence class under the transitive closure of 
 the same restriction
and the symmetric relations $(\dash, a,a,\dash)\sim(\dash,a,\dash)$. 
\end{corollary}


Next define $\equivA$ to be the transitive closure of the 
symmetric relation on words with all letters in $[n]$
that has
\be
(\dash, a,b,\dash) \equivA (\dash,b,a,\dash)
\ee
for all $a,b \in [n]$ with $a-b \notin \{-1,0,1\} + n\ZZ$, along with
  \be\label{aff1}
   (\dash, a,b,a,\dash) \equivA (\dash, b,a,b,\dash) 
   \ee
   for all  $a,b \in [n]$ with $a-b \in \{-1,1\} + n\ZZ$, and finally with
  \be\label{aff2}
   (a,b,c_1,c_2,\dots,c_k)  \equivA (a,b,a,c_1,c_2,\dots,c_k)
   \ee
   for all  $a,b \in [n]$ with $a-b \in \{-1,1\} + n\ZZ$
   and  reduced words $(c_1,c_2,\dots,c_k)$ for permutations $w \in \tS_n$
   with $w^{-1}(a) < w^{-1}(a+1) $ and $w^{-1}(b) < w^{-1}(b+1) $.
   Note that combining \eqref{aff1} and \eqref{aff2} gives
   \be\label{aff3}
      (a,b,c_1,c_2,\dots,c_k)  \equivA (b,a,c_1,c_2,\dots,c_k).
      \ee
By applying Theorem~\ref{hecke-thm} to type $\tilde{\mathsf{A}}_{n-1}$, we obtain:

\begin{corollary}
The set $\cHinvred(z)$ is an equivalence class under $\equivA$.
\end{corollary}


If  $[w_1,w_2,\dots,w_n]$ is an integer sequence
with $w_i -w_j \notin n\ZZ$ for all $1\leq i < j \leq n$
then there is a unique $w \in \tS_n$ such that 
$w_i = w(i+d)$ for all $i \in [n]$ for some $d \in \ZZ$.
We identify $[w_1,w_2,\dots,w_n]$ with this element $w \in \tS_n$.

Given $z =z^{-1} \in \tS_n$, let $a_1<a_2<\dots<a_l$ be the numbers in $[n]$
with $a_i \leq z(a_i)$ and define $\alpha_{\min}(z)\in\tS_n$ to be the inverse of the element
given by $[z(a_1), a_1, z(a_2), a_2,\dots, z(a_l),a_l]$ with all duplicate entries removed.
If $n = 5$ and $z = s_2s_3s_2$
then 
$\alpha_{\min}(z) = [1,4,2,3,5]^{-1} = [1,3,4,2,5] \in \tS_5$.

\begin{proposition}
If $z=z^{-1} \in \tS_n$ then the set 
$\cB(z) := \left\{ w \in \tS_n : w^{-1} \circ w = z\right\}$
is the equivalence class of $\alpha_{\min}(z)$ under the transitive closure of 
the symmetric relation on $\tS_n$ that has $u^{-1} \sim v^{-1} \sim w^{-1}$
whenever 
$ u= [\dash, c,b,a,\dash]$, $v=  [\dash, c,a,b,\dash]$, and $w=  [\dash, b,c,a,\dash]$
for some integers $a<b<c$, where the corresponding dashes are identical subwords.
\end{proposition}

\begin{proof}
We have $\alpha_{\min}(z) \in \cB(z)$ by \cite[Prop. 6.8]{M2019},
and $u,v,w \in \tS_n$ have
$ u^{-1}= [\dash, c,b,a,\dash]$, $v^{-1}=  [\dash, c,a,b,\dash]$, and $w^{-1}=  [\dash, b,c,a,\dash]$
for some $a<b<c$
if and only if $u$, $v$, and $w$ have reduced words related as in \eqref{aff2} or \eqref{aff3}.
The result then follows by the word property  for $\tS_n$ via
\eqref{iH-eq}.
\end{proof}

\end{document}

%% file: preamble.tex
\usepackage{latexsym}
\usepackage{amssymb}
\usepackage{amsthm}
\usepackage{amscd}
\usepackage{amsmath}
\usepackage{mathrsfs}
\usepackage{graphicx}
\usepackage{shuffle}
\usepackage[colorlinks=true, pdfstartview=FitV, linkcolor=blue, citecolor=blue, urlcolor=blue]{hyperref}
\usepackage{mathtools}
\usepackage[bbgreekl]{mathbbol}
\usepackage{mathdots}
\usepackage{ytableau}
\usepackage[enableskew,vcentermath]{youngtab}
\usepackage{tikz-cd}

\usepackage{mathtools}
\usepackage{amssymb}

\usepackage[colorinlistoftodos]{todonotes}
\usetikzlibrary{chains,scopes}
\usetikzlibrary{shapes.geometric,positioning}

\DeclareSymbolFontAlphabet{\mathbb}{AMSb}
\DeclareSymbolFontAlphabet{\mathbbl}{bbold}

\definecolor{darkred}{rgb}{0.7,0,0} 
\newcommand{\defn}[1]{{\color{darkred}\emph{#1}}} 

\numberwithin{equation}{section}
\allowdisplaybreaks[1]

\theoremstyle{definition}
\newtheorem* {theorem*}{Theorem}
\newtheorem* {proposition*}{Proposition}
\newtheorem* {corollary*}{Corollary}
\newtheorem* {conjecture*}{Conjecture}
\newtheorem{theorem}{Theorem}[section]

\theoremstyle{definition}

\newtheorem* {example*}{Example}

\newtheorem{lemma}[theorem]{Lemma}
\theoremstyle{definition}
\newtheorem{definition}[theorem]{Definition}
\theoremstyle{definition}

\newtheorem{proposition}[theorem]{Proposition}
\newtheorem{corollary}[theorem]{Corollary}

\newtheorem*{remark*}{Remark}
\theoremstyle{definition}
\newtheorem {example}[theorem]{Example}
\theoremstyle{definition}

\theoremstyle{definition}

\theoremstyle{definition}

\def\({\left(}
\def\){\right)}

\newcommand{\cR}{\mathcal{R}}

\def\ZZ{\mathbb{Z}}

\def\barr{\begin{array}}
\def\earr{\end{array}}
\def\ba{\begin{aligned}}
\def\ea{\end{aligned}}
\def\be{\begin{equation}}
\def\ee{\end{equation}}

\def\quand{\quad\text{and}\quad}

\def\I{\mathcal{I}}

\def\cH{\mathcal H}

\def\DesR{\mathrm{Des}_R}

\def\hs{\hspace{0.5mm}}

\def\id{\mathrm{id}}

\def\ben{\begin{enumerate}}
\def\een{\end{enumerate}}

\def\bei{\begin{itemize}}
\def\eei{\end{itemize}}

\def\hs{\hspace{0.5mm}}

\newcommand{\cB}{\mathcal{B}}

\def\path{\mathsf{path}}

\def\dash{\text{ --- }}

\def\whSym
{\mathfrak{m}\textsf{Sym}}

\def\whQSym
{\mathfrak{m}\textsf{QSym}}

\def\NN{\mathbb{N}}

\def\Adstar{\mathrm{Ad}^*}
\def\iHH{\mathcal{H}_{\mathsf{inv},*}}
\def\iH{\mathcal{H}^{\mathsf{red}}_{\mathsf{inv},*}}

\def\cHinvred{\mathcal{H}^{\mathsf{red}}_{\mathsf{inv}}}

\def\iR{\mathcal{R}_{\mathsf{inv},*}}
\def\tS{\tilde S}
\def\equivA{\approx_{\mathsf{A}}}
\def\simA{\sim_{\mathsf{A}}}

\def\athreediagram{
\begin{tikzpicture}
\node[label=below:$a$,inner sep=0mm] (a) at (-2,0) {
$\circ$
};
  \node[label=below:$b$,inner sep=0mm] (b) at (-1,0) {
$\circ$
};
  \node[label=below:$c$,inner sep=0mm] (c) at (0,0) {
$\circ$
};
\draw
(a)   --  (b)
(b) -- (c)
;
\end{tikzpicture}
}

\def\bthreediagram{
\begin{tikzpicture}
\node[label=below:$a$,inner sep=0mm] (a) at (-2,0) {
$\circ$
};
  \node[label=below:$b$,inner sep=0mm] (b) at (-1,0) {
$\circ$
};
  \node[label=below:$c$,inner sep=0mm] (c) at (0,0) {
$\circ$
};
\draw[double] (a) -- node [anchor=south] {4} (b);
\draw
(b)   --  (c) 
;
\end{tikzpicture}
}

\def\hthreediagram{
\begin{tikzpicture}
\node[label=below:$a$,inner sep=0mm] (a) at (-2,0) {
$\circ$
};
  \node[label=below:$b$,inner sep=0mm] (b) at (-1,0) {
$\circ$
};
  \node[label=below:$c$,inner sep=0mm] (c) at (0,0) {
$\circ$
};
\draw[double] (a) --  node [anchor=south] {5} (b);
\draw
(b)   --  (c) 
;
\end{tikzpicture}
}

\def\dfourdiagram{
\begin{tikzpicture}
\node[label=below:$a$,inner sep=0mm] (a) at (-2,0) {
$\circ$
};
  \node[label=below:$c$,inner sep=0mm] (b) at (-1,0) {
$\circ$
};
  \node[label=above:$b$,inner sep=0mm] (c) at (-1,0.8) {
$\circ$
};
  \node[label=below:$d$,inner sep=0mm] (d) at (0,0) {
$\circ$
};
\draw
(a)   --  (b) 
(b) -- (c)
(b) -- (d)
;
\end{tikzpicture}
}

%% file: primed_invol4321.png.tex
\begin{tikzpicture}[>=latex,line join=bevel,xscale=0.45,yscale=0.45]  \pgfsetlinewidth{1bp}
\footnotesize%
\pgfsetcolor{black}
  \pgfsetcolor{gray}
  \draw [] (577.95bp,176.98bp) .. controls (570.23bp,166.24bp) and (560.46bp,152.65bp)  .. (552.71bp,141.87bp);
  \pgfsetcolor{orange}
  \draw [solid] (580.96bp,213.19bp) .. controls (574.5bp,225.01bp) and (566.12bp,240.34bp)  .. (559.68bp,252.13bp);
  \draw [solid] (530.7bp,141.98bp) .. controls (524.92bp,153.91bp) and (517.42bp,169.39bp)  .. (511.67bp,181.29bp);
  \pgfsetcolor{gray}
  \draw [solid] (525.25bp,105.27bp) .. controls (516.87bp,94.549bp) and (506.34bp,81.077bp)  .. (498.01bp,70.412bp);
  \draw [solid] (129.37bp,209.21bp) .. controls (123.76bp,223.21bp) and (116.03bp,242.49bp)  .. (110.42bp,256.47bp);
  \pgfsetcolor{orange}
  \draw [] (103.67bp,292.79bp) .. controls (104.08bp,307.91bp) and (104.66bp,329.29bp)  .. (105.06bp,344.4bp);
  \pgfsetcolor{gray}
  \draw [solid] (131.03bp,527.83bp) .. controls (139.33bp,538.35bp) and (149.73bp,551.52bp)  .. (158.05bp,562.07bp);
  \pgfsetcolor{orange}
  \draw [solid] (125.12bp,491.48bp) .. controls (130.81bp,479.42bp) and (138.26bp,463.61bp)  .. (143.96bp,451.52bp);
  \pgfsetcolor{gray}
  \draw [] (103.6bp,491.55bp) .. controls (95.846bp,480.82bp) and (86.029bp,467.23bp)  .. (78.243bp,456.45bp);
  \pgfsetcolor{orange}
  \draw [solid] (200.07bp,591.56bp) .. controls (209.91bp,595.51bp) and (220.99bp,599.96bp)  .. (230.8bp,603.9bp);
  \draw [solid] (84.74bp,173.12bp) .. controls (92.786bp,175.91bp) and (101.59bp,178.96bp)  .. (109.62bp,181.75bp);
  \pgfsetcolor{gray}
  \draw [solid] (50.28bp,181.74bp) .. controls (45.452bp,193.79bp) and (39.124bp,209.59bp)  .. (34.284bp,221.67bp);
  \pgfsetcolor{orange}
  \draw [solid] (54.073bp,252.18bp) .. controls (61.172bp,255.41bp) and (68.83bp,258.9bp)  .. (75.938bp,262.13bp);
  \draw [solid] (601.68bp,380.56bp) .. controls (594.59bp,377.33bp) and (586.93bp,373.83bp)  .. (579.83bp,370.59bp);
  \pgfsetcolor{gray}
  \draw [solid] (621.52bp,410.96bp) .. controls (616.69bp,423.01bp) and (610.37bp,438.81bp)  .. (605.53bp,450.89bp);
  \draw [solid] (545.35bp,376.26bp) .. controls (539.75bp,390.25bp) and (532.03bp,409.52bp)  .. (526.43bp,423.49bp);
  \draw [solid] (200.93bp,448.35bp) .. controls (193.98bp,446.19bp) and (186.52bp,443.87bp)  .. (179.58bp,441.72bp);
  \pgfsetcolor{orange}
  \draw [solid] (255.57bp,467.83bp) .. controls (262.05bp,470.44bp) and (268.95bp,473.22bp)  .. (275.43bp,475.82bp);
  \pgfsetcolor{gray}
  \draw [] (140.44bp,415.11bp) .. controls (133.48bp,404.64bp) and (124.76bp,391.52bp)  .. (117.78bp,381.02bp);
  \pgfsetcolor{orange}
  \draw [solid] (438.94bp,567.76bp) .. controls (442.72bp,553.97bp) and (447.87bp,535.17bp)  .. (451.62bp,521.46bp);
  \draw [solid] (302.73bp,48.256bp) .. controls (298.3bp,61.322bp) and (292.35bp,78.884bp)  .. (287.93bp,91.934bp);
  \pgfsetcolor{gray}
  \draw [] (281.47bp,35.307bp) .. controls (271.26bp,37.232bp) and (259.68bp,39.416bp)  .. (249.49bp,41.339bp);
  \draw [] (254.73bp,116.21bp) .. controls (245.59bp,118.31bp) and (235.4bp,120.65bp)  .. (226.25bp,122.76bp);
  \pgfsetcolor{orange}
  \draw [] (500.74bp,459.79bp) .. controls (492.65bp,467.78bp) and (483.18bp,477.13bp)  .. (475.09bp,485.13bp);
  \pgfsetcolor{gray}
  \draw [solid] (329.99bp,500.35bp) .. controls (335.47bp,503.07bp) and (341.22bp,505.93bp)  .. (346.7bp,508.64bp);
  \draw [] (400.87bp,515.99bp) .. controls (410.02bp,513.93bp) and (420.23bp,511.63bp)  .. (429.38bp,509.57bp);
  \pgfsetcolor{orange}
  \draw [solid] (74.863bp,419.91bp) .. controls (81.229bp,408.03bp) and (89.493bp,392.61bp)  .. (95.842bp,380.76bp);
  \draw [solid] (217.13bp,64.788bp) .. controls (213.29bp,78.463bp) and (208.07bp,97.07bp)  .. (204.24bp,110.73bp);
  \pgfsetcolor{gray}
  \draw [solid] (475.55bp,190.85bp) .. controls (468.61bp,188.66bp) and (461.16bp,186.31bp)  .. (454.23bp,184.12bp);
  \pgfsetcolor{orange}
  \draw [solid] (571.04bp,459.6bp) .. controls (563.09bp,456.83bp) and (554.4bp,453.81bp)  .. (546.45bp,451.04bp);
  \pgfsetcolor{gray}
  \draw [] (374.78bp,597.05bp) .. controls (384.98bp,595.16bp) and (396.54bp,593.02bp)  .. (406.72bp,591.14bp);
  \pgfsetcolor{orange}
  \draw [solid] (353.39bp,583.95bp) .. controls (357.7bp,570.86bp) and (363.5bp,553.26bp)  .. (367.81bp,540.18bp);
  \pgfsetcolor{gray}
  \draw [solid] (320.07bp,605.98bp) .. controls (308.97bp,607.56bp) and (296.17bp,609.37bp)  .. (285.08bp,610.94bp);
  \pgfsetcolor{orange}
  \draw [] (180.71bp,147.3bp) .. controls (172.64bp,155.31bp) and (163.2bp,164.7bp)  .. (155.13bp,172.73bp);
  \draw [] (550.41bp,288.39bp) .. controls (550.88bp,303.5bp) and (551.55bp,324.86bp)  .. (552.02bp,339.96bp);
  \pgfsetcolor{gray}
  \draw [] (537.74bp,251.96bp) .. controls (530.78bp,241.49bp) and (522.06bp,228.37bp)  .. (515.08bp,217.87bp);
  \draw [solid] (325.77bp,131.32bp) .. controls (320.27bp,128.65bp) and (314.49bp,125.84bp)  .. (308.99bp,123.17bp);
  \pgfsetcolor{orange}
  \draw [solid] (380.31bp,155.96bp) .. controls (386.75bp,158.65bp) and (393.62bp,161.53bp)  .. (400.06bp,164.22bp);
  \pgfsetcolor{gray}
  \draw [solid] (371.19bp,21.708bp) .. controls (360.06bp,23.214bp) and (347.23bp,24.95bp)  .. (336.11bp,26.453bp);
  \pgfsetcolor{orange}
  \draw [solid] (426.06bp,29.065bp) .. controls (435.9bp,33.03bp) and (446.98bp,37.494bp)  .. (456.79bp,41.447bp);
\begin{scope}
  \definecolor{strokecol}{rgb}{0.0,0.0,0.0};
  \pgfsetstrokecolor{strokecol}
  \draw (590.9bp,194.99bp) node {$1'232$};
\end{scope}
\begin{scope}
  \definecolor{strokecol}{rgb}{0.0,0.0,0.0};
  \pgfsetstrokecolor{strokecol}
  \draw (539.59bp,123.61bp) node {$1'323$};
\end{scope}
\begin{scope}
  \definecolor{strokecol}{rgb}{0.0,0.0,0.0};
  \pgfsetstrokecolor{strokecol}
  \draw (549.83bp,270.16bp) node {$1232$};
\end{scope}
\begin{scope}
  \definecolor{strokecol}{rgb}{0.0,0.0,0.0};
  \pgfsetstrokecolor{strokecol}
  \draw (502.86bp,199.49bp) node {$1323$};
\end{scope}
\begin{scope}
  \definecolor{strokecol}{rgb}{0.0,0.0,0.0};
  \pgfsetstrokecolor{strokecol}
  \draw (483.91bp,52.374bp) node {$31'23$};
\end{scope}
\begin{scope}
  \definecolor{strokecol}{rgb}{0.0,0.0,0.0};
  \pgfsetstrokecolor{strokecol}
  \draw (136.63bp,191.11bp) node {$2132'$};
\end{scope}
\begin{scope}
  \definecolor{strokecol}{rgb}{0.0,0.0,0.0};
  \pgfsetstrokecolor{strokecol}
  \draw (103.18bp,274.54bp) node {$2312'$};
\end{scope}
\begin{scope}
  \definecolor{strokecol}{rgb}{0.0,0.0,0.0};
  \pgfsetstrokecolor{strokecol}
  \draw (105.55bp,362.63bp) node {$3212'$};
\end{scope}
\begin{scope}
  \definecolor{strokecol}{rgb}{0.0,0.0,0.0};
  \pgfsetstrokecolor{strokecol}
  \draw (116.61bp,509.55bp) node {$3'1'21$};
\end{scope}
\begin{scope}
  \definecolor{strokecol}{rgb}{0.0,0.0,0.0};
  \pgfsetstrokecolor{strokecol}
  \draw (172.62bp,580.53bp) node {$1'3'21$};
\end{scope}
\begin{scope}
  \definecolor{strokecol}{rgb}{0.0,0.0,0.0};
  \pgfsetstrokecolor{strokecol}
  \draw (152.54bp,433.32bp) node {$31'21$};
\end{scope}
\begin{scope}
  \definecolor{strokecol}{rgb}{0.0,0.0,0.0};
  \pgfsetstrokecolor{strokecol}
  \draw (65.06bp,438.21bp) node {$3'212'$};
\end{scope}
\begin{scope}
  \definecolor{strokecol}{rgb}{0.0,0.0,0.0};
  \pgfsetstrokecolor{strokecol}
  \draw (257.91bp,614.78bp) node {$13'21$};
\end{scope}
\begin{scope}
  \definecolor{strokecol}{rgb}{0.0,0.0,0.0};
  \pgfsetstrokecolor{strokecol}
  \draw (57.51bp,163.68bp) node {$2'132'$};
\end{scope}
\begin{scope}
  \definecolor{strokecol}{rgb}{0.0,0.0,0.0};
  \pgfsetstrokecolor{strokecol}
  \draw (27.0bp,239.85bp) node {$2'312'$};
\end{scope}
\begin{scope}
  \definecolor{strokecol}{rgb}{0.0,0.0,0.0};
  \pgfsetstrokecolor{strokecol}
  \draw (628.75bp,392.91bp) node {$2'132$};
\end{scope}
\begin{scope}
  \definecolor{strokecol}{rgb}{0.0,0.0,0.0};
  \pgfsetstrokecolor{strokecol}
  \draw (552.6bp,358.17bp) node {$2132$};
\end{scope}
\begin{scope}
  \definecolor{strokecol}{rgb}{0.0,0.0,0.0};
  \pgfsetstrokecolor{strokecol}
  \draw (598.25bp,469.07bp) node {$2'312$};
\end{scope}
\begin{scope}
  \definecolor{strokecol}{rgb}{0.0,0.0,0.0};
  \pgfsetstrokecolor{strokecol}
  \draw (519.19bp,441.55bp) node {$2312$};
\end{scope}
\begin{scope}
  \definecolor{strokecol}{rgb}{0.0,0.0,0.0};
  \pgfsetstrokecolor{strokecol}
  \draw (228.26bp,456.84bp) node {$1'321$};
\end{scope}
\begin{scope}
  \definecolor{strokecol}{rgb}{0.0,0.0,0.0};
  \pgfsetstrokecolor{strokecol}
  \draw (302.73bp,486.81bp) node {$1321$};
\end{scope}
\begin{scope}
  \definecolor{strokecol}{rgb}{0.0,0.0,0.0};
  \pgfsetstrokecolor{strokecol}
  \draw (433.92bp,586.11bp) node {$3'212$};
\end{scope}
\begin{scope}
  \definecolor{strokecol}{rgb}{0.0,0.0,0.0};
  \pgfsetstrokecolor{strokecol}
  \draw (456.55bp,503.46bp) node {$3212$};
\end{scope}
\begin{scope}
  \definecolor{strokecol}{rgb}{0.0,0.0,0.0};
  \pgfsetstrokecolor{strokecol}
  \draw (308.88bp,30.138bp) node {$1'3'23$};
\end{scope}
\begin{scope}
  \definecolor{strokecol}{rgb}{0.0,0.0,0.0};
  \pgfsetstrokecolor{strokecol}
  \draw (281.81bp,109.98bp) node {$13'23$};
\end{scope}
\begin{scope}
  \definecolor{strokecol}{rgb}{0.0,0.0,0.0};
  \pgfsetstrokecolor{strokecol}
  \draw (222.27bp,46.474bp) node {$1'232'$};
\end{scope}
\begin{scope}
  \definecolor{strokecol}{rgb}{0.0,0.0,0.0};
  \pgfsetstrokecolor{strokecol}
  \draw (199.12bp,129.0bp) node {$1232'$};
\end{scope}
\begin{scope}
  \definecolor{strokecol}{rgb}{0.0,0.0,0.0};
  \pgfsetstrokecolor{strokecol}
  \draw (373.77bp,522.09bp) node {$3121$};
\end{scope}
\begin{scope}
  \definecolor{strokecol}{rgb}{0.0,0.0,0.0};
  \pgfsetstrokecolor{strokecol}
  \draw (427.22bp,175.58bp) node {$3123$};
\end{scope}
\begin{scope}
  \definecolor{strokecol}{rgb}{0.0,0.0,0.0};
  \pgfsetstrokecolor{strokecol}
  \draw (347.41bp,602.11bp) node {$3'121$};
\end{scope}
\begin{scope}
  \definecolor{strokecol}{rgb}{0.0,0.0,0.0};
  \pgfsetstrokecolor{strokecol}
  \draw (353.15bp,144.6bp) node {$3'123$};
\end{scope}
\begin{scope}
  \definecolor{strokecol}{rgb}{0.0,0.0,0.0};
  \pgfsetstrokecolor{strokecol}
  \draw (398.59bp,18.0bp) node {$3'1'23$};
\end{scope}
\end{tikzpicture}

%% file: hecke4321.png.tex
\begin{tikzpicture}[>=latex,line join=bevel,xscale=0.28,yscale=0.25]
  \pgfsetlinewidth{1bp}
\tiny%
\pgfsetcolor{black}
  \pgfsetcolor{gray}
  \draw [solid] (346.48bp,722.35bp) .. controls (363.75bp,725.7bp) and (386.01bp,730.02bp)  .. (403.3bp,733.38bp);
  \draw [] (292.13bp,706.87bp) .. controls (276.82bp,701.11bp) and (257.74bp,693.93bp)  .. (242.4bp,688.16bp);
  \draw [solid] (192.0bp,954.16bp) .. controls (203.04bp,971.19bp) and (219.34bp,996.35bp)  .. (230.29bp,1013.3bp);
  \draw [] (170.57bp,917.64bp) .. controls (161.55bp,900.45bp) and (148.06bp,874.73bp)  .. (139.05bp,857.56bp);
  \draw [solid] (679.02bp,808.99bp) .. controls (686.77bp,788.54bp) and (699.27bp,755.56bp)  .. (707.03bp,735.11bp);
  \pgfsetcolor{orange}
  \draw [] (650.57bp,845.59bp) .. controls (636.58bp,857.53bp) and (618.41bp,873.02bp)  .. (604.46bp,884.93bp);
  \pgfsetcolor{gray}
  \draw [] (716.0bp,698.77bp) .. controls (718.59bp,675.99bp) and (723.01bp,637.09bp)  .. (725.6bp,614.25bp);
  \draw [solid] (103.56bp,1134.0bp) .. controls (117.62bp,1142.4bp) and (134.73bp,1152.5bp)  .. (148.79bp,1160.9bp);
  \draw [] (69.045bp,1099.5bp) .. controls (59.72bp,1076.1bp) and (43.647bp,1035.8bp)  .. (34.308bp,1012.4bp);
  \draw [solid] (136.47bp,1202.6bp) .. controls (123.14bp,1183.8bp) and (102.59bp,1154.8bp)  .. (89.304bp,1136.1bp);
  \draw [solid] (176.82bp,1224.2bp) .. controls (197.88bp,1226.6bp) and (226.91bp,1229.9bp)  .. (247.96bp,1232.3bp);
  \draw [solid] (328.52bp,1039.4bp) .. controls (310.5bp,1037.7bp) and (286.99bp,1035.5bp)  .. (269.02bp,1033.8bp);
  \draw [] (486.42bp,1159.7bp) .. controls (469.61bp,1171.6bp) and (447.71bp,1187.1bp)  .. (430.87bp,1199.1bp);
  \draw [] (377.85bp,1221.1bp) .. controls (355.68bp,1224.2bp) and (324.53bp,1228.6bp)  .. (302.36bp,1231.7bp);
  \draw [solid] (419.28bp,280.46bp) .. controls (430.82bp,267.39bp) and (446.33bp,249.81bp)  .. (457.86bp,236.75bp);
  \draw [] (407.08bp,316.69bp) .. controls (411.49bp,337.71bp) and (418.73bp,372.19bp)  .. (423.15bp,393.19bp);
  \draw [] (754.94bp,589.01bp) .. controls (772.21bp,584.56bp) and (794.47bp,578.82bp)  .. (811.76bp,574.36bp);
  \draw [solid] (546.32bp,1040.0bp) .. controls (538.72bp,1062.5bp) and (525.86bp,1100.6bp)  .. (518.23bp,1123.2bp);
  \draw [] (557.12bp,1003.8bp) .. controls (562.88bp,981.48bp) and (572.63bp,943.71bp)  .. (578.42bp,921.33bp);
  \draw [] (563.29bp,884.88bp) .. controls (550.41bp,873.02bp) and (533.69bp,857.63bp)  .. (520.85bp,845.8bp);
  \draw [] (28.311bp,975.94bp) .. controls (30.167bp,950.31bp) and (33.54bp,903.74bp)  .. (35.402bp,878.05bp);
  \draw [solid] (51.724bp,841.44bp) .. controls (66.996bp,822.74bp) and (90.531bp,793.93bp)  .. (105.76bp,775.29bp);
  \draw [solid] (872.28bp,486.82bp) .. controls (865.0bp,504.49bp) and (853.98bp,531.26bp)  .. (846.67bp,549.0bp);
  \draw [solid] (307.53bp,495.82bp) .. controls (293.76bp,510.63bp) and (274.51bp,531.34bp)  .. (260.83bp,546.06bp);
  \pgfsetcolor{orange}
  \draw [] (239.45bp,582.15bp) .. controls (234.06bp,603.33bp) and (225.17bp,638.27bp)  .. (219.75bp,659.57bp);
  \pgfsetcolor{gray}
  \draw [solid] (573.73bp,457.13bp) .. controls (588.24bp,466.02bp) and (606.06bp,476.93bp)  .. (620.58bp,485.83bp);
  \pgfsetcolor{orange}
  \draw [] (519.46bp,433.89bp) .. controls (499.88bp,429.1bp) and (473.55bp,422.65bp)  .. (453.99bp,417.86bp);
  \draw [] (663.17bp,520.51bp) .. controls (677.23bp,536.97bp) and (697.91bp,561.19bp)  .. (712.04bp,577.73bp);
  \pgfsetcolor{gray}
  \draw [solid] (248.06bp,1219.5bp) .. controls (234.07bp,1211.3bp) and (217.04bp,1201.2bp)  .. (203.05bp,1193.0bp);
  \pgfsetcolor{orange}
  \draw [] (127.61bp,821.25bp) .. controls (126.13bp,807.58bp) and (124.1bp,788.99bp)  .. (122.61bp,775.34bp);
  \pgfsetcolor{gray}
  \draw [solid] (379.24bp,773.68bp) .. controls (366.88bp,762.02bp) and (350.86bp,746.91bp)  .. (338.48bp,735.23bp);
  \draw [solid] (425.42bp,801.21bp) .. controls (440.41bp,806.45bp) and (459.0bp,812.95bp)  .. (474.02bp,818.2bp);
  \draw [solid] (1121.8bp,567.97bp) .. controls (1108.0bp,570.24bp) and (1091.3bp,572.99bp)  .. (1077.5bp,575.26bp);
  \draw [] (1023.2bp,582.24bp) .. controls (1008.1bp,583.64bp) and (989.42bp,585.38bp)  .. (974.31bp,586.78bp);
  \draw [] (438.2bp,1043.2bp) .. controls (421.35bp,1042.9bp) and (399.76bp,1042.6bp)  .. (382.91bp,1042.4bp);
  \pgfsetcolor{orange}
  \draw [] (492.37bp,1036.8bp) .. controls (502.87bp,1034.2bp) and (514.87bp,1031.2bp)  .. (525.37bp,1028.6bp);
  \pgfsetcolor{gray}
  \draw [] (193.18bp,696.26bp) .. controls (177.92bp,709.05bp) and (157.65bp,726.05bp)  .. (142.42bp,738.82bp);
  \draw [] (399.79bp,428.81bp) .. controls (384.98bp,438.38bp) and (366.71bp,450.2bp)  .. (351.89bp,459.79bp);
  \pgfsetcolor{orange}
  \draw [] (919.76bp,583.76bp) .. controls (903.42bp,580.42bp) and (882.69bp,576.19bp)  .. (866.34bp,572.85bp);
  \pgfsetcolor{gray}
  \draw [solid] (252.53bp,959.87bp) .. controls (238.46bp,955.18bp) and (221.34bp,949.48bp)  .. (207.28bp,944.8bp);
  \draw [solid] (298.56bp,987.0bp) .. controls (310.25bp,998.24bp) and (325.24bp,1012.6bp)  .. (336.94bp,1023.9bp);
  \draw [solid] (486.93bp,809.69bp) .. controls (474.77bp,794.35bp) and (457.39bp,772.41bp)  .. (445.17bp,756.98bp);
  \draw [solid] (294.68bp,96.09bp) .. controls (289.81bp,78.895bp) and (282.53bp,53.179bp)  .. (277.67bp,36.01bp);
  \pgfsetcolor{orange}
  \draw [] (350.02bp,229.5bp) .. controls (361.55bp,244.46bp) and (377.85bp,265.61bp)  .. (389.37bp,280.56bp);
  \pgfsetcolor{gray}
  \draw [] (329.27bp,193.15bp) .. controls (322.78bp,175.76bp) and (313.08bp,149.76bp)  .. (306.6bp,132.39bp);
\begin{scope}
  \definecolor{strokecol}{rgb}{0.0,0.0,0.0};
  \pgfsetstrokecolor{strokecol}
  \draw (319.22bp,717.05bp) node {$231213$};
\end{scope}
\begin{scope}
  \definecolor{strokecol}{rgb}{0.0,0.0,0.0};
  \pgfsetstrokecolor{strokecol}
  \draw (430.67bp,738.69bp) node {$231231$};
\end{scope}
\begin{scope}
  \definecolor{strokecol}{rgb}{0.0,0.0,0.0};
  \pgfsetstrokecolor{strokecol}
  \draw (215.08bp,677.89bp) node {$232123$};
\end{scope}
\begin{scope}
  \definecolor{strokecol}{rgb}{0.0,0.0,0.0};
  \pgfsetstrokecolor{strokecol}
  \draw (180.06bp,935.74bp) node {$31213$};
\end{scope}
\begin{scope}
  \definecolor{strokecol}{rgb}{0.0,0.0,0.0};
  \pgfsetstrokecolor{strokecol}
  \draw (241.98bp,1031.3bp) node {$31231$};
\end{scope}
\begin{scope}
  \definecolor{strokecol}{rgb}{0.0,0.0,0.0};
  \pgfsetstrokecolor{strokecol}
  \draw (129.61bp,839.55bp) node {$32123$};
\end{scope}
\begin{scope}
  \definecolor{strokecol}{rgb}{0.0,0.0,0.0};
  \pgfsetstrokecolor{strokecol}
  \draw (672.11bp,827.21bp) node {$21321$};
\end{scope}
\begin{scope}
  \definecolor{strokecol}{rgb}{0.0,0.0,0.0};
  \pgfsetstrokecolor{strokecol}
  \draw (713.95bp,716.86bp) node {$23121$};
\end{scope}
\begin{scope}
  \definecolor{strokecol}{rgb}{0.0,0.0,0.0};
  \pgfsetstrokecolor{strokecol}
  \draw (583.11bp,903.14bp) node {$212321$};
\end{scope}
\begin{scope}
  \definecolor{strokecol}{rgb}{0.0,0.0,0.0};
  \pgfsetstrokecolor{strokecol}
  \draw (727.67bp,596.04bp) node {$23212$};
\end{scope}
\begin{scope}
  \definecolor{strokecol}{rgb}{0.0,0.0,0.0};
  \pgfsetstrokecolor{strokecol}
  \draw (76.323bp,1117.8bp) node {$312132$};
\end{scope}
\begin{scope}
  \definecolor{strokecol}{rgb}{0.0,0.0,0.0};
  \pgfsetstrokecolor{strokecol}
  \draw (175.99bp,1177.1bp) node {$312312$};
\end{scope}
\begin{scope}
  \definecolor{strokecol}{rgb}{0.0,0.0,0.0};
  \pgfsetstrokecolor{strokecol}
  \draw (27.0bp,994.04bp) node {$321232$};
\end{scope}
\begin{scope}
  \definecolor{strokecol}{rgb}{0.0,0.0,0.0};
  \pgfsetstrokecolor{strokecol}
  \draw (149.57bp,1221.1bp) node {$132132$};
\end{scope}
\begin{scope}
  \definecolor{strokecol}{rgb}{0.0,0.0,0.0};
  \pgfsetstrokecolor{strokecol}
  \draw (275.16bp,1235.5bp) node {$132312$};
\end{scope}
\begin{scope}
  \definecolor{strokecol}{rgb}{0.0,0.0,0.0};
  \pgfsetstrokecolor{strokecol}
  \draw (355.78bp,1042.0bp) node {$13231$};
\end{scope}
\begin{scope}
  \definecolor{strokecol}{rgb}{0.0,0.0,0.0};
  \pgfsetstrokecolor{strokecol}
  \draw (512.03bp,1141.6bp) node {$123121$};
\end{scope}
\begin{scope}
  \definecolor{strokecol}{rgb}{0.0,0.0,0.0};
  \pgfsetstrokecolor{strokecol}
  \draw (405.09bp,1217.3bp) node {$123212$};
\end{scope}
\begin{scope}
  \definecolor{strokecol}{rgb}{0.0,0.0,0.0};
  \pgfsetstrokecolor{strokecol}
  \draw (403.27bp,298.59bp) node {$12132$};
\end{scope}
\begin{scope}
  \definecolor{strokecol}{rgb}{0.0,0.0,0.0};
  \pgfsetstrokecolor{strokecol}
  \draw (473.8bp,218.69bp) node {$12312$};
\end{scope}
\begin{scope}
  \definecolor{strokecol}{rgb}{0.0,0.0,0.0};
  \pgfsetstrokecolor{strokecol}
  \draw (426.94bp,411.24bp) node {$21232$};
\end{scope}
\begin{scope}
  \definecolor{strokecol}{rgb}{0.0,0.0,0.0};
  \pgfsetstrokecolor{strokecol}
  \draw (839.14bp,567.3bp) node {$32312$};
\end{scope}
\begin{scope}
  \definecolor{strokecol}{rgb}{0.0,0.0,0.0};
  \pgfsetstrokecolor{strokecol}
  \draw (552.47bp,1021.8bp) node {$121321$};
\end{scope}
\begin{scope}
  \definecolor{strokecol}{rgb}{0.0,0.0,0.0};
  \pgfsetstrokecolor{strokecol}
  \draw (501.2bp,827.71bp) node {$213231$};
\end{scope}
\begin{scope}
  \definecolor{strokecol}{rgb}{0.0,0.0,0.0};
  \pgfsetstrokecolor{strokecol}
  \draw (36.723bp,859.8bp) node {$321323$};
\end{scope}
\begin{scope}
  \definecolor{strokecol}{rgb}{0.0,0.0,0.0};
  \pgfsetstrokecolor{strokecol}
  \draw (120.63bp,757.09bp) node {$323123$};
\end{scope}
\begin{scope}
  \definecolor{strokecol}{rgb}{0.0,0.0,0.0};
  \pgfsetstrokecolor{strokecol}
  \draw (879.72bp,468.74bp) node {$32132$};
\end{scope}
\begin{scope}
  \definecolor{strokecol}{rgb}{0.0,0.0,0.0};
  \pgfsetstrokecolor{strokecol}
  \draw (324.61bp,477.43bp) node {$21323$};
\end{scope}
\begin{scope}
  \definecolor{strokecol}{rgb}{0.0,0.0,0.0};
  \pgfsetstrokecolor{strokecol}
  \draw (244.04bp,564.13bp) node {$23123$};
\end{scope}
\begin{scope}
  \definecolor{strokecol}{rgb}{0.0,0.0,0.0};
  \pgfsetstrokecolor{strokecol}
  \draw (546.62bp,440.54bp) node {$2132$};
\end{scope}
\begin{scope}
  \definecolor{strokecol}{rgb}{0.0,0.0,0.0};
  \pgfsetstrokecolor{strokecol}
  \draw (647.76bp,502.47bp) node {$2312$};
\end{scope}
\begin{scope}
  \definecolor{strokecol}{rgb}{0.0,0.0,0.0};
  \pgfsetstrokecolor{strokecol}
  \draw (398.41bp,791.77bp) node {$213213$};
\end{scope}
\begin{scope}
  \definecolor{strokecol}{rgb}{0.0,0.0,0.0};
  \pgfsetstrokecolor{strokecol}
  \draw (1149.0bp,563.5bp) node {$1321$};
\end{scope}
\begin{scope}
  \definecolor{strokecol}{rgb}{0.0,0.0,0.0};
  \pgfsetstrokecolor{strokecol}
  \draw (1050.3bp,579.73bp) node {$3121$};
\end{scope}
\begin{scope}
  \definecolor{strokecol}{rgb}{0.0,0.0,0.0};
  \pgfsetstrokecolor{strokecol}
  \draw (946.98bp,589.31bp) node {$3212$};
\end{scope}
\begin{scope}
  \definecolor{strokecol}{rgb}{0.0,0.0,0.0};
  \pgfsetstrokecolor{strokecol}
  \draw (465.26bp,1043.6bp) node {$12321$};
\end{scope}
\begin{scope}
  \definecolor{strokecol}{rgb}{0.0,0.0,0.0};
  \pgfsetstrokecolor{strokecol}
  \draw (279.77bp,968.94bp) node {$13213$};
\end{scope}
\begin{scope}
  \definecolor{strokecol}{rgb}{0.0,0.0,0.0};
  \pgfsetstrokecolor{strokecol}
  \draw (299.8bp,114.18bp) node {$1323$};
\end{scope}
\begin{scope}
  \definecolor{strokecol}{rgb}{0.0,0.0,0.0};
  \pgfsetstrokecolor{strokecol}
  \draw (272.57bp,18.0bp) node {$3123$};
\end{scope}
\begin{scope}
  \definecolor{strokecol}{rgb}{0.0,0.0,0.0};
  \pgfsetstrokecolor{strokecol}
  \draw (336.1bp,211.44bp) node {$1232$};
\end{scope}
\end{tikzpicture}

%% file: invol54321.png.tex
\begin{tikzpicture}[>=latex,line join=bevel,scale=0.18]
  \pgfsetlinewidth{0.5bp}
\small%
\pgfsetcolor{black}
  \pgfsetcolor{orange}
  \draw [] (578.61bp,1042.9bp) .. controls (565.92bp,1044.6bp) and (523.61bp,1050.4bp)  .. (510.7bp,1052.1bp);
  \pgfsetcolor{gray}
  \draw [solid] (503.14bp,1052.1bp) .. controls (489.8bp,1050.1bp) and (445.03bp,1043.4bp)  .. (432.03bp,1041.4bp);
  \pgfsetcolor{orange}
  \draw [] (424.58bp,1040.6bp) .. controls (410.85bp,1039.5bp) and (364.82bp,1035.7bp)  .. (351.46bp,1034.6bp);
  \pgfsetcolor{gray}
  \draw [solid] (426.62bp,1037.8bp) .. controls (419.61bp,1026.1bp) and (394.69bp,984.38bp)  .. (387.45bp,972.26bp);
  \draw [solid] (345.98bp,1031.0bp) .. controls (339.04bp,1018.4bp) and (314.21bp,973.29bp)  .. (307.36bp,960.84bp);
  \pgfsetcolor{orange}
  \draw [] (757.87bp,132.71bp) .. controls (744.09bp,131.27bp) and (697.86bp,126.45bp)  .. (684.44bp,125.05bp);
  \pgfsetcolor{gray}
  \draw [solid] (761.96bp,129.13bp) .. controls (762.55bp,116.15bp) and (764.42bp,75.164bp)  .. (764.98bp,62.673bp);
  \draw [solid] (676.94bp,123.97bp) .. controls (663.48bp,121.53bp) and (618.31bp,113.33bp)  .. (605.19bp,110.96bp);
  \draw [solid] (680.28bp,120.85bp) .. controls (678.74bp,108.53bp) and (673.9bp,69.899bp)  .. (672.36bp,57.554bp);
  \draw [] (900.22bp,341.18bp) .. controls (913.92bp,344.58bp) and (959.9bp,355.99bp)  .. (973.24bp,359.3bp);
  \pgfsetcolor{orange}
  \draw [] (895.24bp,336.42bp) .. controls (891.41bp,323.02bp) and (878.54bp,278.1bp)  .. (874.81bp,265.06bp);
  \draw [] (976.23bp,356.54bp) .. controls (973.65bp,342.48bp) and (964.43bp,292.18bp)  .. (961.89bp,278.31bp);
  \pgfsetcolor{gray}
  \draw [solid] (977.88bp,364.09bp) .. controls (981.31bp,377.77bp) and (992.81bp,423.67bp)  .. (996.15bp,437.0bp);
  \draw [solid] (720.24bp,193.21bp) .. controls (713.19bp,180.96bp) and (689.52bp,139.87bp)  .. (682.65bp,127.94bp);
  \draw [solid] (291.03bp,244.03bp) .. controls (295.43bp,230.86bp) and (311.06bp,184.05bp)  .. (315.6bp,170.44bp);
  \draw [] (293.43bp,245.66bp) .. controls (305.99bp,239.16bp) and (348.12bp,217.34bp)  .. (360.35bp,211.0bp);
  \draw [] (320.5bp,165.5bp) .. controls (334.5bp,160.85bp) and (384.6bp,144.24bp)  .. (398.41bp,139.67bp);
  \draw [solid] (302.51bp,955.46bp) .. controls (290.71bp,947.08bp) and (248.46bp,917.11bp)  .. (236.81bp,908.85bp);
  \draw [solid] (626.45bp,1079.6bp) .. controls (618.22bp,1072.7bp) and (594.18bp,1052.5bp)  .. (585.46bp,1045.1bp);
  \draw [] (632.68bp,1080.1bp) .. controls (645.82bp,1072.6bp) and (692.85bp,1046.0bp)  .. (705.81bp,1038.6bp);
  \pgfsetcolor{orange}
  \draw [] (625.71bp,1083.2bp) .. controls (610.96bp,1088.0bp) and (554.44bp,1106.5bp)  .. (539.69bp,1111.4bp);
  \pgfsetcolor{gray}
  \draw [] (711.68bp,1033.9bp) .. controls (721.37bp,1023.0bp) and (756.07bp,983.84bp)  .. (765.63bp,973.06bp);
  \draw [solid] (713.13bp,1036.3bp) .. controls (727.18bp,1034.9bp) and (774.32bp,1030.1bp)  .. (788.0bp,1028.7bp);
  \draw [solid] (218.07bp,289.53bp) .. controls (230.41bp,282.3bp) and (274.57bp,256.46bp)  .. (286.75bp,249.34bp);
  \draw [solid] (215.92bp,287.94bp) .. controls (219.98bp,274.75bp) and (234.42bp,227.87bp)  .. (238.62bp,214.25bp);
  \draw [solid] (243.09bp,208.63bp) .. controls (255.76bp,201.42bp) and (301.12bp,175.64bp)  .. (313.63bp,168.53bp);
  \draw [] (877.54bp,262.05bp) .. controls (891.93bp,264.23bp) and (943.39bp,272.04bp)  .. (957.59bp,274.19bp);
  \draw [solid] (870.71bp,259.27bp) .. controls (858.94bp,250.81bp) and (816.83bp,220.55bp)  .. (805.22bp,212.21bp);
  \draw [solid] (770.38bp,967.47bp) .. controls (779.07bp,956.8bp) and (809.95bp,918.85bp)  .. (818.93bp,907.82bp);
  \draw [solid] (771.67bp,969.91bp) .. controls (785.26bp,968.48bp) and (833.55bp,963.4bp)  .. (847.59bp,961.92bp);
  \draw [solid] (825.08bp,903.73bp) .. controls (838.1bp,899.98bp) and (881.76bp,887.39bp)  .. (894.43bp,883.74bp);
  \draw [solid] (1051.7bp,219.88bp) .. controls (1038.5bp,215.07bp) and (994.16bp,198.92bp)  .. (981.29bp,194.23bp);
  \draw [solid] (1057.8bp,218.47bp) .. controls (1066.5bp,207.92bp) and (1097.7bp,170.43bp)  .. (1106.7bp,159.53bp);
  \draw [solid] (1053.5bp,217.86bp) .. controls (1046.1bp,204.91bp) and (1019.5bp,158.57bp)  .. (1012.2bp,145.8bp);
  \draw [solid] (980.35bp,189.96bp) .. controls (990.31bp,178.51bp) and (1025.9bp,137.56bp)  .. (1035.8bp,126.26bp);
  \draw [solid] (975.81bp,189.6bp) .. controls (968.36bp,176.79bp) and (941.72bp,130.97bp)  .. (934.38bp,118.33bp);
  \draw [solid] (1043.2bp,299.79bp) .. controls (1045.4bp,285.75bp) and (1052.8bp,238.67bp)  .. (1054.9bp,225.0bp);
  \draw [solid] (1039.1bp,302.52bp) .. controls (1025.7bp,297.75bp) and (977.83bp,280.66bp)  .. (964.63bp,275.95bp);
  \draw [solid] (962.03bp,270.8bp) .. controls (964.84bp,256.88bp) and (974.27bp,210.21bp)  .. (977.0bp,196.66bp);
  \draw [solid] (853.44bp,958.13bp) .. controls (861.08bp,945.17bp) and (888.43bp,898.8bp)  .. (895.97bp,886.02bp);
  \draw [solid] (1058.2bp,481.59bp) .. controls (1047.4bp,474.34bp) and (1011.3bp,450.18bp)  .. (1000.3bp,442.81bp);
  \draw [solid] (1059.8bp,487.2bp) .. controls (1053.8bp,500.75bp) and (1032.6bp,549.23bp)  .. (1026.8bp,562.59bp);
  \draw [] (1062.0bp,487.46bp) .. controls (1064.8bp,501.97bp) and (1074.8bp,553.92bp)  .. (1077.6bp,568.24bp);
  \draw [solid] (996.07bp,444.26bp) .. controls (992.59bp,456.85bp) and (981.0bp,498.82bp)  .. (977.47bp,511.63bp);
  \pgfsetcolor{orange}
  \draw [] (598.24bp,108.59bp) .. controls (585.38bp,102.03bp) and (539.38bp,78.565bp)  .. (526.69bp,72.095bp);
  \pgfsetcolor{gray}
  \draw [solid] (599.91bp,107.07bp) .. controls (595.26bp,98.144bp) and (582.31bp,73.278bp)  .. (577.62bp,64.266bp);
  \draw [solid] (520.48bp,72.859bp) .. controls (512.96bp,79.024bp) and (493.19bp,95.249bp)  .. (485.47bp,101.59bp);
  \draw [solid] (117.78bp,444.68bp) .. controls (112.4bp,458.22bp) and (93.135bp,506.7bp)  .. (87.824bp,520.06bp);
  \pgfsetcolor{orange}
  \draw [] (121.05bp,437.8bp) .. controls (128.17bp,425.04bp) and (153.63bp,379.38bp)  .. (160.66bp,366.79bp);
  \pgfsetcolor{gray}
  \draw [] (84.754bp,526.88bp) .. controls (78.226bp,539.82bp) and (54.87bp,586.15bp)  .. (48.43bp,598.92bp);
  \draw [solid] (87.483bp,527.28bp) .. controls (91.097bp,540.65bp) and (103.22bp,585.47bp)  .. (106.73bp,598.48bp);
  \draw [solid] (900.15bp,879.59bp) .. controls (908.74bp,867.55bp) and (939.48bp,824.45bp)  .. (947.95bp,812.57bp);
  \draw [solid] (668.23bp,53.996bp) .. controls (653.03bp,55.141bp) and (594.75bp,59.531bp)  .. (579.54bp,60.677bp);
  \pgfsetcolor{orange}
  \draw [] (572.34bp,62.593bp) .. controls (557.51bp,69.44bp) and (500.7bp,95.681bp)  .. (485.87bp,102.53bp);
  \pgfsetcolor{gray}
  \draw [solid] (78.811bp,734.58bp) .. controls (78.954bp,724.49bp) and (79.353bp,696.35bp)  .. (79.498bp,686.16bp);
  \draw [solid] (80.885bp,741.4bp) .. controls (89.023bp,753.49bp) and (118.14bp,796.72bp)  .. (126.17bp,808.64bp);
  \draw [solid] (78.928bp,686.19bp) .. controls (76.332bp,701.95bp) and (66.382bp,762.35bp)  .. (63.784bp,778.11bp);
  \pgfsetcolor{orange}
  \draw [] (761.14bp,58.796bp) .. controls (745.8bp,57.924bp) and (690.9bp,54.803bp)  .. (675.77bp,53.942bp);
  \pgfsetcolor{gray}
  \draw [] (769.03bp,59.221bp) .. controls (783.89bp,59.973bp) and (837.06bp,62.664bp)  .. (851.72bp,63.406bp);
  \pgfsetcolor{orange}
  \draw [] (381.61bp,968.39bp) .. controls (368.02bp,966.47bp) and (322.46bp,960.03bp)  .. (309.23bp,958.16bp);
  \pgfsetcolor{gray}
  \draw [solid] (129.03bp,979.1bp) .. controls (116.0bp,974.53bp) and (72.294bp,959.18bp)  .. (59.605bp,954.72bp);
  \draw [solid] (133.33bp,976.84bp) .. controls (135.63bp,963.33bp) and (143.78bp,915.31bp)  .. (146.15bp,901.36bp);
  \draw [solid] (56.667bp,949.92bp) .. controls (58.726bp,936.32bp) and (66.047bp,888.01bp)  .. (68.175bp,873.97bp);
  \draw [solid] (177.72bp,1054.3bp) .. controls (169.99bp,1041.6bp) and (142.32bp,996.14bp)  .. (134.69bp,983.62bp);
  \draw [solid] (176.02bp,1056.3bp) .. controls (162.88bp,1052.0bp) and (118.82bp,1037.2bp)  .. (106.03bp,1033.0bp);
  \draw [solid] (124.73bp,810.1bp) .. controls (113.44bp,804.92bp) and (78.0bp,788.69bp)  .. (66.677bp,783.5bp);
  \draw [] (132.09bp,812.48bp) .. controls (145.71bp,815.24bp) and (191.38bp,824.5bp)  .. (204.64bp,827.19bp);
  \draw [solid] (1105.8bp,154.94bp) .. controls (1093.8bp,149.35bp) and (1054.0bp,130.71bp)  .. (1041.9bp,125.03bp);
  \draw [solid] (1107.5bp,153.05bp) .. controls (1100.9bp,139.67bp) and (1077.4bp,91.807bp)  .. (1071.0bp,78.61bp);
  \draw [solid] (1036.2bp,120.2bp) .. controls (1028.1bp,108.09bp) and (999.06bp,64.765bp)  .. (991.06bp,52.819bp);
  \draw [solid] (1006.6bp,141.19bp) .. controls (993.34bp,136.53bp) and (948.92bp,120.87bp)  .. (936.03bp,116.33bp);
  \draw [solid] (1012.9bp,139.62bp) .. controls (1022.6bp,128.55bp) and (1057.3bp,88.946bp)  .. (1066.9bp,78.025bp);
  \draw [solid] (934.92bp,112.28bp) .. controls (944.22bp,101.53bp) and (977.48bp,63.077bp)  .. (986.65bp,52.474bp);
  \draw [solid] (929.18bp,112.87bp) .. controls (916.52bp,104.4bp) and (871.2bp,74.099bp)  .. (858.7bp,65.743bp);
  \pgfsetcolor{orange}
  \draw [] (146.0bp,893.84bp) .. controls (142.95bp,879.72bp) and (132.02bp,829.21bp)  .. (129.01bp,815.28bp);
  \pgfsetcolor{gray}
  \draw [] (150.54bp,897.92bp) .. controls (164.85bp,899.43bp) and (216.06bp,904.85bp)  .. (230.18bp,906.34bp);
  \draw [solid] (143.03bp,896.2bp) .. controls (129.76bp,891.54bp) and (85.232bp,875.9bp)  .. (72.305bp,871.36bp);
  \pgfsetcolor{orange}
  \draw [] (232.58bp,902.92bp) .. controls (228.24bp,889.52bp) and (213.67bp,844.56bp)  .. (209.44bp,831.51bp);
  \pgfsetcolor{gray}
  \draw [solid] (439.49bp,1111.5bp) .. controls (437.54bp,1099.0bp) and (431.03bp,1057.3bp)  .. (429.05bp,1044.6bp);
  \pgfsetcolor{orange}
  \draw [] (436.07bp,1114.9bp) .. controls (420.8bp,1114.1bp) and (366.19bp,1111.2bp)  .. (351.13bp,1110.4bp);
  \pgfsetcolor{gray}
  \draw [solid] (347.28bp,1106.5bp) .. controls (347.37bp,1093.7bp) and (347.68bp,1051.1bp)  .. (347.77bp,1038.1bp);
  \draw [] (1025.5bp,570.03bp) .. controls (1026.3bp,584.13bp) and (1029.0bp,631.43bp)  .. (1029.8bp,645.16bp);
  \draw [solid] (1022.6bp,563.31bp) .. controls (1014.2bp,554.59bp) and (988.06bp,527.44bp)  .. (979.31bp,518.36bp);
  \draw [solid] (1028.5bp,652.47bp) .. controls (1022.6bp,666.05bp) and (1001.5bp,714.63bp)  .. (995.7bp,728.03bp);
  \pgfsetcolor{orange}
  \draw [] (68.519bp,866.32bp) .. controls (67.598bp,851.81bp) and (64.303bp,799.89bp)  .. (63.395bp,785.57bp);
  \pgfsetcolor{gray}
  \draw [solid] (127.33bp,1121.6bp) .. controls (136.34bp,1110.6bp) and (168.57bp,1071.2bp)  .. (177.46bp,1060.4bp);
  \draw [solid] (122.79bp,1121.4bp) .. controls (114.4bp,1109.4bp) and (84.369bp,1066.6bp)  .. (76.089bp,1054.8bp);
  \draw [solid] (121.12bp,1123.4bp) .. controls (107.52bp,1119.3bp) and (61.908bp,1105.7bp)  .. (48.666bp,1101.8bp);
  \draw [solid] (76.487bp,1048.7bp) .. controls (86.153bp,1036.9bp) and (120.74bp,994.95bp)  .. (130.28bp,983.37bp);
  \draw [solid] (70.568bp,1050.2bp) .. controls (58.709bp,1044.8bp) and (19.187bp,1026.8bp)  .. (7.1254bp,1021.3bp);
  \draw [solid] (1065.9bp,74.125bp) .. controls (1052.8bp,69.973bp) and (1006.2bp,55.217bp)  .. (992.72bp,50.927bp);
  \pgfsetcolor{orange}
  \draw [] (952.02bp,806.17bp) .. controls (959.27bp,793.33bp) and (985.21bp,747.39bp)  .. (992.37bp,734.72bp);
  \pgfsetcolor{gray}
  \draw [solid] (1050.8bp,389.33bp) .. controls (1037.6bp,384.12bp) and (993.28bp,366.67bp)  .. (980.43bp,361.6bp);
  \pgfsetcolor{orange}
  \draw [] (1054.0bp,387.06bp) .. controls (1052.0bp,372.75bp) and (1045.1bp,321.52bp)  .. (1043.1bp,307.4bp);
  \pgfsetcolor{gray}
  \draw [solid] (1054.8bp,394.79bp) .. controls (1055.9bp,410.07bp) and (1059.9bp,464.72bp)  .. (1061.0bp,479.79bp);
  \draw [solid] (47.503bp,1097.7bp) .. controls (56.959bp,1086.4bp) and (90.794bp,1045.8bp)  .. (100.12bp,1034.7bp);
  \draw [solid] (43.252bp,1097.2bp) .. controls (36.437bp,1083.9bp) and (12.051bp,1036.2bp)  .. (5.3276bp,1023.1bp);
  \draw [solid] (100.53bp,1028.4bp) .. controls (92.895bp,1015.5bp) and (65.588bp,969.47bp)  .. (58.058bp,956.76bp);
  \draw [solid] (21.785bp,672.74bp) .. controls (31.577bp,684.0bp) and (66.616bp,724.29bp)  .. (76.277bp,735.39bp);
  \draw [solid] (20.708bp,666.17bp) .. controls (25.491bp,654.43bp) and (40.493bp,617.62bp)  .. (45.287bp,605.86bp);
  \draw [solid] (48.184bp,605.65bp) .. controls (53.575bp,618.84bp) and (72.865bp,666.05bp)  .. (78.184bp,679.06bp);
  \draw [solid] (365.55bp,205.85bp) .. controls (372.06bp,193.81bp) and (393.89bp,153.43bp)  .. (400.22bp,141.7bp);
  \draw [solid] (1076.2bp,575.23bp) .. controls (1068.3bp,587.9bp) and (1039.9bp,633.22bp)  .. (1032.1bp,645.72bp);
  \draw [solid] (800.5bp,206.76bp) .. controls (793.85bp,194.1bp) and (770.03bp,148.81bp)  .. (763.46bp,136.33bp);
  \pgfsetcolor{orange}
  \draw [] (798.38bp,209.42bp) .. controls (784.77bp,207.14bp) and (739.13bp,199.51bp)  .. (725.88bp,197.29bp);
  \pgfsetcolor{gray}
  \draw [solid] (985.74bp,47.767bp) .. controls (973.23bp,40.176bp) and (928.48bp,13.014bp)  .. (916.14bp,5.5243bp);
  \draw [solid] (858.26bp,60.703bp) .. controls (868.04bp,50.497bp) and (900.84bp,16.264bp)  .. (910.36bp,6.3264bp);
  \draw [solid] (534.41bp,1109.0bp) .. controls (529.2bp,1098.3bp) and (513.77bp,1066.7bp)  .. (508.62bp,1056.1bp);
  \draw [solid] (532.5bp,1112.6bp) .. controls (517.28bp,1113.0bp) and (458.94bp,1114.6bp)  .. (443.71bp,1115.0bp);
  \draw [] (794.32bp,1025.5bp) .. controls (804.14bp,1014.5bp) and (839.27bp,975.15bp)  .. (848.95bp,964.3bp);
  \draw [solid] (164.71bp,360.46bp) .. controls (173.33bp,348.59bp) and (204.16bp,306.13bp)  .. (212.66bp,294.42bp);
  \draw [] (405.41bp,137.01bp) .. controls (418.63bp,131.37bp) and (465.92bp,111.17bp)  .. (478.96bp,105.61bp);
  \draw [] (106.49bp,605.5bp) .. controls (101.9bp,618.6bp) and (85.593bp,665.16bp)  .. (80.852bp,678.69bp);
  \draw [solid] (253.92bp,1105.9bp) .. controls (241.17bp,1097.6bp) and (195.55bp,1067.9bp)  .. (182.97bp,1059.7bp);
  \draw [] (261.12bp,1108.2bp) .. controls (275.93bp,1108.5bp) and (328.9bp,1109.8bp)  .. (343.5bp,1110.1bp);
  \draw [solid] (5.8569bp,1016.9bp) .. controls (14.418bp,1006.1bp) and (44.85bp,967.71bp)  .. (53.696bp,956.56bp);
  \draw [solid] (198.23bp,1167.6bp) .. controls (185.64bp,1160.2bp) and (140.59bp,1133.7bp)  .. (128.17bp,1126.4bp);
  \draw [solid] (204.21bp,1166.6bp) .. controls (213.69bp,1156.2bp) and (245.49bp,1121.1bp)  .. (254.72bp,1110.9bp);
\begin{scope}
  \definecolor{strokecol}{rgb}{0.0,0.0,0.0};
  \pgfsetstrokecolor{strokecol}
  \definecolor{fillcol}{rgb}{0.0,0.0,0.0};
  \pgfsetfillcolor{fillcol}
  \filldraw [opacity=1] (582.24bp,1042.4bp) ellipse (3.6bp and 3.6bp);
\end{scope}
\begin{scope}
  \definecolor{strokecol}{rgb}{0.0,0.0,0.0};
  \pgfsetstrokecolor{strokecol}
  \definecolor{fillcol}{rgb}{0.0,0.0,0.0};
  \pgfsetfillcolor{fillcol}
  \filldraw [opacity=1] (506.93bp,1052.62bp) ellipse (3.6bp and 3.6bp);
\end{scope}
\begin{scope}
  \definecolor{strokecol}{rgb}{0.0,0.0,0.0};
  \pgfsetstrokecolor{strokecol}
  \definecolor{fillcol}{rgb}{0.0,0.0,0.0};
  \pgfsetfillcolor{fillcol}
  \filldraw [opacity=1] (428.47bp,1040.9bp) ellipse (3.6bp and 3.6bp);
\end{scope}
\begin{scope}
  \definecolor{strokecol}{rgb}{0.0,0.0,0.0};
  \pgfsetstrokecolor{strokecol}
  \definecolor{fillcol}{rgb}{0.0,0.0,0.0};
  \pgfsetfillcolor{fillcol}
  \filldraw [opacity=1] (347.79bp,1034.33bp) ellipse (3.6bp and 3.6bp);
\end{scope}
\begin{scope}
  \definecolor{strokecol}{rgb}{0.0,0.0,0.0};
  \pgfsetstrokecolor{strokecol}
  \definecolor{fillcol}{rgb}{0.0,0.0,0.0};
  \pgfsetfillcolor{fillcol}
  \filldraw [opacity=1] (385.46bp,968.94bp) ellipse (3.6bp and 3.6bp);
\end{scope}
\begin{scope}
  \definecolor{strokecol}{rgb}{0.0,0.0,0.0};
  \pgfsetstrokecolor{strokecol}
  \definecolor{fillcol}{rgb}{0.0,0.0,0.0};
  \pgfsetfillcolor{fillcol}
  \filldraw [opacity=1] (305.6bp,957.65bp) ellipse (3.6bp and 3.6bp);
\end{scope}
\begin{scope}
  \definecolor{strokecol}{rgb}{0.0,0.0,0.0};
  \pgfsetstrokecolor{strokecol}
  \definecolor{fillcol}{rgb}{0.0,0.0,0.0};
  \pgfsetfillcolor{fillcol}
  \filldraw [opacity=1] (761.78bp,133.12bp) ellipse (3.6bp and 3.6bp);
\end{scope}
\begin{scope}
  \definecolor{strokecol}{rgb}{0.0,0.0,0.0};
  \pgfsetstrokecolor{strokecol}
  \definecolor{fillcol}{rgb}{0.0,0.0,0.0};
  \pgfsetfillcolor{fillcol}
  \filldraw [opacity=1] (680.76bp,124.66bp) ellipse (3.6bp and 3.6bp);
\end{scope}
\begin{scope}
  \definecolor{strokecol}{rgb}{0.0,0.0,0.0};
  \pgfsetstrokecolor{strokecol}
  \definecolor{fillcol}{rgb}{0.0,0.0,0.0};
  \pgfsetfillcolor{fillcol}
  \filldraw [opacity=1] (765.15bp,59.02bp) ellipse (3.6bp and 3.6bp);
\end{scope}
\begin{scope}
  \definecolor{strokecol}{rgb}{0.0,0.0,0.0};
  \pgfsetstrokecolor{strokecol}
  \definecolor{fillcol}{rgb}{0.0,0.0,0.0};
  \pgfsetfillcolor{fillcol}
  \filldraw [opacity=1] (601.6bp,110.3bp) ellipse (3.6bp and 3.6bp);
\end{scope}
\begin{scope}
  \definecolor{strokecol}{rgb}{0.0,0.0,0.0};
  \pgfsetstrokecolor{strokecol}
  \definecolor{fillcol}{rgb}{0.0,0.0,0.0};
  \pgfsetfillcolor{fillcol}
  \filldraw [opacity=1] (671.88bp,53.72bp) ellipse (3.6bp and 3.6bp);
\end{scope}
\begin{scope}
  \definecolor{strokecol}{rgb}{0.0,0.0,0.0};
  \pgfsetstrokecolor{strokecol}
  \definecolor{fillcol}{rgb}{0.0,0.0,0.0};
  \pgfsetfillcolor{fillcol}
  \filldraw [opacity=1] (896.33bp,340.21bp) ellipse (3.6bp and 3.6bp);
\end{scope}
\begin{scope}
  \definecolor{strokecol}{rgb}{0.0,0.0,0.0};
  \pgfsetstrokecolor{strokecol}
  \definecolor{fillcol}{rgb}{0.0,0.0,0.0};
  \pgfsetfillcolor{fillcol}
  \filldraw [opacity=1] (976.9bp,360.21bp) ellipse (3.6bp and 3.6bp);
\end{scope}
\begin{scope}
  \definecolor{strokecol}{rgb}{0.0,0.0,0.0};
  \pgfsetstrokecolor{strokecol}
  \definecolor{fillcol}{rgb}{0.0,0.0,0.0};
  \pgfsetfillcolor{fillcol}
  \filldraw [opacity=1] (873.79bp,261.48bp) ellipse (3.6bp and 3.6bp);
\end{scope}
\begin{scope}
  \definecolor{strokecol}{rgb}{0.0,0.0,0.0};
  \pgfsetstrokecolor{strokecol}
  \definecolor{fillcol}{rgb}{0.0,0.0,0.0};
  \pgfsetfillcolor{fillcol}
  \filldraw [opacity=1] (961.23bp,274.74bp) ellipse (3.6bp and 3.6bp);
\end{scope}
\begin{scope}
  \definecolor{strokecol}{rgb}{0.0,0.0,0.0};
  \pgfsetstrokecolor{strokecol}
  \definecolor{fillcol}{rgb}{0.0,0.0,0.0};
  \pgfsetfillcolor{fillcol}
  \filldraw [opacity=1] (997.06bp,440.65bp) ellipse (3.6bp and 3.6bp);
\end{scope}
\begin{scope}
  \definecolor{strokecol}{rgb}{0.0,0.0,0.0};
  \pgfsetstrokecolor{strokecol}
  \definecolor{fillcol}{rgb}{0.0,0.0,0.0};
  \pgfsetfillcolor{fillcol}
  \filldraw [opacity=1] (722.24bp,196.69bp) ellipse (3.6bp and 3.6bp);
\end{scope}
\begin{scope}
  \definecolor{strokecol}{rgb}{0.0,0.0,0.0};
  \pgfsetstrokecolor{strokecol}
  \definecolor{fillcol}{rgb}{0.0,0.0,0.0};
  \pgfsetfillcolor{fillcol}
  \filldraw [opacity=1] (289.87bp,247.51bp) ellipse (3.6bp and 3.6bp);
\end{scope}
\begin{scope}
  \definecolor{strokecol}{rgb}{0.0,0.0,0.0};
  \pgfsetstrokecolor{strokecol}
  \definecolor{fillcol}{rgb}{0.0,0.0,0.0};
  \pgfsetfillcolor{fillcol}
  \filldraw [opacity=1] (316.84bp,166.71bp) ellipse (3.6bp and 3.6bp);
\end{scope}
\begin{scope}
  \definecolor{strokecol}{rgb}{0.0,0.0,0.0};
  \pgfsetstrokecolor{strokecol}
  \definecolor{fillcol}{rgb}{0.0,0.0,0.0};
  \pgfsetfillcolor{fillcol}
  \filldraw [opacity=1] (363.7bp,209.26bp) ellipse (3.6bp and 3.6bp);
\end{scope}
\begin{scope}
  \definecolor{strokecol}{rgb}{0.0,0.0,0.0};
  \pgfsetstrokecolor{strokecol}
  \definecolor{fillcol}{rgb}{0.0,0.0,0.0};
  \pgfsetfillcolor{fillcol}
  \filldraw [opacity=1] (401.96bp,138.49bp) ellipse (3.6bp and 3.6bp);
\end{scope}
\begin{scope}
  \definecolor{strokecol}{rgb}{0.0,0.0,0.0};
  \pgfsetstrokecolor{strokecol}
  \definecolor{fillcol}{rgb}{0.0,0.0,0.0};
  \pgfsetfillcolor{fillcol}
  \filldraw [opacity=1] (233.81bp,906.73bp) ellipse (3.6bp and 3.6bp);
\end{scope}
\begin{scope}
  \definecolor{strokecol}{rgb}{0.0,0.0,0.0};
  \pgfsetstrokecolor{strokecol}
  \definecolor{fillcol}{rgb}{0.0,0.0,0.0};
  \pgfsetfillcolor{fillcol}
  \filldraw [opacity=1] (629.24bp,1082.0bp) ellipse (3.6bp and 3.6bp);
\end{scope}
\begin{scope}
  \definecolor{strokecol}{rgb}{0.0,0.0,0.0};
  \pgfsetstrokecolor{strokecol}
  \definecolor{fillcol}{rgb}{0.0,0.0,0.0};
  \pgfsetfillcolor{fillcol}
  \filldraw [opacity=1] (709.14bp,1036.74bp) ellipse (3.6bp and 3.6bp);
\end{scope}
\begin{scope}
  \definecolor{strokecol}{rgb}{0.0,0.0,0.0};
  \pgfsetstrokecolor{strokecol}
  \definecolor{fillcol}{rgb}{0.0,0.0,0.0};
  \pgfsetfillcolor{fillcol}
  \filldraw [opacity=1] (536.15bp,1112.53bp) ellipse (3.6bp and 3.6bp);
\end{scope}
\begin{scope}
  \definecolor{strokecol}{rgb}{0.0,0.0,0.0};
  \pgfsetstrokecolor{strokecol}
  \definecolor{fillcol}{rgb}{0.0,0.0,0.0};
  \pgfsetfillcolor{fillcol}
  \filldraw [opacity=1] (768.09bp,970.28bp) ellipse (3.6bp and 3.6bp);
\end{scope}
\begin{scope}
  \definecolor{strokecol}{rgb}{0.0,0.0,0.0};
  \pgfsetstrokecolor{strokecol}
  \definecolor{fillcol}{rgb}{0.0,0.0,0.0};
  \pgfsetfillcolor{fillcol}
  \filldraw [opacity=1] (791.75bp,1028.34bp) ellipse (3.6bp and 3.6bp);
\end{scope}
\begin{scope}
  \definecolor{strokecol}{rgb}{0.0,0.0,0.0};
  \pgfsetstrokecolor{strokecol}
  \definecolor{fillcol}{rgb}{0.0,0.0,0.0};
  \pgfsetfillcolor{fillcol}
  \filldraw [opacity=1] (214.85bp,291.41bp) ellipse (3.6bp and 3.6bp);
\end{scope}
\begin{scope}
  \definecolor{strokecol}{rgb}{0.0,0.0,0.0};
  \pgfsetstrokecolor{strokecol}
  \definecolor{fillcol}{rgb}{0.0,0.0,0.0};
  \pgfsetfillcolor{fillcol}
  \filldraw [opacity=1] (239.78bp,210.51bp) ellipse (3.6bp and 3.6bp);
\end{scope}
\begin{scope}
  \definecolor{strokecol}{rgb}{0.0,0.0,0.0};
  \pgfsetstrokecolor{strokecol}
  \definecolor{fillcol}{rgb}{0.0,0.0,0.0};
  \pgfsetfillcolor{fillcol}
  \filldraw [opacity=1] (802.24bp,210.07bp) ellipse (3.6bp and 3.6bp);
\end{scope}
\begin{scope}
  \definecolor{strokecol}{rgb}{0.0,0.0,0.0};
  \pgfsetstrokecolor{strokecol}
  \definecolor{fillcol}{rgb}{0.0,0.0,0.0};
  \pgfsetfillcolor{fillcol}
  \filldraw [opacity=1] (821.39bp,904.79bp) ellipse (3.6bp and 3.6bp);
\end{scope}
\begin{scope}
  \definecolor{strokecol}{rgb}{0.0,0.0,0.0};
  \pgfsetstrokecolor{strokecol}
  \definecolor{fillcol}{rgb}{0.0,0.0,0.0};
  \pgfsetfillcolor{fillcol}
  \filldraw [opacity=1] (851.44bp,961.52bp) ellipse (3.6bp and 3.6bp);
\end{scope}
\begin{scope}
  \definecolor{strokecol}{rgb}{0.0,0.0,0.0};
  \pgfsetstrokecolor{strokecol}
  \definecolor{fillcol}{rgb}{0.0,0.0,0.0};
  \pgfsetfillcolor{fillcol}
  \filldraw [opacity=1] (897.91bp,882.74bp) ellipse (3.6bp and 3.6bp);
\end{scope}
\begin{scope}
  \definecolor{strokecol}{rgb}{0.0,0.0,0.0};
  \pgfsetstrokecolor{strokecol}
  \definecolor{fillcol}{rgb}{0.0,0.0,0.0};
  \pgfsetfillcolor{fillcol}
  \filldraw [opacity=1] (1055.49bp,221.25bp) ellipse (3.6bp and 3.6bp);
\end{scope}
\begin{scope}
  \definecolor{strokecol}{rgb}{0.0,0.0,0.0};
  \pgfsetstrokecolor{strokecol}
  \definecolor{fillcol}{rgb}{0.0,0.0,0.0};
  \pgfsetfillcolor{fillcol}
  \filldraw [opacity=1] (977.75bp,192.95bp) ellipse (3.6bp and 3.6bp);
\end{scope}
\begin{scope}
  \definecolor{strokecol}{rgb}{0.0,0.0,0.0};
  \pgfsetstrokecolor{strokecol}
  \definecolor{fillcol}{rgb}{0.0,0.0,0.0};
  \pgfsetfillcolor{fillcol}
  \filldraw [opacity=1] (1109.19bp,156.54bp) ellipse (3.6bp and 3.6bp);
\end{scope}
\begin{scope}
  \definecolor{strokecol}{rgb}{0.0,0.0,0.0};
  \pgfsetstrokecolor{strokecol}
  \definecolor{fillcol}{rgb}{0.0,0.0,0.0};
  \pgfsetfillcolor{fillcol}
  \filldraw [opacity=1] (1010.33bp,142.51bp) ellipse (3.6bp and 3.6bp);
\end{scope}
\begin{scope}
  \definecolor{strokecol}{rgb}{0.0,0.0,0.0};
  \pgfsetstrokecolor{strokecol}
  \definecolor{fillcol}{rgb}{0.0,0.0,0.0};
  \pgfsetfillcolor{fillcol}
  \filldraw [opacity=1] (1038.3bp,123.36bp) ellipse (3.6bp and 3.6bp);
\end{scope}
\begin{scope}
  \definecolor{strokecol}{rgb}{0.0,0.0,0.0};
  \pgfsetstrokecolor{strokecol}
  \definecolor{fillcol}{rgb}{0.0,0.0,0.0};
  \pgfsetfillcolor{fillcol}
  \filldraw [opacity=1] (932.49bp,115.09bp) ellipse (3.6bp and 3.6bp);
\end{scope}
\begin{scope}
  \definecolor{strokecol}{rgb}{0.0,0.0,0.0};
  \pgfsetstrokecolor{strokecol}
  \definecolor{fillcol}{rgb}{0.0,0.0,0.0};
  \pgfsetfillcolor{fillcol}
  \filldraw [opacity=1] (1042.63bp,303.77bp) ellipse (3.6bp and 3.6bp);
\end{scope}
\begin{scope}
  \definecolor{strokecol}{rgb}{0.0,0.0,0.0};
  \pgfsetstrokecolor{strokecol}
  \definecolor{fillcol}{rgb}{0.0,0.0,0.0};
  \pgfsetfillcolor{fillcol}
  \filldraw [opacity=1] (1061.3bp,483.66bp) ellipse (3.6bp and 3.6bp);
\end{scope}
\begin{scope}
  \definecolor{strokecol}{rgb}{0.0,0.0,0.0};
  \pgfsetstrokecolor{strokecol}
  \definecolor{fillcol}{rgb}{0.0,0.0,0.0};
  \pgfsetfillcolor{fillcol}
  \filldraw [opacity=1] (1025.27bp,566.03bp) ellipse (3.6bp and 3.6bp);
\end{scope}
\begin{scope}
  \definecolor{strokecol}{rgb}{0.0,0.0,0.0};
  \pgfsetstrokecolor{strokecol}
  \definecolor{fillcol}{rgb}{0.0,0.0,0.0};
  \pgfsetfillcolor{fillcol}
  \filldraw [opacity=1] (1078.27bp,571.92bp) ellipse (3.6bp and 3.6bp);
\end{scope}
\begin{scope}
  \definecolor{strokecol}{rgb}{0.0,0.0,0.0};
  \pgfsetstrokecolor{strokecol}
  \definecolor{fillcol}{rgb}{0.0,0.0,0.0};
  \pgfsetfillcolor{fillcol}
  \filldraw [opacity=1] (976.43bp,515.37bp) ellipse (3.6bp and 3.6bp);
\end{scope}
\begin{scope}
  \definecolor{strokecol}{rgb}{0.0,0.0,0.0};
  \pgfsetstrokecolor{strokecol}
  \definecolor{fillcol}{rgb}{0.0,0.0,0.0};
  \pgfsetfillcolor{fillcol}
  \filldraw [opacity=1] (523.43bp,70.43bp) ellipse (3.6bp and 3.6bp);
\end{scope}
\begin{scope}
  \definecolor{strokecol}{rgb}{0.0,0.0,0.0};
  \pgfsetstrokecolor{strokecol}
  \definecolor{fillcol}{rgb}{0.0,0.0,0.0};
  \pgfsetfillcolor{fillcol}
  \filldraw [opacity=1] (575.89bp,60.95bp) ellipse (3.6bp and 3.6bp);
\end{scope}
\begin{scope}
  \definecolor{strokecol}{rgb}{0.0,0.0,0.0};
  \pgfsetstrokecolor{strokecol}
  \definecolor{fillcol}{rgb}{0.0,0.0,0.0};
  \pgfsetfillcolor{fillcol}
  \filldraw [opacity=1] (482.31bp,104.17bp) ellipse (3.6bp and 3.6bp);
\end{scope}
\begin{scope}
  \definecolor{strokecol}{rgb}{0.0,0.0,0.0};
  \pgfsetstrokecolor{strokecol}
  \definecolor{fillcol}{rgb}{0.0,0.0,0.0};
  \pgfsetfillcolor{fillcol}
  \filldraw [opacity=1] (119.19bp,441.14bp) ellipse (3.6bp and 3.6bp);
\end{scope}
\begin{scope}
  \definecolor{strokecol}{rgb}{0.0,0.0,0.0};
  \pgfsetstrokecolor{strokecol}
  \definecolor{fillcol}{rgb}{0.0,0.0,0.0};
  \pgfsetfillcolor{fillcol}
  \filldraw [opacity=1] (86.46bp,523.5bp) ellipse (3.6bp and 3.6bp);
\end{scope}
\begin{scope}
  \definecolor{strokecol}{rgb}{0.0,0.0,0.0};
  \pgfsetstrokecolor{strokecol}
  \definecolor{fillcol}{rgb}{0.0,0.0,0.0};
  \pgfsetfillcolor{fillcol}
  \filldraw [opacity=1] (162.46bp,363.56bp) ellipse (3.6bp and 3.6bp);
\end{scope}
\begin{scope}
  \definecolor{strokecol}{rgb}{0.0,0.0,0.0};
  \pgfsetstrokecolor{strokecol}
  \definecolor{fillcol}{rgb}{0.0,0.0,0.0};
  \pgfsetfillcolor{fillcol}
  \filldraw [opacity=1] (46.78bp,602.2bp) ellipse (3.6bp and 3.6bp);
\end{scope}
\begin{scope}
  \definecolor{strokecol}{rgb}{0.0,0.0,0.0};
  \pgfsetstrokecolor{strokecol}
  \definecolor{fillcol}{rgb}{0.0,0.0,0.0};
  \pgfsetfillcolor{fillcol}
  \filldraw [opacity=1] (107.7bp,602.05bp) ellipse (3.6bp and 3.6bp);
\end{scope}
\begin{scope}
  \definecolor{strokecol}{rgb}{0.0,0.0,0.0};
  \pgfsetstrokecolor{strokecol}
  \definecolor{fillcol}{rgb}{0.0,0.0,0.0};
  \pgfsetfillcolor{fillcol}
  \filldraw [opacity=1] (950.13bp,809.52bp) ellipse (3.6bp and 3.6bp);
\end{scope}
\begin{scope}
  \definecolor{strokecol}{rgb}{0.0,0.0,0.0};
  \pgfsetstrokecolor{strokecol}
  \definecolor{fillcol}{rgb}{0.0,0.0,0.0};
  \pgfsetfillcolor{fillcol}
  \filldraw [opacity=1] (78.76bp,738.25bp) ellipse (3.6bp and 3.6bp);
\end{scope}
\begin{scope}
  \definecolor{strokecol}{rgb}{0.0,0.0,0.0};
  \pgfsetstrokecolor{strokecol}
  \definecolor{fillcol}{rgb}{0.0,0.0,0.0};
  \pgfsetfillcolor{fillcol}
  \filldraw [opacity=1] (79.55bp,682.41bp) ellipse (3.6bp and 3.6bp);
\end{scope}
\begin{scope}
  \definecolor{strokecol}{rgb}{0.0,0.0,0.0};
  \pgfsetstrokecolor{strokecol}
  \definecolor{fillcol}{rgb}{0.0,0.0,0.0};
  \pgfsetfillcolor{fillcol}
  \filldraw [opacity=1] (128.23bp,811.7bp) ellipse (3.6bp and 3.6bp);
\end{scope}
\begin{scope}
  \definecolor{strokecol}{rgb}{0.0,0.0,0.0};
  \pgfsetstrokecolor{strokecol}
  \definecolor{fillcol}{rgb}{0.0,0.0,0.0};
  \pgfsetfillcolor{fillcol}
  \filldraw [opacity=1] (63.16bp,781.89bp) ellipse (3.6bp and 3.6bp);
\end{scope}
\begin{scope}
  \definecolor{strokecol}{rgb}{0.0,0.0,0.0};
  \pgfsetstrokecolor{strokecol}
  \definecolor{fillcol}{rgb}{0.0,0.0,0.0};
  \pgfsetfillcolor{fillcol}
  \filldraw [opacity=1] (855.49bp,63.6bp) ellipse (3.6bp and 3.6bp);
\end{scope}
\begin{scope}
  \definecolor{strokecol}{rgb}{0.0,0.0,0.0};
  \pgfsetstrokecolor{strokecol}
  \definecolor{fillcol}{rgb}{0.0,0.0,0.0};
  \pgfsetfillcolor{fillcol}
  \filldraw [opacity=1] (132.73bp,980.4bp) ellipse (3.6bp and 3.6bp);
\end{scope}
\begin{scope}
  \definecolor{strokecol}{rgb}{0.0,0.0,0.0};
  \pgfsetstrokecolor{strokecol}
  \definecolor{fillcol}{rgb}{0.0,0.0,0.0};
  \pgfsetfillcolor{fillcol}
  \filldraw [opacity=1] (56.12bp,953.5bp) ellipse (3.6bp and 3.6bp);
\end{scope}
\begin{scope}
  \definecolor{strokecol}{rgb}{0.0,0.0,0.0};
  \pgfsetstrokecolor{strokecol}
  \definecolor{fillcol}{rgb}{0.0,0.0,0.0};
  \pgfsetfillcolor{fillcol}
  \filldraw [opacity=1] (146.8bp,897.53bp) ellipse (3.6bp and 3.6bp);
\end{scope}
\begin{scope}
  \definecolor{strokecol}{rgb}{0.0,0.0,0.0};
  \pgfsetstrokecolor{strokecol}
  \definecolor{fillcol}{rgb}{0.0,0.0,0.0};
  \pgfsetfillcolor{fillcol}
  \filldraw [opacity=1] (68.76bp,870.11bp) ellipse (3.6bp and 3.6bp);
\end{scope}
\begin{scope}
  \definecolor{strokecol}{rgb}{0.0,0.0,0.0};
  \pgfsetstrokecolor{strokecol}
  \definecolor{fillcol}{rgb}{0.0,0.0,0.0};
  \pgfsetfillcolor{fillcol}
  \filldraw [opacity=1] (179.74bp,1057.59bp) ellipse (3.6bp and 3.6bp);
\end{scope}
\begin{scope}
  \definecolor{strokecol}{rgb}{0.0,0.0,0.0};
  \pgfsetstrokecolor{strokecol}
  \definecolor{fillcol}{rgb}{0.0,0.0,0.0};
  \pgfsetfillcolor{fillcol}
  \filldraw [opacity=1] (102.52bp,1031.79bp) ellipse (3.6bp and 3.6bp);
\end{scope}
\begin{scope}
  \definecolor{strokecol}{rgb}{0.0,0.0,0.0};
  \pgfsetstrokecolor{strokecol}
  \definecolor{fillcol}{rgb}{0.0,0.0,0.0};
  \pgfsetfillcolor{fillcol}
  \filldraw [opacity=1] (208.28bp,827.93bp) ellipse (3.6bp and 3.6bp);
\end{scope}
\begin{scope}
  \definecolor{strokecol}{rgb}{0.0,0.0,0.0};
  \pgfsetstrokecolor{strokecol}
  \definecolor{fillcol}{rgb}{0.0,0.0,0.0};
  \pgfsetfillcolor{fillcol}
  \filldraw [opacity=1] (1069.31bp,75.22bp) ellipse (3.6bp and 3.6bp);
\end{scope}
\begin{scope}
  \definecolor{strokecol}{rgb}{0.0,0.0,0.0};
  \pgfsetstrokecolor{strokecol}
  \definecolor{fillcol}{rgb}{0.0,0.0,0.0};
  \pgfsetfillcolor{fillcol}
  \filldraw [opacity=1] (989.01bp,49.75bp) ellipse (3.6bp and 3.6bp);
\end{scope}
\begin{scope}
  \definecolor{strokecol}{rgb}{0.0,0.0,0.0};
  \pgfsetstrokecolor{strokecol}
  \definecolor{fillcol}{rgb}{0.0,0.0,0.0};
  \pgfsetfillcolor{fillcol}
  \filldraw [opacity=1] (440.05bp,1115.11bp) ellipse (3.6bp and 3.6bp);
\end{scope}
\begin{scope}
  \definecolor{strokecol}{rgb}{0.0,0.0,0.0};
  \pgfsetstrokecolor{strokecol}
  \definecolor{fillcol}{rgb}{0.0,0.0,0.0};
  \pgfsetfillcolor{fillcol}
  \filldraw [opacity=1] (347.26bp,1110.19bp) ellipse (3.6bp and 3.6bp);
\end{scope}
\begin{scope}
  \definecolor{strokecol}{rgb}{0.0,0.0,0.0};
  \pgfsetstrokecolor{strokecol}
  \definecolor{fillcol}{rgb}{0.0,0.0,0.0};
  \pgfsetfillcolor{fillcol}
  \filldraw [opacity=1] (1030.04bp,648.93bp) ellipse (3.6bp and 3.6bp);
\end{scope}
\begin{scope}
  \definecolor{strokecol}{rgb}{0.0,0.0,0.0};
  \pgfsetstrokecolor{strokecol}
  \definecolor{fillcol}{rgb}{0.0,0.0,0.0};
  \pgfsetfillcolor{fillcol}
  \filldraw [opacity=1] (994.21bp,731.47bp) ellipse (3.6bp and 3.6bp);
\end{scope}
\begin{scope}
  \definecolor{strokecol}{rgb}{0.0,0.0,0.0};
  \pgfsetstrokecolor{strokecol}
  \definecolor{fillcol}{rgb}{0.0,0.0,0.0};
  \pgfsetfillcolor{fillcol}
  \filldraw [opacity=1] (124.98bp,1124.5bp) ellipse (3.6bp and 3.6bp);
\end{scope}
\begin{scope}
  \definecolor{strokecol}{rgb}{0.0,0.0,0.0};
  \pgfsetstrokecolor{strokecol}
  \definecolor{fillcol}{rgb}{0.0,0.0,0.0};
  \pgfsetfillcolor{fillcol}
  \filldraw [opacity=1] (73.96bp,1051.73bp) ellipse (3.6bp and 3.6bp);
\end{scope}
\begin{scope}
  \definecolor{strokecol}{rgb}{0.0,0.0,0.0};
  \pgfsetstrokecolor{strokecol}
  \definecolor{fillcol}{rgb}{0.0,0.0,0.0};
  \pgfsetfillcolor{fillcol}
  \filldraw [opacity=1] (45.03bp,1100.67bp) ellipse (3.6bp and 3.6bp);
\end{scope}
\begin{scope}
  \definecolor{strokecol}{rgb}{0.0,0.0,0.0};
  \pgfsetstrokecolor{strokecol}
  \definecolor{fillcol}{rgb}{0.0,0.0,0.0};
  \pgfsetfillcolor{fillcol}
  \filldraw [opacity=1] (3.6bp,1019.72bp) ellipse (3.6bp and 3.6bp);
\end{scope}
\begin{scope}
  \definecolor{strokecol}{rgb}{0.0,0.0,0.0};
  \pgfsetstrokecolor{strokecol}
  \definecolor{fillcol}{rgb}{0.0,0.0,0.0};
  \pgfsetfillcolor{fillcol}
  \filldraw [opacity=1] (1054.5bp,390.8bp) ellipse (3.6bp and 3.6bp);
\end{scope}
\begin{scope}
  \definecolor{strokecol}{rgb}{0.0,0.0,0.0};
  \pgfsetstrokecolor{strokecol}
  \definecolor{fillcol}{rgb}{0.0,0.0,0.0};
  \pgfsetfillcolor{fillcol}
  \filldraw [opacity=1] (19.23bp,669.8bp) ellipse (3.6bp and 3.6bp);
\end{scope}
\begin{scope}
  \definecolor{strokecol}{rgb}{0.0,0.0,0.0};
  \pgfsetstrokecolor{strokecol}
  \definecolor{fillcol}{rgb}{0.0,0.0,0.0};
  \pgfsetfillcolor{fillcol}
  \filldraw [opacity=1] (912.97bp,3.6bp) ellipse (3.6bp and 3.6bp);
\end{scope}
\begin{scope}
  \definecolor{strokecol}{rgb}{0.0,0.0,0.0};
  \pgfsetstrokecolor{strokecol}
  \definecolor{fillcol}{rgb}{0.0,0.0,0.0};
  \pgfsetfillcolor{fillcol}
  \filldraw [opacity=1] (257.26bp,1108.11bp) ellipse (3.6bp and 3.6bp);
\end{scope}
\begin{scope}
  \definecolor{strokecol}{rgb}{0.0,0.0,0.0};
  \pgfsetstrokecolor{strokecol}
  \definecolor{fillcol}{rgb}{0.0,0.0,0.0};
  \pgfsetfillcolor{fillcol}
  \filldraw [opacity=1] (201.52bp,1169.57bp) ellipse (3.6bp and 3.6bp);
\end{scope}
\end{tikzpicture}

%% file: twisted54321.png.tex
\begin{tikzpicture}[>=latex,line join=bevel,scale=0.14]
  \pgfsetlinewidth{0.5bp}
\small%
\pgfsetcolor{black}
  \pgfsetcolor{orange}
  \draw [] (69.495bp,1450.9bp) .. controls (58.767bp,1459.9bp) and (22.779bp,1490.2bp)  .. (12.332bp,1499.0bp);
  \pgfsetcolor{gray}
  \draw [solid] (73.761bp,1451.9bp) .. controls (78.444bp,1465.4bp) and (95.201bp,1513.8bp)  .. (99.821bp,1527.2bp);
  \draw [solid] (10.796bp,1504.9bp) .. controls (15.889bp,1518.4bp) and (34.112bp,1566.6bp)  .. (39.137bp,1579.9bp);
  \draw [] (183.24bp,1570.0bp) .. controls (196.35bp,1565.4bp) and (242.96bp,1549.1bp)  .. (256.5bp,1544.4bp);
  \pgfsetcolor{orange}
  \draw [] (176.93bp,1573.9bp) .. controls (166.86bp,1583.3bp) and (133.08bp,1614.8bp)  .. (123.27bp,1624.0bp);
  \pgfsetcolor{gray}
  \draw [solid] (263.96bp,1543.5bp) .. controls (278.26bp,1545.2bp) and (329.43bp,1551.0bp)  .. (343.54bp,1552.7bp);
  \pgfsetcolor{orange}
  \draw [solid] (1408.6bp,381.59bp) .. controls (1408.2bp,395.57bp) and (1406.9bp,445.25bp)  .. (1406.5bp,459.69bp);
  \pgfsetcolor{gray}
  \draw [solid] (1408.7bp,374.18bp) .. controls (1408.6bp,359.92bp) and (1408.3bp,308.89bp)  .. (1408.2bp,294.83bp);
  \draw [solid] (118.6bp,1629.7bp) .. controls (111.67bp,1641.1bp) and (88.594bp,1679.1bp)  .. (81.551bp,1690.7bp);
  \draw [solid] (1568.3bp,453.88bp) .. controls (1556.5bp,459.44bp) and (1519.2bp,476.97bp)  .. (1507.8bp,482.31bp);
  \draw [solid] (1567.9bp,452.08bp) .. controls (1552.5bp,451.7bp) and (1497.4bp,450.35bp)  .. (1482.1bp,449.98bp);
  \draw [solid] (1500.8bp,483.1bp) .. controls (1485.2bp,479.9bp) and (1425.7bp,467.62bp)  .. (1410.1bp,464.42bp);
  \draw [solid] (1685.8bp,175.17bp) .. controls (1680.2bp,161.73bp) and (1660.3bp,113.63bp)  .. (1654.8bp,100.37bp);
  \draw [solid] (1650.0bp,95.474bp) .. controls (1636.9bp,89.78bp) and (1590.0bp,69.404bp)  .. (1577.1bp,63.786bp);
  \draw [] (955.54bp,932.1bp) .. controls (945.72bp,941.03bp) and (913.02bp,970.8bp)  .. (903.04bp,979.88bp);
  \pgfsetcolor{orange}
  \draw [solid] (958.21bp,925.62bp) .. controls (957.71bp,910.58bp) and (955.9bp,856.8bp)  .. (955.4bp,841.97bp);
  \pgfsetcolor{gray}
  \draw [solid] (962.24bp,929.95bp) .. controls (977.15bp,931.51bp) and (1030.5bp,937.11bp)  .. (1045.2bp,938.65bp);
  \draw [solid] (903.94bp,983.05bp) .. controls (918.59bp,985.0bp) and (970.99bp,992.0bp)  .. (985.44bp,993.93bp);
  \pgfsetcolor{orange}
  \draw [solid] (899.94bp,978.69bp) .. controls (899.26bp,963.95bp) and (896.84bp,911.24bp)  .. (896.17bp,896.71bp);
  \pgfsetcolor{gray}
  \draw [solid] (899.24bp,986.13bp) .. controls (895.93bp,999.78bp) and (884.16bp,1048.3bp)  .. (880.73bp,1062.4bp);
  \draw [solid] (43.872bp,1585.2bp) .. controls (57.056bp,1592.3bp) and (104.23bp,1617.7bp)  .. (117.24bp,1624.7bp);
  \draw [solid] (38.651bp,1586.8bp) .. controls (32.387bp,1599.2bp) and (11.374bp,1640.7bp)  .. (5.2736bp,1652.7bp);
  \draw [solid] (6.861bp,1657.7bp) .. controls (19.231bp,1663.9bp) and (63.202bp,1686.0bp)  .. (75.984bp,1692.4bp);
  \draw [solid] (976.51bp,762.28bp) .. controls (972.71bp,775.84bp) and (959.98bp,821.33bp)  .. (956.29bp,834.54bp);
  \draw [solid] (981.2bp,758.71bp) .. controls (994.91bp,759.78bp) and (1043.6bp,763.57bp)  .. (1057.8bp,764.67bp);
  \draw [solid] (959.09bp,838.64bp) .. controls (973.72bp,840.48bp) and (1026.1bp,847.05bp)  .. (1040.5bp,848.87bp);
  \draw [] (952.41bp,840.81bp) .. controls (942.33bp,850.13bp) and (908.51bp,881.4bp)  .. (898.7bp,890.48bp);
  \draw [solid] (320.4bp,1305.6bp) .. controls (334.93bp,1307.2bp) and (386.95bp,1313.1bp)  .. (401.29bp,1314.8bp);
  \draw [solid] (314.58bp,1308.4bp) .. controls (307.45bp,1319.9bp) and (283.53bp,1358.6bp)  .. (276.58bp,1369.8bp);
  \draw [solid] (316.91bp,1309.0bp) .. controls (318.08bp,1323.9bp) and (322.3bp,1377.1bp)  .. (323.47bp,1391.7bp);
  \draw [solid] (403.02bp,1318.6bp) .. controls (396.66bp,1329.7bp) and (376.57bp,1364.8bp)  .. (370.45bp,1375.5bp);
  \draw [solid] (405.35bp,1319.2bp) .. controls (406.77bp,1334.7bp) and (411.87bp,1390.2bp)  .. (413.27bp,1405.5bp);
  \draw [solid] (1477.5bp,355.63bp) .. controls (1465.3bp,359.57bp) and (1424.7bp,372.72bp)  .. (1412.3bp,376.73bp);
  \pgfsetcolor{orange}
  \draw [solid] (1480.9bp,358.12bp) .. controls (1480.4bp,373.23bp) and (1478.8bp,431.14bp)  .. (1478.3bp,446.25bp);
  \pgfsetcolor{gray}
  \draw [solid] (1481.0bp,350.86bp) .. controls (1481.1bp,335.69bp) and (1481.7bp,277.56bp)  .. (1481.8bp,262.39bp);
  \draw [solid] (286.14bp,1465.3bp) .. controls (281.72bp,1478.6bp) and (266.04bp,1525.6bp)  .. (261.47bp,1539.3bp);
  \draw [solid] (291.04bp,1462.3bp) .. controls (305.33bp,1463.9bp) and (356.47bp,1470.0bp)  .. (370.58bp,1471.7bp);
  \draw [] (1444.8bp,531.01bp) .. controls (1450.6bp,517.07bp) and (1471.1bp,467.17bp)  .. (1476.8bp,453.41bp);
  \draw [solid] (1439.5bp,534.55bp) .. controls (1425.8bp,534.18bp) and (1380.1bp,532.94bp)  .. (1366.9bp,532.58bp);
  \draw [solid] (1474.4bp,450.62bp) .. controls (1461.8bp,453.03bp) and (1422.1bp,460.65bp)  .. (1409.9bp,462.97bp);
  \draw [solid] (1497.5bp,393.0bp) .. controls (1482.2bp,390.4bp) and (1427.6bp,381.12bp)  .. (1412.6bp,378.56bp);
  \pgfsetcolor{orange}
  \draw [solid] (1501.6bp,397.55bp) .. controls (1502.1bp,412.39bp) and (1503.9bp,465.47bp)  .. (1504.4bp,480.11bp);
  \pgfsetcolor{gray}
  \draw [solid] (1501.3bp,389.86bp) .. controls (1500.4bp,375.27bp) and (1497.5bp,323.06bp)  .. (1496.7bp,308.66bp);
  \draw [] (1041.4bp,851.93bp) .. controls (1031.8bp,861.03bp) and (999.55bp,891.33bp)  .. (989.72bp,900.58bp);
  \draw [solid] (1047.9bp,848.74bp) .. controls (1061.7bp,846.5bp) and (1111.0bp,838.53bp)  .. (1125.3bp,836.21bp);
  \pgfsetcolor{orange}
  \draw [solid] (989.05bp,990.51bp) .. controls (988.67bp,975.52bp) and (987.32bp,921.87bp)  .. (986.94bp,907.08bp);
  \pgfsetcolor{gray}
  \draw [solid] (987.92bp,998.24bp) .. controls (983.59bp,1011.7bp) and (969.05bp,1056.9bp)  .. (964.83bp,1070.0bp);
  \draw [solid] (275.22bp,1376.7bp) .. controls (277.3bp,1391.3bp) and (284.73bp,1443.7bp)  .. (286.78bp,1458.1bp);
  \draw [solid] (278.72bp,1373.1bp) .. controls (294.18bp,1374.1bp) and (349.49bp,1377.5bp)  .. (364.74bp,1378.4bp);
  \draw [solid] (1407.0bp,287.73bp) .. controls (1402.8bp,274.53bp) and (1387.6bp,227.62bp)  .. (1383.2bp,213.99bp);
  \pgfsetcolor{orange}
  \draw [] (1708.0bp,43.518bp) .. controls (1698.3bp,53.07bp) and (1665.5bp,85.109bp)  .. (1656.0bp,94.41bp);
  \pgfsetcolor{gray}
  \draw [solid] (1707.4bp,39.211bp) .. controls (1694.4bp,33.09bp) and (1647.8bp,11.19bp)  .. (1634.9bp,5.1515bp);
  \pgfsetcolor{orange}
  \draw [] (1628.8bp,6.4333bp) .. controls (1619.0bp,16.426bp) and (1586.0bp,49.943bp)  .. (1576.4bp,59.673bp);
  \pgfsetcolor{gray}
  \draw [solid] (460.25bp,1380.7bp) .. controls (460.45bp,1396.2bp) and (461.22bp,1455.4bp)  .. (461.42bp,1470.9bp);
  \pgfsetcolor{orange}
  \draw [solid] (456.27bp,1377.1bp) .. controls (441.21bp,1377.4bp) and (387.34bp,1378.3bp)  .. (372.48bp,1378.6bp);
  \pgfsetcolor{gray}
  \draw [] (461.66bp,1373.6bp) .. controls (467.73bp,1359.5bp) and (491.0bp,1305.5bp)  .. (497.07bp,1291.4bp);
  \pgfsetcolor{orange}
  \draw [solid] (457.71bp,1474.5bp) .. controls (443.36bp,1474.1bp) and (392.0bp,1472.6bp)  .. (377.84bp,1472.2bp);
  \pgfsetcolor{gray}
  \draw [solid] (460.28bp,1478.1bp) .. controls (455.77bp,1491.4bp) and (439.75bp,1538.7bp)  .. (435.09bp,1552.4bp);
  \draw [solid] (198.02bp,1341.9bp) .. controls (211.19bp,1347.2bp) and (258.34bp,1366.2bp)  .. (271.34bp,1371.5bp);
  \draw [solid] (104.39bp,1532.3bp) .. controls (117.35bp,1539.0bp) and (163.72bp,1562.9bp)  .. (176.5bp,1569.5bp);
  \pgfsetcolor{orange}
  \draw [] (98.086bp,1533.1bp) .. controls (87.781bp,1542.1bp) and (53.215bp,1572.2bp)  .. (43.181bp,1580.9bp);
  \pgfsetcolor{gray}
  \draw [solid] (1570.5bp,63.981bp) .. controls (1557.9bp,70.195bp) and (1513.3bp,92.286bp)  .. (1500.4bp,98.708bp);
  \pgfsetcolor{orange}
  \draw [solid] (497.67bp,1330.9bp) .. controls (482.4bp,1328.3bp) and (423.91bp,1318.4bp)  .. (408.64bp,1315.8bp);
  \pgfsetcolor{gray}
  \draw [solid] (498.88bp,1334.3bp) .. controls (491.68bp,1342.2bp) and (470.65bp,1365.5bp)  .. (463.01bp,1373.9bp);
  \draw [] (504.28bp,1328.9bp) .. controls (514.7bp,1319.3bp) and (549.66bp,1287.4bp)  .. (559.81bp,1278.1bp);
  \draw [solid] (501.56bp,1335.4bp) .. controls (502.47bp,1349.9bp) and (505.72bp,1401.9bp)  .. (506.61bp,1416.3bp);
  \draw [] (1308.9bp,583.19bp) .. controls (1318.5bp,574.21bp) and (1350.6bp,544.29bp)  .. (1360.4bp,535.15bp);
  \draw [] (1365.3bp,529.16bp) .. controls (1372.7bp,517.46bp) and (1397.3bp,478.18bp)  .. (1404.4bp,466.78bp);
  \draw [] (565.35bp,1273.1bp) .. controls (574.98bp,1264.6bp) and (607.09bp,1236.0bp)  .. (616.88bp,1227.3bp);
  \draw [solid] (558.77bp,1276.3bp) .. controls (547.35bp,1278.5bp) and (513.53bp,1285.1bp)  .. (502.25bp,1287.3bp);
  \draw [solid] (1782.5bp,11.376bp) .. controls (1769.7bp,16.636bp) and (1726.7bp,34.282bp)  .. (1714.2bp,39.404bp);
  \pgfsetcolor{orange}
  \draw [solid] (1572.3bp,368.23bp) .. controls (1572.3bp,382.66bp) and (1572.0bp,434.29bp)  .. (1572.0bp,448.52bp);
  \pgfsetcolor{gray}
  \draw [solid] (1568.4bp,364.03bp) .. controls (1553.4bp,362.39bp) and (1499.6bp,356.53bp)  .. (1484.8bp,354.91bp);
  \draw [solid] (1568.9bp,365.87bp) .. controls (1557.0bp,370.79bp) and (1517.2bp,387.2bp)  .. (1505.0bp,392.21bp);
  \draw [solid] (1572.2bp,360.71bp) .. controls (1571.6bp,346.36bp) and (1569.4bp,294.99bp)  .. (1568.8bp,280.83bp);
  \draw [solid] (1492.7bp,304.37bp) .. controls (1478.1bp,302.11bp) and (1426.2bp,294.02bp)  .. (1411.8bp,291.78bp);
  \draw [solid] (241.1bp,1272.5bp) .. controls (253.95bp,1278.1bp) and (299.66bp,1297.8bp)  .. (312.95bp,1303.6bp);
  \draw [solid] (235.63bp,1274.4bp) .. controls (228.29bp,1286.2bp) and (203.68bp,1325.8bp)  .. (196.54bp,1337.3bp);
  \draw [solid] (1060.9bp,768.6bp) .. controls (1058.1bp,782.35bp) and (1048.0bp,831.22bp)  .. (1045.0bp,845.43bp);
  \draw [solid] (1065.3bp,764.45bp) .. controls (1078.9bp,762.48bp) and (1127.1bp,755.46bp)  .. (1141.2bp,753.42bp);
  \draw [] (1198.7bp,700.09bp) .. controls (1189.2bp,709.44bp) and (1157.5bp,740.58bp)  .. (1147.9bp,750.08bp);
  \pgfsetcolor{orange}
  \draw [] (1204.0bp,694.74bp) .. controls (1212.9bp,685.39bp) and (1242.6bp,654.24bp)  .. (1251.7bp,644.73bp);
  \pgfsetcolor{gray}
  \draw [solid] (1144.4bp,756.41bp) .. controls (1141.8bp,769.9bp) and (1132.6bp,817.82bp)  .. (1130.0bp,831.75bp);
  \draw [solid] (368.9bp,1382.7bp) .. controls (369.81bp,1398.0bp) and (373.07bp,1453.0bp)  .. (373.97bp,1468.2bp);
  \draw [solid] (622.58bp,1222.4bp) .. controls (632.49bp,1214.4bp) and (665.51bp,1187.5bp)  .. (675.58bp,1179.3bp);
  \draw [solid] (899.91bp,893.41bp) .. controls (914.85bp,895.11bp) and (968.32bp,901.18bp)  .. (983.06bp,902.85bp);
  \pgfsetcolor{orange}
  \draw [] (681.34bp,1174.5bp) .. controls (691.19bp,1166.2bp) and (724.02bp,1138.6bp)  .. (734.03bp,1130.2bp);
  \draw [solid] (1048.8bp,935.19bp) .. controls (1048.0bp,920.44bp) and (1045.2bp,867.63bp)  .. (1044.4bp,853.07bp);
  \pgfsetcolor{gray}
  \draw [] (1046.1bp,941.72bp) .. controls (1035.9bp,951.14bp) and (1001.8bp,982.74bp)  .. (991.87bp,991.91bp);
  \draw [solid] (1052.9bp,938.12bp) .. controls (1066.7bp,934.85bp) and (1112.8bp,923.89bp)  .. (1126.2bp,920.7bp);
  \draw [solid] (114.42bp,1377.5bp) .. controls (107.29bp,1389.6bp) and (81.931bp,1432.5bp)  .. (74.56bp,1444.9bp);
  \draw [] (119.66bp,1372.9bp) .. controls (132.42bp,1367.4bp) and (177.77bp,1347.7bp)  .. (190.96bp,1342.0bp);
  \draw [solid] (1453.0bp,177.51bp) .. controls (1440.3bp,183.36bp) and (1397.8bp,202.98bp)  .. (1385.4bp,208.68bp);
  \draw [] (1458.3bp,172.62bp) .. controls (1464.9bp,160.33bp) and (1488.2bp,116.65bp)  .. (1494.9bp,103.96bp);
  \draw [solid] (1650.9bp,430.93bp) .. controls (1637.5bp,434.55bp) and (1589.7bp,447.41bp)  .. (1575.8bp,451.15bp);
  \draw [solid] (373.04bp,1475.6bp) .. controls (368.63bp,1488.8bp) and (352.97bp,1535.7bp)  .. (348.41bp,1549.3bp);
  \draw [solid] (322.01bp,1398.7bp) .. controls (315.86bp,1409.9bp) and (295.38bp,1447.1bp)  .. (289.13bp,1458.5bp);
  \draw [solid] (327.63bp,1396.1bp) .. controls (342.41bp,1398.4bp) and (395.3bp,1406.6bp)  .. (409.89bp,1408.8bp);
  \draw [solid] (1256.8bp,639.24bp) .. controls (1265.6bp,629.77bp) and (1294.6bp,598.2bp)  .. (1303.5bp,588.57bp);
  \draw [solid] (504.4bp,1422.9bp) .. controls (496.52bp,1432.4bp) and (471.81bp,1462.1bp)  .. (463.92bp,1471.7bp);
  \pgfsetcolor{orange}
  \draw [solid] (502.84bp,1419.5bp) .. controls (487.51bp,1417.8bp) and (432.65bp,1411.6bp)  .. (417.52bp,1409.8bp);
  \pgfsetcolor{gray}
  \draw [solid] (818.69bp,908.13bp) .. controls (832.51bp,905.42bp) and (878.86bp,896.33bp)  .. (892.31bp,893.7bp);
  \draw [solid] (883.4bp,1066.6bp) .. controls (897.07bp,1067.8bp) and (945.67bp,1072.0bp)  .. (959.79bp,1073.2bp);
  \draw [solid] (1565.2bp,278.53bp) .. controls (1553.0bp,283.21bp) and (1512.5bp,298.81bp)  .. (1500.1bp,303.57bp);
  \draw [solid] (1565.0bp,276.4bp) .. controls (1550.7bp,273.36bp) and (1499.5bp,262.51bp)  .. (1485.4bp,259.52bp);
  \draw [solid] (1478.3bp,260.31bp) .. controls (1465.7bp,265.84bp) and (1423.7bp,284.36bp)  .. (1411.5bp,289.74bp);
  \draw [solid] (1480.7bp,255.19bp) .. controls (1476.6bp,241.55bp) and (1461.8bp,192.76bp)  .. (1457.7bp,179.31bp);
  \draw [solid] (818.34bp,992.82bp) .. controls (832.27bp,991.07bp) and (881.78bp,984.84bp)  .. (896.17bp,983.03bp);
  \pgfsetcolor{orange}
  \draw [solid] (814.67bp,989.66bp) .. controls (814.69bp,975.9bp) and (814.75bp,927.01bp)  .. (814.77bp,912.8bp);
  \pgfsetcolor{gray}
  \draw [solid] (813.87bp,996.81bp) .. controls (810.86bp,1010.2bp) and (800.14bp,1057.8bp)  .. (797.02bp,1071.7bp);
  \draw [solid] (411.51bp,1412.8bp) .. controls (404.6bp,1423.8bp) and (382.79bp,1458.4bp)  .. (376.14bp,1469.0bp);
  \pgfsetcolor{orange}
  \draw [solid] (430.09bp,1556.1bp) .. controls (415.83bp,1555.5bp) and (364.84bp,1553.7bp)  .. (350.78bp,1553.2bp);
  \pgfsetcolor{gray}
  \draw [] (793.31bp,1078.0bp) .. controls (783.33bp,1086.8bp) and (750.08bp,1116.2bp)  .. (739.93bp,1125.1bp);
  \draw [solid] (799.76bp,1075.1bp) .. controls (813.39bp,1073.6bp) and (861.84bp,1068.2bp)  .. (875.93bp,1066.7bp);
  \draw [] (1653.6bp,252.88bp) .. controls (1659.3bp,240.24bp) and (1679.7bp,195.32bp)  .. (1685.6bp,182.26bp);
  \draw [solid] (1648.5bp,257.11bp) .. controls (1634.9bp,260.53bp) and (1586.6bp,272.69bp)  .. (1572.5bp,276.22bp);
  \draw [solid] (1651.9bp,344.31bp) .. controls (1638.4bp,347.74bp) and (1590.2bp,359.94bp)  .. (1576.2bp,363.49bp);
  \pgfsetcolor{orange}
  \draw [solid] (1655.5bp,347.13bp) .. controls (1655.3bp,361.37bp) and (1654.7bp,412.32bp)  .. (1654.5bp,426.36bp);
  \pgfsetcolor{gray}
  \draw [solid] (1655.4bp,339.66bp) .. controls (1654.8bp,325.32bp) and (1652.7bp,274.0bp)  .. (1652.2bp,259.85bp);
  \draw [solid] (1816.6bp,85.204bp) .. controls (1811.4bp,72.376bp) and (1793.0bp,26.777bp)  .. (1787.7bp,13.522bp);
  \draw [solid] (1814.6bp,90.302bp) .. controls (1802.7bp,96.298bp) and (1762.9bp,116.28bp)  .. (1750.8bp,122.38bp);
  \pgfsetcolor{orange}
  \draw [] (1744.4bp,126.79bp) .. controls (1734.2bp,136.07bp) and (1699.9bp,167.17bp)  .. (1690.0bp,176.2bp);
  \pgfsetcolor{gray}
  \draw [solid] (1745.7bp,120.58bp) .. controls (1739.7bp,106.87bp) and (1718.2bp,57.812bp)  .. (1712.3bp,44.285bp);
  \pgfsetcolor{orange}
  \draw [solid] (1129.9bp,916.21bp) .. controls (1129.7bp,902.48bp) and (1129.4bp,853.66bp)  .. (1129.3bp,839.47bp);
\begin{scope}
  \definecolor{strokecol}{rgb}{0.0,0.0,0.0};
  \pgfsetstrokecolor{strokecol}
  \definecolor{fillcol}{rgb}{0.0,0.0,0.0};
  \pgfsetfillcolor{fillcol}
  \filldraw [opacity=1] (72.54bp,1448.35bp) ellipse (3.6bp and 3.6bp);
\end{scope}
\begin{scope}
  \definecolor{strokecol}{rgb}{0.0,0.0,0.0};
  \pgfsetstrokecolor{strokecol}
  \definecolor{fillcol}{rgb}{0.0,0.0,0.0};
  \pgfsetfillcolor{fillcol}
  \filldraw [opacity=1] (9.47bp,1501.41bp) ellipse (3.6bp and 3.6bp);
\end{scope}
\begin{scope}
  \definecolor{strokecol}{rgb}{0.0,0.0,0.0};
  \pgfsetstrokecolor{strokecol}
  \definecolor{fillcol}{rgb}{0.0,0.0,0.0};
  \pgfsetfillcolor{fillcol}
  \filldraw [opacity=1] (101.01bp,1530.6bp) ellipse (3.6bp and 3.6bp);
\end{scope}
\begin{scope}
  \definecolor{strokecol}{rgb}{0.0,0.0,0.0};
  \pgfsetstrokecolor{strokecol}
  \definecolor{fillcol}{rgb}{0.0,0.0,0.0};
  \pgfsetfillcolor{fillcol}
  \filldraw [opacity=1] (40.43bp,1583.33bp) ellipse (3.6bp and 3.6bp);
\end{scope}
\begin{scope}
  \definecolor{strokecol}{rgb}{0.0,0.0,0.0};
  \pgfsetstrokecolor{strokecol}
  \definecolor{fillcol}{rgb}{0.0,0.0,0.0};
  \pgfsetfillcolor{fillcol}
  \filldraw [opacity=1] (179.79bp,1571.21bp) ellipse (3.6bp and 3.6bp);
\end{scope}
\begin{scope}
  \definecolor{strokecol}{rgb}{0.0,0.0,0.0};
  \pgfsetstrokecolor{strokecol}
  \definecolor{fillcol}{rgb}{0.0,0.0,0.0};
  \pgfsetfillcolor{fillcol}
  \filldraw [opacity=1] (260.22bp,1543.09bp) ellipse (3.6bp and 3.6bp);
\end{scope}
\begin{scope}
  \definecolor{strokecol}{rgb}{0.0,0.0,0.0};
  \pgfsetstrokecolor{strokecol}
  \definecolor{fillcol}{rgb}{0.0,0.0,0.0};
  \pgfsetfillcolor{fillcol}
  \filldraw [opacity=1] (120.58bp,1626.48bp) ellipse (3.6bp and 3.6bp);
\end{scope}
\begin{scope}
  \definecolor{strokecol}{rgb}{0.0,0.0,0.0};
  \pgfsetstrokecolor{strokecol}
  \definecolor{fillcol}{rgb}{0.0,0.0,0.0};
  \pgfsetfillcolor{fillcol}
  \filldraw [opacity=1] (347.16bp,1553.07bp) ellipse (3.6bp and 3.6bp);
\end{scope}
\begin{scope}
  \definecolor{strokecol}{rgb}{0.0,0.0,0.0};
  \pgfsetstrokecolor{strokecol}
  \definecolor{fillcol}{rgb}{0.0,0.0,0.0};
  \pgfsetfillcolor{fillcol}
  \filldraw [opacity=1] (1408.72bp,377.9bp) ellipse (3.6bp and 3.6bp);
\end{scope}
\begin{scope}
  \definecolor{strokecol}{rgb}{0.0,0.0,0.0};
  \pgfsetstrokecolor{strokecol}
  \definecolor{fillcol}{rgb}{0.0,0.0,0.0};
  \pgfsetfillcolor{fillcol}
  \filldraw [opacity=1] (1406.4bp,463.65bp) ellipse (3.6bp and 3.6bp);
\end{scope}
\begin{scope}
  \definecolor{strokecol}{rgb}{0.0,0.0,0.0};
  \pgfsetstrokecolor{strokecol}
  \definecolor{fillcol}{rgb}{0.0,0.0,0.0};
  \pgfsetfillcolor{fillcol}
  \filldraw [opacity=1] (1408.17bp,291.21bp) ellipse (3.6bp and 3.6bp);
\end{scope}
\begin{scope}
  \definecolor{strokecol}{rgb}{0.0,0.0,0.0};
  \pgfsetstrokecolor{strokecol}
  \definecolor{fillcol}{rgb}{0.0,0.0,0.0};
  \pgfsetfillcolor{fillcol}
  \filldraw [opacity=1] (79.49bp,1694.13bp) ellipse (3.6bp and 3.6bp);
\end{scope}
\begin{scope}
  \definecolor{strokecol}{rgb}{0.0,0.0,0.0};
  \pgfsetstrokecolor{strokecol}
  \definecolor{fillcol}{rgb}{0.0,0.0,0.0};
  \pgfsetfillcolor{fillcol}
  \filldraw [opacity=1] (1571.96bp,452.18bp) ellipse (3.6bp and 3.6bp);
\end{scope}
\begin{scope}
  \definecolor{strokecol}{rgb}{0.0,0.0,0.0};
  \pgfsetstrokecolor{strokecol}
  \definecolor{fillcol}{rgb}{0.0,0.0,0.0};
  \pgfsetfillcolor{fillcol}
  \filldraw [opacity=1] (1504.51bp,483.87bp) ellipse (3.6bp and 3.6bp);
\end{scope}
\begin{scope}
  \definecolor{strokecol}{rgb}{0.0,0.0,0.0};
  \pgfsetstrokecolor{strokecol}
  \definecolor{fillcol}{rgb}{0.0,0.0,0.0};
  \pgfsetfillcolor{fillcol}
  \filldraw [opacity=1] (1478.24bp,449.88bp) ellipse (3.6bp and 3.6bp);
\end{scope}
\begin{scope}
  \definecolor{strokecol}{rgb}{0.0,0.0,0.0};
  \pgfsetstrokecolor{strokecol}
  \definecolor{fillcol}{rgb}{0.0,0.0,0.0};
  \pgfsetfillcolor{fillcol}
  \filldraw [opacity=1] (1687.26bp,178.68bp) ellipse (3.6bp and 3.6bp);
\end{scope}
\begin{scope}
  \definecolor{strokecol}{rgb}{0.0,0.0,0.0};
  \pgfsetstrokecolor{strokecol}
  \definecolor{fillcol}{rgb}{0.0,0.0,0.0};
  \pgfsetfillcolor{fillcol}
  \filldraw [opacity=1] (1653.38bp,96.96bp) ellipse (3.6bp and 3.6bp);
\end{scope}
\begin{scope}
  \definecolor{strokecol}{rgb}{0.0,0.0,0.0};
  \pgfsetstrokecolor{strokecol}
  \definecolor{fillcol}{rgb}{0.0,0.0,0.0};
  \pgfsetfillcolor{fillcol}
  \filldraw [opacity=1] (1573.76bp,62.34bp) ellipse (3.6bp and 3.6bp);
\end{scope}
\begin{scope}
  \definecolor{strokecol}{rgb}{0.0,0.0,0.0};
  \pgfsetstrokecolor{strokecol}
  \definecolor{fillcol}{rgb}{0.0,0.0,0.0};
  \pgfsetfillcolor{fillcol}
  \filldraw [opacity=1] (958.34bp,929.54bp) ellipse (3.6bp and 3.6bp);
\end{scope}
\begin{scope}
  \definecolor{strokecol}{rgb}{0.0,0.0,0.0};
  \pgfsetstrokecolor{strokecol}
  \definecolor{fillcol}{rgb}{0.0,0.0,0.0};
  \pgfsetfillcolor{fillcol}
  \filldraw [opacity=1] (900.12bp,982.54bp) ellipse (3.6bp and 3.6bp);
\end{scope}
\begin{scope}
  \definecolor{strokecol}{rgb}{0.0,0.0,0.0};
  \pgfsetstrokecolor{strokecol}
  \definecolor{fillcol}{rgb}{0.0,0.0,0.0};
  \pgfsetfillcolor{fillcol}
  \filldraw [opacity=1] (955.27bp,838.16bp) ellipse (3.6bp and 3.6bp);
\end{scope}
\begin{scope}
  \definecolor{strokecol}{rgb}{0.0,0.0,0.0};
  \pgfsetstrokecolor{strokecol}
  \definecolor{fillcol}{rgb}{0.0,0.0,0.0};
  \pgfsetfillcolor{fillcol}
  \filldraw [opacity=1] (1049.01bp,939.05bp) ellipse (3.6bp and 3.6bp);
\end{scope}
\begin{scope}
  \definecolor{strokecol}{rgb}{0.0,0.0,0.0};
  \pgfsetstrokecolor{strokecol}
  \definecolor{fillcol}{rgb}{0.0,0.0,0.0};
  \pgfsetfillcolor{fillcol}
  \filldraw [opacity=1] (989.15bp,994.42bp) ellipse (3.6bp and 3.6bp);
\end{scope}
\begin{scope}
  \definecolor{strokecol}{rgb}{0.0,0.0,0.0};
  \pgfsetstrokecolor{strokecol}
  \definecolor{fillcol}{rgb}{0.0,0.0,0.0};
  \pgfsetfillcolor{fillcol}
  \filldraw [opacity=1] (896.0bp,892.97bp) ellipse (3.6bp and 3.6bp);
\end{scope}
\begin{scope}
  \definecolor{strokecol}{rgb}{0.0,0.0,0.0};
  \pgfsetstrokecolor{strokecol}
  \definecolor{fillcol}{rgb}{0.0,0.0,0.0};
  \pgfsetfillcolor{fillcol}
  \filldraw [opacity=1] (879.79bp,1066.26bp) ellipse (3.6bp and 3.6bp);
\end{scope}
\begin{scope}
  \definecolor{strokecol}{rgb}{0.0,0.0,0.0};
  \pgfsetstrokecolor{strokecol}
  \definecolor{fillcol}{rgb}{0.0,0.0,0.0};
  \pgfsetfillcolor{fillcol}
  \filldraw [opacity=1] (3.6bp,1656.03bp) ellipse (3.6bp and 3.6bp);
\end{scope}
\begin{scope}
  \definecolor{strokecol}{rgb}{0.0,0.0,0.0};
  \pgfsetstrokecolor{strokecol}
  \definecolor{fillcol}{rgb}{0.0,0.0,0.0};
  \pgfsetfillcolor{fillcol}
  \filldraw [opacity=1] (977.59bp,758.43bp) ellipse (3.6bp and 3.6bp);
\end{scope}
\begin{scope}
  \definecolor{strokecol}{rgb}{0.0,0.0,0.0};
  \pgfsetstrokecolor{strokecol}
  \definecolor{fillcol}{rgb}{0.0,0.0,0.0};
  \pgfsetfillcolor{fillcol}
  \filldraw [opacity=1] (1061.69bp,764.97bp) ellipse (3.6bp and 3.6bp);
\end{scope}
\begin{scope}
  \definecolor{strokecol}{rgb}{0.0,0.0,0.0};
  \pgfsetstrokecolor{strokecol}
  \definecolor{fillcol}{rgb}{0.0,0.0,0.0};
  \pgfsetfillcolor{fillcol}
  \filldraw [opacity=1] (1044.21bp,849.33bp) ellipse (3.6bp and 3.6bp);
\end{scope}
\begin{scope}
  \definecolor{strokecol}{rgb}{0.0,0.0,0.0};
  \pgfsetstrokecolor{strokecol}
  \definecolor{fillcol}{rgb}{0.0,0.0,0.0};
  \pgfsetfillcolor{fillcol}
  \filldraw [opacity=1] (316.6bp,1305.15bp) ellipse (3.6bp and 3.6bp);
\end{scope}
\begin{scope}
  \definecolor{strokecol}{rgb}{0.0,0.0,0.0};
  \pgfsetstrokecolor{strokecol}
  \definecolor{fillcol}{rgb}{0.0,0.0,0.0};
  \pgfsetfillcolor{fillcol}
  \filldraw [opacity=1] (404.98bp,1315.2bp) ellipse (3.6bp and 3.6bp);
\end{scope}
\begin{scope}
  \definecolor{strokecol}{rgb}{0.0,0.0,0.0};
  \pgfsetstrokecolor{strokecol}
  \definecolor{fillcol}{rgb}{0.0,0.0,0.0};
  \pgfsetfillcolor{fillcol}
  \filldraw [opacity=1] (274.68bp,1372.85bp) ellipse (3.6bp and 3.6bp);
\end{scope}
\begin{scope}
  \definecolor{strokecol}{rgb}{0.0,0.0,0.0};
  \pgfsetstrokecolor{strokecol}
  \definecolor{fillcol}{rgb}{0.0,0.0,0.0};
  \pgfsetfillcolor{fillcol}
  \filldraw [opacity=1] (323.76bp,1395.52bp) ellipse (3.6bp and 3.6bp);
\end{scope}
\begin{scope}
  \definecolor{strokecol}{rgb}{0.0,0.0,0.0};
  \pgfsetstrokecolor{strokecol}
  \definecolor{fillcol}{rgb}{0.0,0.0,0.0};
  \pgfsetfillcolor{fillcol}
  \filldraw [opacity=1] (368.66bp,1378.66bp) ellipse (3.6bp and 3.6bp);
\end{scope}
\begin{scope}
  \definecolor{strokecol}{rgb}{0.0,0.0,0.0};
  \pgfsetstrokecolor{strokecol}
  \definecolor{fillcol}{rgb}{0.0,0.0,0.0};
  \pgfsetfillcolor{fillcol}
  \filldraw [opacity=1] (413.64bp,1409.4bp) ellipse (3.6bp and 3.6bp);
\end{scope}
\begin{scope}
  \definecolor{strokecol}{rgb}{0.0,0.0,0.0};
  \pgfsetstrokecolor{strokecol}
  \definecolor{fillcol}{rgb}{0.0,0.0,0.0};
  \pgfsetfillcolor{fillcol}
  \filldraw [opacity=1] (1480.98bp,354.5bp) ellipse (3.6bp and 3.6bp);
\end{scope}
\begin{scope}
  \definecolor{strokecol}{rgb}{0.0,0.0,0.0};
  \pgfsetstrokecolor{strokecol}
  \definecolor{fillcol}{rgb}{0.0,0.0,0.0};
  \pgfsetfillcolor{fillcol}
  \filldraw [opacity=1] (1481.82bp,258.75bp) ellipse (3.6bp and 3.6bp);
\end{scope}
\begin{scope}
  \definecolor{strokecol}{rgb}{0.0,0.0,0.0};
  \pgfsetstrokecolor{strokecol}
  \definecolor{fillcol}{rgb}{0.0,0.0,0.0};
  \pgfsetfillcolor{fillcol}
  \filldraw [opacity=1] (287.3bp,1461.81bp) ellipse (3.6bp and 3.6bp);
\end{scope}
\begin{scope}
  \definecolor{strokecol}{rgb}{0.0,0.0,0.0};
  \pgfsetstrokecolor{strokecol}
  \definecolor{fillcol}{rgb}{0.0,0.0,0.0};
  \pgfsetfillcolor{fillcol}
  \filldraw [opacity=1] (374.2bp,1472.09bp) ellipse (3.6bp and 3.6bp);
\end{scope}
\begin{scope}
  \definecolor{strokecol}{rgb}{0.0,0.0,0.0};
  \pgfsetstrokecolor{strokecol}
  \definecolor{fillcol}{rgb}{0.0,0.0,0.0};
  \pgfsetfillcolor{fillcol}
  \filldraw [opacity=1] (1443.32bp,534.65bp) ellipse (3.6bp and 3.6bp);
\end{scope}
\begin{scope}
  \definecolor{strokecol}{rgb}{0.0,0.0,0.0};
  \pgfsetstrokecolor{strokecol}
  \definecolor{fillcol}{rgb}{0.0,0.0,0.0};
  \pgfsetfillcolor{fillcol}
  \filldraw [opacity=1] (1363.25bp,532.48bp) ellipse (3.6bp and 3.6bp);
\end{scope}
\begin{scope}
  \definecolor{strokecol}{rgb}{0.0,0.0,0.0};
  \pgfsetstrokecolor{strokecol}
  \definecolor{fillcol}{rgb}{0.0,0.0,0.0};
  \pgfsetfillcolor{fillcol}
  \filldraw [opacity=1] (1501.47bp,393.67bp) ellipse (3.6bp and 3.6bp);
\end{scope}
\begin{scope}
  \definecolor{strokecol}{rgb}{0.0,0.0,0.0};
  \pgfsetstrokecolor{strokecol}
  \definecolor{fillcol}{rgb}{0.0,0.0,0.0};
  \pgfsetfillcolor{fillcol}
  \filldraw [opacity=1] (1496.46bp,304.96bp) ellipse (3.6bp and 3.6bp);
\end{scope}
\begin{scope}
  \definecolor{strokecol}{rgb}{0.0,0.0,0.0};
  \pgfsetstrokecolor{strokecol}
  \definecolor{fillcol}{rgb}{0.0,0.0,0.0};
  \pgfsetfillcolor{fillcol}
  \filldraw [opacity=1] (986.85bp,903.28bp) ellipse (3.6bp and 3.6bp);
\end{scope}
\begin{scope}
  \definecolor{strokecol}{rgb}{0.0,0.0,0.0};
  \pgfsetstrokecolor{strokecol}
  \definecolor{fillcol}{rgb}{0.0,0.0,0.0};
  \pgfsetfillcolor{fillcol}
  \filldraw [opacity=1] (1129.23bp,835.57bp) ellipse (3.6bp and 3.6bp);
\end{scope}
\begin{scope}
  \definecolor{strokecol}{rgb}{0.0,0.0,0.0};
  \pgfsetstrokecolor{strokecol}
  \definecolor{fillcol}{rgb}{0.0,0.0,0.0};
  \pgfsetfillcolor{fillcol}
  \filldraw [opacity=1] (963.67bp,1073.58bp) ellipse (3.6bp and 3.6bp);
\end{scope}
\begin{scope}
  \definecolor{strokecol}{rgb}{0.0,0.0,0.0};
  \pgfsetstrokecolor{strokecol}
  \definecolor{fillcol}{rgb}{0.0,0.0,0.0};
  \pgfsetfillcolor{fillcol}
  \filldraw [opacity=1] (1382.02bp,210.24bp) ellipse (3.6bp and 3.6bp);
\end{scope}
\begin{scope}
  \definecolor{strokecol}{rgb}{0.0,0.0,0.0};
  \pgfsetstrokecolor{strokecol}
  \definecolor{fillcol}{rgb}{0.0,0.0,0.0};
  \pgfsetfillcolor{fillcol}
  \filldraw [opacity=1] (1710.79bp,40.81bp) ellipse (3.6bp and 3.6bp);
\end{scope}
\begin{scope}
  \definecolor{strokecol}{rgb}{0.0,0.0,0.0};
  \pgfsetstrokecolor{strokecol}
  \definecolor{fillcol}{rgb}{0.0,0.0,0.0};
  \pgfsetfillcolor{fillcol}
  \filldraw [opacity=1] (1631.63bp,3.6bp) ellipse (3.6bp and 3.6bp);
\end{scope}
\begin{scope}
  \definecolor{strokecol}{rgb}{0.0,0.0,0.0};
  \pgfsetstrokecolor{strokecol}
  \definecolor{fillcol}{rgb}{0.0,0.0,0.0};
  \pgfsetfillcolor{fillcol}
  \filldraw [opacity=1] (460.2bp,1377.01bp) ellipse (3.6bp and 3.6bp);
\end{scope}
\begin{scope}
  \definecolor{strokecol}{rgb}{0.0,0.0,0.0};
  \pgfsetstrokecolor{strokecol}
  \definecolor{fillcol}{rgb}{0.0,0.0,0.0};
  \pgfsetfillcolor{fillcol}
  \filldraw [opacity=1] (461.46bp,1474.61bp) ellipse (3.6bp and 3.6bp);
\end{scope}
\begin{scope}
  \definecolor{strokecol}{rgb}{0.0,0.0,0.0};
  \pgfsetstrokecolor{strokecol}
  \definecolor{fillcol}{rgb}{0.0,0.0,0.0};
  \pgfsetfillcolor{fillcol}
  \filldraw [opacity=1] (498.53bp,1288.03bp) ellipse (3.6bp and 3.6bp);
\end{scope}
\begin{scope}
  \definecolor{strokecol}{rgb}{0.0,0.0,0.0};
  \pgfsetstrokecolor{strokecol}
  \definecolor{fillcol}{rgb}{0.0,0.0,0.0};
  \pgfsetfillcolor{fillcol}
  \filldraw [opacity=1] (433.81bp,1556.19bp) ellipse (3.6bp and 3.6bp);
\end{scope}
\begin{scope}
  \definecolor{strokecol}{rgb}{0.0,0.0,0.0};
  \pgfsetstrokecolor{strokecol}
  \definecolor{fillcol}{rgb}{0.0,0.0,0.0};
  \pgfsetfillcolor{fillcol}
  \filldraw [opacity=1] (194.58bp,1340.46bp) ellipse (3.6bp and 3.6bp);
\end{scope}
\begin{scope}
  \definecolor{strokecol}{rgb}{0.0,0.0,0.0};
  \pgfsetstrokecolor{strokecol}
  \definecolor{fillcol}{rgb}{0.0,0.0,0.0};
  \pgfsetfillcolor{fillcol}
  \filldraw [opacity=1] (1496.81bp,100.47bp) ellipse (3.6bp and 3.6bp);
\end{scope}
\begin{scope}
  \definecolor{strokecol}{rgb}{0.0,0.0,0.0};
  \pgfsetstrokecolor{strokecol}
  \definecolor{fillcol}{rgb}{0.0,0.0,0.0};
  \pgfsetfillcolor{fillcol}
  \filldraw [opacity=1] (501.32bp,1331.56bp) ellipse (3.6bp and 3.6bp);
\end{scope}
\begin{scope}
  \definecolor{strokecol}{rgb}{0.0,0.0,0.0};
  \pgfsetstrokecolor{strokecol}
  \definecolor{fillcol}{rgb}{0.0,0.0,0.0};
  \pgfsetfillcolor{fillcol}
  \filldraw [opacity=1] (562.59bp,1275.57bp) ellipse (3.6bp and 3.6bp);
\end{scope}
\begin{scope}
  \definecolor{strokecol}{rgb}{0.0,0.0,0.0};
  \pgfsetstrokecolor{strokecol}
  \definecolor{fillcol}{rgb}{0.0,0.0,0.0};
  \pgfsetfillcolor{fillcol}
  \filldraw [opacity=1] (506.84bp,1419.95bp) ellipse (3.6bp and 3.6bp);
\end{scope}
\begin{scope}
  \definecolor{strokecol}{rgb}{0.0,0.0,0.0};
  \pgfsetstrokecolor{strokecol}
  \definecolor{fillcol}{rgb}{0.0,0.0,0.0};
  \pgfsetfillcolor{fillcol}
  \filldraw [opacity=1] (1306.12bp,585.76bp) ellipse (3.6bp and 3.6bp);
\end{scope}
\begin{scope}
  \definecolor{strokecol}{rgb}{0.0,0.0,0.0};
  \pgfsetstrokecolor{strokecol}
  \definecolor{fillcol}{rgb}{0.0,0.0,0.0};
  \pgfsetfillcolor{fillcol}
  \filldraw [opacity=1] (619.75bp,1224.73bp) ellipse (3.6bp and 3.6bp);
\end{scope}
\begin{scope}
  \definecolor{strokecol}{rgb}{0.0,0.0,0.0};
  \pgfsetstrokecolor{strokecol}
  \definecolor{fillcol}{rgb}{0.0,0.0,0.0};
  \pgfsetfillcolor{fillcol}
  \filldraw [opacity=1] (1786.19bp,9.88bp) ellipse (3.6bp and 3.6bp);
\end{scope}
\begin{scope}
  \definecolor{strokecol}{rgb}{0.0,0.0,0.0};
  \pgfsetstrokecolor{strokecol}
  \definecolor{fillcol}{rgb}{0.0,0.0,0.0};
  \pgfsetfillcolor{fillcol}
  \filldraw [opacity=1] (1572.37bp,364.46bp) ellipse (3.6bp and 3.6bp);
\end{scope}
\begin{scope}
  \definecolor{strokecol}{rgb}{0.0,0.0,0.0};
  \pgfsetstrokecolor{strokecol}
  \definecolor{fillcol}{rgb}{0.0,0.0,0.0};
  \pgfsetfillcolor{fillcol}
  \filldraw [opacity=1] (1568.69bp,277.19bp) ellipse (3.6bp and 3.6bp);
\end{scope}
\begin{scope}
  \definecolor{strokecol}{rgb}{0.0,0.0,0.0};
  \pgfsetstrokecolor{strokecol}
  \definecolor{fillcol}{rgb}{0.0,0.0,0.0};
  \pgfsetfillcolor{fillcol}
  \filldraw [opacity=1] (237.71bp,1271.03bp) ellipse (3.6bp and 3.6bp);
\end{scope}
\begin{scope}
  \definecolor{strokecol}{rgb}{0.0,0.0,0.0};
  \pgfsetstrokecolor{strokecol}
  \definecolor{fillcol}{rgb}{0.0,0.0,0.0};
  \pgfsetfillcolor{fillcol}
  \filldraw [opacity=1] (1145.03bp,752.86bp) ellipse (3.6bp and 3.6bp);
\end{scope}
\begin{scope}
  \definecolor{strokecol}{rgb}{0.0,0.0,0.0};
  \pgfsetstrokecolor{strokecol}
  \definecolor{fillcol}{rgb}{0.0,0.0,0.0};
  \pgfsetfillcolor{fillcol}
  \filldraw [opacity=1] (1201.47bp,697.42bp) ellipse (3.6bp and 3.6bp);
\end{scope}
\begin{scope}
  \definecolor{strokecol}{rgb}{0.0,0.0,0.0};
  \pgfsetstrokecolor{strokecol}
  \definecolor{fillcol}{rgb}{0.0,0.0,0.0};
  \pgfsetfillcolor{fillcol}
  \filldraw [opacity=1] (1254.32bp,641.95bp) ellipse (3.6bp and 3.6bp);
\end{scope}
\begin{scope}
  \definecolor{strokecol}{rgb}{0.0,0.0,0.0};
  \pgfsetstrokecolor{strokecol}
  \definecolor{fillcol}{rgb}{0.0,0.0,0.0};
  \pgfsetfillcolor{fillcol}
  \filldraw [opacity=1] (678.53bp,1176.87bp) ellipse (3.6bp and 3.6bp);
\end{scope}
\begin{scope}
  \definecolor{strokecol}{rgb}{0.0,0.0,0.0};
  \pgfsetstrokecolor{strokecol}
  \definecolor{fillcol}{rgb}{0.0,0.0,0.0};
  \pgfsetfillcolor{fillcol}
  \filldraw [opacity=1] (736.96bp,1127.74bp) ellipse (3.6bp and 3.6bp);
\end{scope}
\begin{scope}
  \definecolor{strokecol}{rgb}{0.0,0.0,0.0};
  \pgfsetstrokecolor{strokecol}
  \definecolor{fillcol}{rgb}{0.0,0.0,0.0};
  \pgfsetfillcolor{fillcol}
  \filldraw [opacity=1] (1129.88bp,919.83bp) ellipse (3.6bp and 3.6bp);
\end{scope}
\begin{scope}
  \definecolor{strokecol}{rgb}{0.0,0.0,0.0};
  \pgfsetstrokecolor{strokecol}
  \definecolor{fillcol}{rgb}{0.0,0.0,0.0};
  \pgfsetfillcolor{fillcol}
  \filldraw [opacity=1] (116.3bp,1374.36bp) ellipse (3.6bp and 3.6bp);
\end{scope}
\begin{scope}
  \definecolor{strokecol}{rgb}{0.0,0.0,0.0};
  \pgfsetstrokecolor{strokecol}
  \definecolor{fillcol}{rgb}{0.0,0.0,0.0};
  \pgfsetfillcolor{fillcol}
  \filldraw [opacity=1] (1456.62bp,175.86bp) ellipse (3.6bp and 3.6bp);
\end{scope}
\begin{scope}
  \definecolor{strokecol}{rgb}{0.0,0.0,0.0};
  \pgfsetstrokecolor{strokecol}
  \definecolor{fillcol}{rgb}{0.0,0.0,0.0};
  \pgfsetfillcolor{fillcol}
  \filldraw [opacity=1] (1654.45bp,429.97bp) ellipse (3.6bp and 3.6bp);
\end{scope}
\begin{scope}
  \definecolor{strokecol}{rgb}{0.0,0.0,0.0};
  \pgfsetstrokecolor{strokecol}
  \definecolor{fillcol}{rgb}{0.0,0.0,0.0};
  \pgfsetfillcolor{fillcol}
  \filldraw [opacity=1] (814.77bp,908.9bp) ellipse (3.6bp and 3.6bp);
\end{scope}
\begin{scope}
  \definecolor{strokecol}{rgb}{0.0,0.0,0.0};
  \pgfsetstrokecolor{strokecol}
  \definecolor{fillcol}{rgb}{0.0,0.0,0.0};
  \pgfsetfillcolor{fillcol}
  \filldraw [opacity=1] (814.67bp,993.28bp) ellipse (3.6bp and 3.6bp);
\end{scope}
\begin{scope}
  \definecolor{strokecol}{rgb}{0.0,0.0,0.0};
  \pgfsetstrokecolor{strokecol}
  \definecolor{fillcol}{rgb}{0.0,0.0,0.0};
  \pgfsetfillcolor{fillcol}
  \filldraw [opacity=1] (796.17bp,1075.45bp) ellipse (3.6bp and 3.6bp);
\end{scope}
\begin{scope}
  \definecolor{strokecol}{rgb}{0.0,0.0,0.0};
  \pgfsetstrokecolor{strokecol}
  \definecolor{fillcol}{rgb}{0.0,0.0,0.0};
  \pgfsetfillcolor{fillcol}
  \filldraw [opacity=1] (1652.04bp,256.21bp) ellipse (3.6bp and 3.6bp);
\end{scope}
\begin{scope}
  \definecolor{strokecol}{rgb}{0.0,0.0,0.0};
  \pgfsetstrokecolor{strokecol}
  \definecolor{fillcol}{rgb}{0.0,0.0,0.0};
  \pgfsetfillcolor{fillcol}
  \filldraw [opacity=1] (1655.5bp,343.41bp) ellipse (3.6bp and 3.6bp);
\end{scope}
\begin{scope}
  \definecolor{strokecol}{rgb}{0.0,0.0,0.0};
  \pgfsetstrokecolor{strokecol}
  \definecolor{fillcol}{rgb}{0.0,0.0,0.0};
  \pgfsetfillcolor{fillcol}
  \filldraw [opacity=1] (1818.0bp,88.59bp) ellipse (3.6bp and 3.6bp);
\end{scope}
\begin{scope}
  \definecolor{strokecol}{rgb}{0.0,0.0,0.0};
  \pgfsetstrokecolor{strokecol}
  \definecolor{fillcol}{rgb}{0.0,0.0,0.0};
  \pgfsetfillcolor{fillcol}
  \filldraw [opacity=1] (1747.26bp,124.16bp) ellipse (3.6bp and 3.6bp);
\end{scope}
\end{tikzpicture}